\documentclass[com,crcready]{iosart2x}  %the pdf wasn't clear if I should include option crcready
\usepackage{xifthen,xparse,ifmtarg,suffix,xstring,iftex}
\usepackage[disallowspaces,fixamsmath]{mathtools}
\usepackage{amsmath,amsthm,amssymb}
\usepackage{stackrel,thmtools,thm-restate}
\usepackage{parskip}
\usepackage[dvipsnames,svgnames,x11names]{xcolor}
\usepackage{enumitem}
\usepackage{framed}
\usepackage{pgffor,pgfmath,tikz}
\usetikzlibrary{calc,shapes,matrix,positioning,tikzmark,patterns,fit,matrix.skeleton}
\usepackage{suffix}

\RequirePackage{expl3}
\RequirePackage{ltxcmds}
\RequirePackage{iftex,ifpdf}
\RequirePackage{suffix}
\RequirePackage{ifmtarg}
\RequirePackage{xifthen}
\RequirePackage{xkeyval}
\RequirePackage{etoolbox}
\RequirePackage{pict2e,picture}
\RequirePackage{xparse}

\usepackage{hyperxmp}

\RequirePackage{hyperref}
\hypersetup{pdfauthor={Peter Cholak and Peter M.  Gerdes},
            pdfsubject={Computability Theory, Turing Reducibility and r.e. sets},
            pdfcontactemail={gerdes@invariant.org},
            pdftitle={Extending Properly n-REA Sets}
            pdfcontacturl={http://invariant.org},
            pdflang={en},
            pdfcopyright={Copyright (C) \today, Peter M. Gerdes and Peter Cholak},
            pdfkeywords={\@keywords},
            baseurl={http://invariant.org},
            unicode=true,
            pdfdisplaydoctitle=true,
            pdfusetitle,
            colorlinks=true,
            citecolor=blue,
            urlcolor=blue,
            linkcolor=magenta,
            pdfborder={0 0 0},
            breaklinks=true,
            pdfdisplaydoctitle=true,
            pagebackref=true,
            hyperfootnotes=true.
            hyperindex=true
}

\usepackage[nameinlink]{cleveref}

\makeatletter
\let\pmg@casefont\textrm
\makeatother

    \usepackage{mathrsfs}  %need distinct  mathcal and mathscr fonts
    \DeclareMathAlphabet{\mathbrush}{T1}{pbsi}{xl}{n}  
\declaretheorem[numberwithin=section,name=Theorem,refname={theorem,theorems},
Refname={Theorem,Theorems}]{theorem}
\declaretheorem[sibling=theorem,name=Lemma,refname={lemma,lemmas},
Refname={Lemma,Lemmas}]{lemma}

\declaretheorem[sibling=theorem,name=Proposition,refname={proposition,propositions},
Refname={Proposition,Propositions}]{proposition}

\declaretheorem[sibling=theorem,name=Definition,refname={definition,definitions},
Refname={Definition,Definitions},style=definition]{definition}
\declaretheorem[sibling=theorem,name=Notation,refname={notation,notations},
Refname={Notation,Notations},style=definition]{notation}

\declaretheorem[sibling=theorem,name=Hypothesis,refname={hypothesis,hypotheses},
Refname={Hypothesis,Hypotheses},style=definition]{hypo}

\declaretheorem[name=Property,refname={property,properties},
Refname={Property,Properties},style=definition]{property}
\declaretheorem[name=Procedure,refname={procedure,procedures},
Refname={Procedure,Procedures},style=definition]{procedure}
\declaretheorem[name=Condition,refname={condition,conditions},
Refname={Condition,Conditions},style=definition]{condition}

\newlist{propenum}{enumerate}{1}
\setlist[propenum]{label=\alph*), ref=\theproperty\alph*}
\crefalias{propenumi}{property} 

\declaretheorem[numberwithin=section,name=Rule,refname={rule,rules}, Refname={Rule,Rules},style=definition]{crule} %Sadly rule is already taken so must do
\newlist{ruleenum}{enumerate}{1}
\setlist[ruleenum]{label=\alph*), ref=\thecrule\alph*}
\crefalias{ruleenumi}{crule}

\declaretheorem[sibling=theorem,name=Remark,refname={remark,remarks},
Refname={Remark,Remarks},style=remark]{remark}

\makeatletter

    \crefformat{pfcasesnonumi}{#2case~#1#3}
    \Crefformat{pfcasesnonumi}{#2Case~#1#3}
    \crefformat{pfcasesnonumii}{#2subcase~#1#3}
    \Crefformat{pfcasesnonumii}{#2Subcase~#1#3}
    \crefformat{pfcasesnonumiii}{#2subsubcase~#1#3}
    \Crefformat{pfcasesnonumiii}{#2Subsubcase~#1#3}
    \crefformat{pfcasesnumberedi}{#2case~#1#3}
    \Crefformat{pfcasesnumberedi}{#2Case~#1#3}
    \crefformat{pfcasesnumberedii}{#2case~#1#3}
    \Crefformat{pfcasesnumberedii}{#2Case~#1#3}
    \crefformat{pfcasesnumberediii}{#2case~#1#3}
    \Crefformat{pfcasesnumberediii}{#2Case~#1#3}

    \newlist{pfcasesnonum}{enumerate}{3}
    \setlist[pfcasesnonum]{
      label={\pmg@casefont{Case}},
      align=left,
      left=0pt .. 1.5em,
      itemindent=* 
    }
    \setlist[pfcasesnonum,1]{
        before=\def\pfcasecounter@pmg{pfcasesnonumi},
    }
    \setlist[pfcasesnonum,2]{
        before=\def\pfcasecounter@pmg{pfcasesnonumii},
    }
    \setlist[pfcasesnonum,3]{
        before=\def\pfcasecounter@pmg{pfcasesnonumiii},
    }
    \newlist{pfcasesnumbered}{enumerate}{3}
    \setlist[pfcasesnumbered]{
      align=left,
      left=0pt .. 1.5em,
      itemindent=* 
    }
    \setlist[pfcasesnumbered,1]{
        before=\def\pfcasecounter@pmg{pfcasesnumberedi},
      label={{\pmg@casefont{Case}}~\arabic*},
      ref={\arabic*},
    }
    \setlist[pfcasesnumbered,2]{
        before=\def\pfcasecounter@pmg{pfcasesnumberedii},
      label={{\pmg@casefont{Case}}~\arabic{pfcasesnumberedi}\alph*},
      ref={\arabic{pfcasesnumberedi}\alph*},
    }
    \setlist[pfcasesnumbered,3]{
        before=\def\pfcasecounter@pmg{pfcasesnumberediii},
      label={{\pmg@casefont{Case}}~\arabic{pfcasesnumberedi}\alph{pfcasesnumberedii}.\roman*},
      ref={\arabic{pfcasesnumberedi}\alph{pfcasesnumberedii}.\roman*},
    }
    \def\@recthy@cases@labelindent{0pt}
    \newenvironment{pfcases*}{
        \def\@recthy@cases@labelindent{-1em}
        \ProvideDocumentCommand{\case}{r[]}{
                \def\thiscase{~##1}%
            \item~##1\textbf{:} %
            \ltx@ifpackageloaded{cleveref}{%
            \cref@constructprefix{pfcases}{\cref@result}%
            \protected@xdef\cref@currentlabel{[\pfcasecounter@pmg][][\cref@result]##1}}{}  \protected@edef\@currentlabel{##1}\ignorespaces}
    \begin{pfcasesnonum}\def\@recthy@cases@labelindent{1em}\ignorespaces}{\end{pfcasesnonum}\ignorespacesafterend}

  \newlist{pmg@steps}{enumerate}{4}
  \setlist[pmg@steps]{
      align=left,
      left=0pt .. 1.5em,
      itemindent=*
    }
  \setlist[pmg@steps,1]{label={Step \arabic*:}, ref={\arabic*}}
  \setlist[pmg@steps,2]{label={Step \arabic{pmg@stepsi}\alph*:}, ref={\arabic{pmg@stepsi}\alph*}}
  \setlist[pmg@steps,3]{label={Step \arabic{pmg@stepsi}\alph{pmg@stepsii}.\Roman*:}, ref={\arabic{pmg@stepsi}\alph{pmg@stepsii}.\Roman*}}
  \setlist[pmg@steps,4]{label={Step \arabic{pmg@stepsi}\alph{pmg@stepsii}.\Roman{pmg@stepsiii}\Alph*:}, ref={\arabic{pmg@stepsi}\alph{pmg@stepsii}.\Roman{pmg@stepsiii}\Alph*}}
  \crefname{pmg@stepsi}{step}{steps}
  \Crefname{pmg@stepsi}{Step}{Steps}
  \crefname{pmg@stepsii}{step}{steps}
  \Crefname{pmg@stepsii}{Step}{Steps}
  \crefname{pmg@stepsiii}{step}{steps}
  \Crefname{pmg@stepsiii}{Step}{Steps}
  \crefname{pmg@stepsiiii}{step}{steps}
  \Crefname{pmg@stepsiiii}{Step}{Steps}
  \NewDocumentEnvironment{steps}{o}{
      \begingroup
      \ProvideDocumentCommand{\step}{o}{\IfValueTF{##1}{\item\textbf{##1}:}{\item}}
      \IfValueTF{#1}{\begin{pmg@steps}[#1]}{\begin{pmg@steps}}%
  }{%
      \end{pmg@steps}%
      \endgroup\ignorespacesafterend%
      }

    \DeclarePairedDelimiter\@recthy@abs{\lvert}{\rvert}
    \let\@recthy@modulescr\mathcal
    \NewDocumentCommand{\PriorityTreeModule}{mO{}m}{\ensuremath{{\@recthy@modulescr{#1}}^{#2}_{#3} }}
    \let\module=\PriorityTreeModule
    \robustify\(
    \robustify\)
    \providecommand*{\iffdef}{\stackrel{\text{\tiny def}}{\iff}}
    \newcommand*{\card}[1]{\lvert#1\rvert}
    \DeclareMathOperator{\dom}{dom}
    \ExplSyntaxOn %thx stackexghance
      \NewDocumentCommand \setcol {m m}
      {
        \bool_lazy_and:nnTF
          { \int_compare_p:nNn { \tl_count:n {#1} } = 3 }
          { \tl_if_head_eq_meaning_p:nN {#1} \setcol }
          {
            \pgerdes_setcol_special:nnnn #1 {#2}
          }
          {
            \pgerdes_setcol_normal:nn {#1} {#2}
          }
      }
      \cs_new:Npn \pgerdes_setcol_special:nnnn #1 #2 #3 #4
      {
        {#2} \sp { [#3] [#4] }
      }
      \cs_new:Npn \pgerdes_setcol_normal:nn #1 #2
      {
        {#1} \sp { [#2] }
      }
    \ExplSyntaxOff
    
    \providecommand*{\lh}[2][]{\@recthy@abs{#2}_{#1}}
    \providecommand*{\nin}{\notin}
    \providecommand*{\@pmg@saveeqstate}{\let\@origif@fleqn@pmg=\if@fleqn \let\@origiftagsleft@pmg=\iftagsleft@}
    \providecommand*{\@pmg@restoreeqstate}{ \let\if@fleqn=\@origif@fleqn@pmg\let\iftagsleft@=\@origiftagsleft@pmg}
    \providecommand*{\req@orig}[3]{\ensuremath{\@ifmtarg{#1}{\mathscr{#2}_{#3}}{\mathscr{#2}_{#3}^{#1}}}} 
    \NewDocumentCommand{\req@nohref}{sO{}mO{}m}{\req@orig{#2#4}{#3}{#5}}
    \NewDocumentCommand{\req@href}{O{}mO{}m}{\hyperref[\detokenize{req:#2@#1#3}]{\req@orig{#1#3}{#2}{#4}}}

   \NewDocumentCommand\req@href@star{s}{\IfBooleanTF#1{\req@nohref}{\req@href}}
   \newcommand{\req}{\req@nohref}
      \AtBeginDocument{\@ifundefined{hyperref}{}{\let\req=\req@href@star}}

    \newtagform{colon}{}{:}
    \newenvironment{requirement}[1]{
        \@pmg@saveeqstate
        \@fleqntrue
        \setlength\@mathmargin{1.5cm}
        \tagsleft@true
        \usetagform{colon}
        \begin{equation}\tag{#1}
    }{
        \end{equation}
        \@pmg@restoreeqstate
        \usetagform{default}
    }

    \NewDocumentEnvironment{require}{O{}mO{}m}{
        \begin{requirement}{\req@nohref[#1#3]{#2}{#4}} \label{\detokenize{req:#2@#1#3}}%
    }{
        \end{requirement}%
    }
    \NewDocumentEnvironment{require*}{O{}mO{}m}{
        \begin{requirement}{\req@nohref[#1#3]{#2}{#4}}
    }{
        \end{requirement}%
    }

    \NewDocumentCommand{\refreq}{O{}mO{}}{\ref{\detokenize{req:#2@#1#3}}}

    \NewDocumentCommand{\require@nested}{sO{}mO{}m}{%
        \tag{\req@nohref[#2#4]{#3}{#5}}\IfBooleanTF{#1}{}{\label{\detokenize{req:#3@#2#4}}}%
    }
    \NewDocumentEnvironment{requirements}{}{%
        \@pmg@saveeqstate%
        \@fleqntrue%
        \setlength\@mathmargin{1.5cm}%
        \tagsleft@true%
        \usetagform{colon}%
        \let\require=\require@nested%
        \gather
        }{\endgather%
        \@pmg@restoreeqstate%
        \usetagform{default}%
        \ignorespacesafterend}

      \let\save@mathaccent\mathaccent
  \newcommand*\if@single[3]{%
    \setbox0\hbox{${\mathaccent"0362{#1}}^H$}%
    \setbox2\hbox{${\mathaccent"0362{\kern0pt#1}}^H$}%
    \ifdim\ht0=\ht2 #3\else #2\fi
    }
  \newcommand*\rel@kern[1]{\kern#1\dimexpr\macc@kerna}
  \providecommand*\overbar{\relax}
  \renewcommand*\overbar[1]{\@ifnextchar^{{\over@bar{#1}{0}}}{\over@bar{#1}{1}}}
  \newcommand*\over@bar[2]{\if@single{#1}{\over@bar@{#1}{#2}{1}}{\over@bar@{#1}{#2}{2}}}
  \newcommand*\over@bar@[3]{%
    \begingroup
    \def\mathaccent##1##2{%
      \let\mathaccent\save@mathaccent
      \if#32 \let\macc@nucleus\first@char \fi
      \setbox\z@\hbox{$\macc@style{\macc@nucleus}_{}$}%
      \setbox\tw@\hbox{$\macc@style{\macc@nucleus}{}_{}$}%
      \dimen@\wd\tw@
      \advance\dimen@-\wd\z@
      \divide\dimen@ 3
      \@tempdima\wd\tw@
      \advance\@tempdima-\scriptspace
      \divide\@tempdima 10
      \advance\dimen@-\@tempdima
      \ifdim\dimen@>\z@ \dimen@0pt\fi
      \rel@kern{0.6}\kern-\dimen@
      \if#31
        \overline{\rel@kern{-0.6}\kern\dimen@\macc@nucleus\rel@kern{0.4}\kern\dimen@}%
        \advance\dimen@0.4\dimexpr\macc@kerna
        \let\final@kern#2%
        \ifdim\dimen@<\z@ \let\final@kern1\fi
        \if\final@kern1 \kern-\dimen@\fi
      \else
        \overline{\rel@kern{-0.6}\kern\dimen@#1}%
      \fi
    }%
    \macc@depth\@ne
    \let\math@bgroup\@empty \let\math@egroup\macc@set@skewchar
    \mathsurround\z@ \frozen@everymath{\mathgroup\macc@group\relax}%
    \macc@set@skewchar\relax
    \let\mathaccentV\macc@nested@a
    \if#31
      \macc@nested@a\relax111{#1}%
    \else
      \def\gobble@till@marker##1\endmarker{}%
      \futurelet\first@char\gobble@till@marker#1\endmarker
      \ifcat\noexpand\first@char A\else
        \def\first@char{}%
      \fi
      \macc@nested@a\relax111{\first@char}%
    \fi
    \endgroup
  }
      \let\exists@orig@recthy=\exists
\let\forall@orig@recthy=\forall
\providecommand*{\existsinf}{\exists^{\infty}}
        \WithSuffix\def\existsinf(#1){\left(\existsinf #1\right)\!}
        \WithSuffix\def\existsinf[#1]{\left[\existsinf #1\right]\!}
        \WithSuffix\def\exists(#1){\left(\exists #1 \right)\!}
        \WithSuffix\def\forall(#1){\left(\forall #1 \right)\!}
    \providecommand*{\union}{\mathbin{\cup}}
    \ProvideDocumentCommand{\conv}{O{}}{\mathpunct{\downarrow}_{#1}}
    
    \providecommand*{\sigman}[1]{\Sigma_{#1}}
    \WithSuffix\def\sigman[#1]#2{\Sigma^{#1}_{#2}}
    \providecommand*{\sigmaZeroN}[1]{\Sigma^{0}_{#1}}
    \WithSuffix\def\sigmaZeroN[#1]#2{\Sigma^{0,#1}_{#2}}
    \providecommand*{\sigmaZeroOne}{\sigmaZeroN{1}}
    \WithSuffix\def\sigmaZeroOne[#1]{\sigmaZeroN[#1]{1}}
    \providecommand*{\sigmazi}{\sigmaZeroOne}
    
    \providecommand*{\eset}{\emptyset}
    \ProvideDocumentCommand{\REAop}{st+d()od()mm}{%
      {\IfBooleanTF{#1}%
        {\widehat{\mathcal{J}}}%
        {\mathcal{J}}%
      }^{%
          \IfBooleanTF{#2}{\vphantom{x}^\dagger}{}%
          #7}_{#6\IfValueTF{#4}{, #4}{}}%
          \IfValueTF{#3}{\left(#3\right)}{\IfValueTF{#5}{\left(#5\right)}{}}%
    }
     
   \def\kleeneOSYM{\mathcal{O}}
   \ProvideDocumentCommand{\kleeneO}{sD(){}oD(){}}{{
       \kleeneOSYM^{#2#4}_{%
            \IfNoValueTF{#1}{%
                \IfNoValueTF{#3}{}{%
                    \abs{#3}}% 
                }%
                {%
                    1%
                    \IfNoValueTF{#3}{}{%
                        , \abs{#3}
                    }%
                }%
        }%
    }}
    \ProvideDocumentCommand{\subfun}{o}{\prec\IfNoValueTF{#1}{}{_{#1}}}
    \ProvideDocumentCommand{\supfun}{o}{\succ\IfNoValueTF{#1}{}{_{#1}}}
    \ProvideDocumentCommand{\nsubfun}{o}{\nprec\IfNoValueTF{#1}{}{_{#1}}}
    \ProvideDocumentCommand{\nsupfun}{o}{\nsucc\IfNoValueTF{#1}{}{_{#1}}}
    \ProvideDocumentCommand{\subfuneq}{o}{\preceq\IfNoValueTF{#1}{}{_{#1}}}
    \ProvideDocumentCommand{\subfunneq}{o}{\@precneq\IfNoValueTF{#1}{}{_{#1}}}
    \ProvideDocumentCommand{\supfuneq}{o}{\succeq\IfNoValueTF{#1}{}{_{#1}}}
    \ProvideDocumentCommand{\supfunneq}{o}{\@succneq\IfNoValueTF{#1}{}{_{#1}}}
    \ProvideDocumentCommand{\nsubfuneq}{o}{\npreceq\IfNoValueTF{#1}{}{_{#1}}}
    \ProvideDocumentCommand{\nsupfuneq}{o}{\nsucceq\IfNoValueTF{#1}{}{_{#1}}}
    \providecommand*{\diverge}{\mathpunct{\uparrow}}
    \ProvideDocumentCommand{\set}{m!G{}}{{\let\st=\mid\left\{#1 \ifthenelse{\isempty{#2}}{}{\mid #2} \right\}}}
    \providecommand*{\eqdef}{\overset{\text{\tiny def}}{=}} %\newcommand*{\eqdef}{\ensuremath{=\limits_{\text{\tiny def}}}}

\providecommand*{\@recthy@llangle}{\langle\!\langle}
\providecommand*{\@recthy@rrangle}{\rangle\!\rangle}
\def\@recthy@langle{\langle}
\def\@recthy@rangle{\rangle}
\let\@recthy@lstrdelim=\@recthy@llangle
\let\@recthy@rstrdelim=\@recthy@rrangle
\let\@recthy@lcodedelim=\@recthy@llangle
\let\@recthy@rcodedelim=\@recthy@rrangle
\let\@recthy@lpairdelim=\@recthy@langle
\let\@recthy@rpairdelim=\@recthy@rangle
  \providecommand*{\pair}[2]{\mathopen{\@recthy@lpairdelim} #1, #2 \mathclose{\@recthy@rpairdelim}}
  \providecommand*{\str}[1]{\mathopen{\@recthy@lstrdelim}#1\mathclose{\@recthy@rstrdelim}}
  \providecommand*{\code}[1]{\mathopen{\@recthy@lcodedelim}#1\mathclose{\@recthy@rcodedelim}}

    \ProvideDocumentCommand{\incompat}{o}{\mathrel{\mid}\IfValueTF{#1}{_{#1}}{}}
    \ProvideDocumentCommand{\compat}{o}{\mathrel{\rlap{\ensuremath{\not}}\ensuremath{\mid}}\IfValueTF{#1}{_{#1}}{}}
\ifpdftex
    \def\symbf#1{\mathbf{#1}}
    \mathchardef\@recthy@mhyphen="2D % Define a "math hyphen"
\else
  \def\@recthy@mhyphen{\mathhyphen}
\fi

\ProvideDocumentCommand{\REset}{D(){}oD(){#1}mO{#2}}{{W_{#4\IfValueTF{#5}{, #5}{}}^{#3}}}
\ProvideDocumentCommand{\REA}{d()o}{\ensuremath{\IfValueTF{#2}{#2 \@recthy@mhyphen}{}\text{REA}\IfValueTF{#1}{(#1)}{}}}
\providecommand*{\@recthy@ensuretext}[1]{\ensuremath{\text{#1}}}
\providecommand*{\re}{\@recthy@ensuretext{r.e.\ }}
\def\@recthy@recfnlSYM{\Phi}
\ProvideDocumentCommand{\recfnl@imp}{oD(){}m!G{#2}!g!d()}{{
      \def\temp@pmg@arg{}
      \IfValueTF{#5}{\ifthenelse{\isempty{#5}}{}{\def\temp@pmg@arg{; #5}}}{\IfValueTF{#6}{\ifthenelse{\isempty{#6}}{}{\def\temp@pmg@arg{; #6}}}}
        \@recthy@recfnlSYM_{#3\IfValueT{#1}{, #1}}\ifthenelse{\isempty{#4}}{}{\!\left(#4\temp@pmg@arg \right)}
    }}
\let\recfnl=\recfnl@imp

\providecommand*{\Tsetjoin}{\mathbin{\oplus}}
\providecommand*{\TsetJoin}{\mathop{\bigoplus}}
\providecommand*{\Tplus}{\Tsetjoin}
\providecommand*{\TPlus}{\TsetJoin}
\providecommand*{\restr}[1]{\mathpunct{\restriction_{#1}}}
    \providecommand*{\@recthy@TSYM}{\symbf{T}}
    \providecommand*{\Tequiv}{\mathrel{\equiv_{\@recthy@TSYM}}}
    
    \providecommand*{\nTequiv}{\mathrel{\ncong_{\@recthy@TSYM}}}
    
    \providecommand*{\Tlneq}{\lneq_{\@recthy@TSYM}}
    \providecommand*{\Tleq}{\leq_{\@recthy@TSYM}}
    \providecommand*{\Tgneq}{\gneq_{\@recthy@TSYM}}
    \providecommand*{\Tgeq}{\geq_{\@recthy@TSYM}}
    \providecommand*{\Tgtr}{>_{\@recthy@TSYM}}
    \providecommand*{\Tless}{<_{\@recthy@TSYM}}
    \providecommand*{\nTleq}{\nleq_{\@recthy@TSYM}}
    \providecommand*{\nTgeq}{\ngeq_{\@recthy@TSYM}}

    \providecommand*{\Tzerosym}{\symbf{0}}
    \providecommand*{\Tzero}{{\Tzerosym}}

    \def\@recthy@useSYM@default{\ltx@ifpackageloaded{unicode-math}{\symbffrak{u}}{\mathfrak{u}}}
\let\@recthy@useSYM=\@recthy@useSYM@default
\providecommand*{\use}[1]{\mathop{\@recthy@useSYM}\left[#1\right]}
\providecommand*{\entersat}[1]{\mathbin{\searrow_{#1}}}
\providecommand*{\@recthy@setovercmp}[1]{\overline{#1}}
\providecommand*{\@recthy@setsimcmp}[1]{\backsim #1}
\let\setcmp=\@recthy@setovercmp

\NewDocumentCommand{\axset}{D<>{}O{}D<>{#1}d()}{\mathcal{A}^{#3}_{#2}\IfValueTF{#4}{\left(#4\right)}{}} 
     \newcommand{\laxiom}[2]{ \mathopen{\langle}#1 \rightarrow #2 \mathclose{\rangle} }
\NewDocumentCommand{\Cuse}{O{}mm}{\mathop{\overbar{\@recthy@useSYM}}^{#1}\left(#2,#3\right)}
\newcommand{\bpfuncs}{{\lbrace 0,1, \diverge \rbrace}^{< \omega}}
\DeclareMathOperator{\supp}{supp}
\let\defword=\textbf
\NewDocumentCommand{\codeY}{O{}D<>{}O{}m}{\left[{#4}\right]^{#1#3}_{#2}} 
\WithSuffix\def\Tplus+{\overbar{\Tplus}}
\WithSuffix\def\TPlus+{\overbar{\TPlus}}
\newcommand*{\colrestr}[1]{\mathpunct{\restriction_{\left[#1\right]}}} %columwise restriction
\WithSuffix\def\restr*{\colrestr}
\NewDocumentCommand{\Gfunc}{ommd()}{\Gamma_{\IfValueTF{#1}{#2, #1}{#2}}\ifthenelse{\isempty{#3}}{}{\left(#3\IfValueTF{#4}{; #4}{}\right)}} %Written as \Gfunc[s]{i}{Z}(x)
\NewDocumentCommand{\Tfunc}{O{}md()}{\Theta_{#1}\ifthenelse{\isempty{#2}}{}{\left(#2\IfValueTF{#3}{; #3}{}\right)}}
\NewDocumentCommand{\SelfComp}{D<>{}O{}mm}{\Psi^{#1}_{#2}(#3, #4)}
\NewDocumentCommand{\bY}{D<>{}O{}m}{\ensuremath{\mathbrush{b}^{#1}_{#3\IfValueTF{#2}{, #2}{}}}}

\NewDocumentCommand{\Yagree}{D<>{}}{\mathbin{\approx^{#1}}}

        \def\@IfColorTF#1\relax#2#3{\@ifundefinedcolor{#1}{#2}{#3}}

    \def\@@IfColorTF#1#2\relax#3#4{\uppercase{\@IfColorTF#1}#2\relax{#3}{#4}}
     
    \newcommand{\IfColorTF}[3]{\@@IfColorTF#1\relax{#2}{#3}}%\@ifundefined{\string\color@\MakeUppercase#1}{#2}{#3}}

    \NewDocumentCommand{\StrBehindLast}{smmo}{
        \IfBooleanTF {#1}{\IfNoValueTF{#4}{\StrBehind*[\StrCount*{#2}{#3}]{#2}{#3}[#4]}{\StrBehind*[\StrCount*{#2}{#3}]{#2}{#3}}}
        {\IfNoValueTF{#4}{\StrBehind[\StrCount{#2}{#3}]{#2}{#3}[#4]}{\StrBehind[\StrCount{#2}{#3}]{#2}{#3}}}
    }
           
    \newcommand{\matrixcontent@pmg}{}
    \newcommand{\REAmatrix@lastname}{}

    \newcommand{\row@pmg}[1]{%
      \foreach \i [count=\ci] in {1,...,#1} {%
        \begingroup\edef\x{\endgroup
           \noexpand\gappto\noexpand\matrixcontent@pmg{\noexpand\phantom{1} 
            \ifnum\ci=#1%
            \else%  
                \&%
            \fi%
           }}\x%{\noexpand\node{\i}; \&}}\x
        }%
      \gappto\matrixcontent@pmg{\\}%
    }
    \newcommand{\REAmatrixContents}[2]{
     \foreach \j in {1,...,#1} {%
        \row@pmg{#2}
    }}
    \NewDocumentCommand{\REAmatrix}% \REAmatrix[style override]{height}{width}
        {D(){m}d<>O{nodes in empty cells=true}mm}{
            \let\matrixcontent@pmg\empty
            \REAmatrixContents{#4}{#5}
            \expandafter\def\csname #1-height\endcsname{#4}
            \expandafter\def\csname #1-width\endcsname{#5}
            \def\REAmatrix@lastname{#1}      
            \IfValueTF{#2}{\matrix (#1) [matrix of math nodes,nodes={draw, thin, outer sep=0pt},row sep=-\pgflinewidth,column sep=-\pgflinewidth,#3,label={above:#2},nodes in empty cells=true,ampersand replacement=\&] {
            \matrixcontent@pmg
            };}{\matrix (#1) [matrix of math nodes,nodes={draw, thin, outer sep=0pt},row sep=-\pgflinewidth,column sep=-\pgflinewidth,#3,nodes in empty cells=true,ampersand replacement=\&] {
            \matrixcontent@pmg
            };}
    }

    \NewDocumentCommand{\REAmatrixFillRect}{D<>{pattern=crosshatch}d()mm}{
        \IfValueTF{#2}{
            \edef\temp@pmg{%
                    \noexpand\draw[#1] (#2-#3.north west) rectangle (#2-#4.south east);}%
        }{%
            \edef\temp@pmg{%
                    \noexpand\draw[#1] (\REAmatrix@lastname-#3.north west) rectangle (\REAmatrix@lastname-#4.south east);}%
        }%
         \begin{scope}[on background layer]
            \temp@pmg
        \end{scope}
    }

    \NewDocumentCommand{\GetREAmatrixHeight}{D(){\REAmatrix@lastname}}{
        \edef\temp@pmg@name{#1}
        \expandafter\csname \temp@pmg@name-height\endcsname
    }
 
    \NewDocumentCommand{\pmg@draw@node}{D<>{black}mmm} {\node[#1] (#4) at (#2) {#3};} 
    \NewDocumentCommand{\pmg@circle@node}{D<>{black}mm} {\node[circle,draw, minimum size=2.2mm,#1] (#3) at (#2) {};} 
    \NewDocumentCommand{\REAmatrixEnum}{D(){val-#3}O{1}m !d()}{
        \foreach \node@pmg in {#3} {
            \IfValueTF{#4}{\pmg@draw@node{#4-\node@pmg}{#2}{#1}}{\expandafter\pmg@draw@node{\REAmatrix@lastname-\node@pmg}{#2}{#1}}
        }
    }

    \NewDocumentCommand{\REAmatrixCircle}{D(){circ-#3}m !D(){\REAmatrix@lastname}}{
        \foreach \node@pmg in {#2} {
            \pmg@circle@node{#3-\node@pmg}{#1}%}{\expandafter\pmg@circle@node{\REAmatrix@lastname-\node@pmg}{#1}}
        }
    }

    \NewDocumentCommand{\LineLabelNode}{D(){label-#5}D<>{black}O{2ex}O{right}mm}{
            \node[#4 = #3 of #5] (#1) {#6};
            \draw[color=#2,align=center] (#5) -- (#1);
    }

    \NewDocumentCommand{\LineLabelCell}{D(){label-#7-#5}D<>{black}O{right}O{2ex}mm !D(){\REAmatrix@lastname}}{
           \edef\temp@pmg{\noexpand\node[circle,draw,minimum size=2.2mm,transparent] (circ-#7-#5) at (#7-#5) {};
            \noexpand\node[#3 = #4 of circ-#7-#5] (#1) {#6};
            \noexpand\draw[color=#2] (circ-#7-#5) -- (#1);}
            \temp@pmg
    }
    \NewDocumentCommand{\CircleLabelNode}{D(){label-#7-#5}D<>{black}O{right}O{2ex}mm !D(){\REAmatrix@lastname}}{
           \edef\temp@pmg{\noexpand\node[circle,draw,minimum size=2.2mm,#2] (circ-#7-#5) at (#7-#5) {};
            \noexpand\node[#3 = #4 of circ-#7-#5] (#1) {#6};
            \noexpand\draw[color=#2] (circ-#7-#5) -- (#1);}
            \temp@pmg
    }

\makeatother

\begin{document}
\begin{frontmatter} % The preamble begins here.

\title{Extending Properly \( \REA[n] \) Sets\thanks{Thanks to two anonymous referees for many helpful comments and editing suggestions.}}
\runtitle{Extending Properly \( \REA[n] \) Sets}
\begin{aug}
\author[A]{\inits{P.}\fnms{Peter} \snm{Cholak}\ead[label=e1]{cholak@nd.edu}
  \thanks{Cholak was partially supported by a Focused Research Group grant from the National Science Foundation of the United States, DMS-1854136.}}
\author[B]{\inits{P.}\fnms{Peter} \snm{Gerdes}\ead[label=e2]{gerdes@invariant.org}}
    \address[A]{Mathematics Department, \orgname{University of Notre Dame du Lac}, 255 Hurley Building, Notre Dame, IN 46556, \cny{USA}\printead[presep={\\}]{e1}}
\address[B]{Mathematics Department, \orgname{Indiana University, Bloomington}, Rawles Hall, 831 East 3rd St., Bloomington, IN 47405 \cny{USA}\printead[presep={\\}]{e2}}
% \author[A]{\inits{P.}\fnms{Peter} \snm{Gerdes}\ead[label=e1]{gerdes@invariant.org}}%
% \author[B]{\inits{P.}\fnms{Peter} \snm{Cholak}\ead[label=e2]{cholak@nd.edu}
%     \thanks{Cholak was partially supported by a Focused Research Group grant from the National Science Foundation of the United States, DMS-1854136.}}
% \address[A]{Mathematics Department, \orgname{Indiana University, Bloomington}, Rawles Hall, 831 East 3rd St., Bloomington, IN 47405 \cny{USA}\printead[presep={\\}]{e1}}
% \address[B]{Mathematics Department, \orgname{University of Notre Dame du Lac}, 255 Hurley Building, Notre Dame, IN 46556, \cny{USA}\printead[presep={\\}]{e2}}
\end{aug}


%Is there a place to put the MSC subject classes?
%\subjclass[2010]{Primary 03D25, 03D30; Secondary (03D60)}

\end{frontmatter}

\def\hat#1{{\widehat{#1}}}  %use widehat because it looks nicer

\section{Introduction}
In \cite{Jockusch1984PseudoJump} Jockusch and Shore introduce the \( \REA[n] \) sets.  A \( \REA[1] \) set is just an \re set, while an \( \REA[n+1] \) set is a set  of the form \( A \Tplus \REset(A){e} \) where \( A \) is \( \REA[n] \).  Despite the fact that the \( \REA[n] \) sets are merely the result of iterating the operation of adding an \re set, many basic questions remain even when \( n \) is finite.

One natural property to investigate about \( \REA[n] \) sets is when they are properly \( \REA[n] \).

\begin{definition}\label{def:properly-n-REA}
    A set \( B \) is properly \( \REA[n] \) if \( B \) is \( \REA[n] \) and \( B \) isn't Turing equivalent to any \( \REA[m] \) set with \( m < n \).
\end{definition}

It is evident that for all \( n \in \omega \) there are properly \( \REA[n] \) sets. But it's obviously not the case that for every \( \REA[n] \) set \( A \) there is a properly \( \REA[n+1] \) set of the form \( A \Tplus \REset(A){e} \), i.e., not every \( \REA[n] \) set can be extended to a properly \( \REA[n+1] \) set.  For instance, the empty set is a \( \REA[1] \) set but can't be extended to a properly \( \REA[2] \) set.

A natural hypothesis is that any properly \( \REA[n] \) set can be extended to a properly \( \REA[n+1] \) set.  This represents the most optimistic possible hypothesis about when \( \REA[n] \) sets can be extended to properly \( \REA[n+1] \) sets.  

\begin{hypo}\label{hypo:extendable}
    If \( A \) is a properly \( \REA[n] \) set then there is a properly \( \REA[n+1] \) set (Turing equivalent to a set\footnote{We add this caveat because, in the argument below, we adopt a slightly different form for \REA[n] sets that makes the argument more convenient.}) of the form \( A \Tplus \REset(A){e} \).  In such cases, we say \( A \) can be extended to a properly \( \REA[n+1] \) set.  
\end{hypo}

Further evidence for this hypothesis comes in the form of a result by Soare and Stob \cite{Soare1982Relative}  who demonstrate that for any \re set \( \REset{e} \Tgtr \Tzero \) there is a set \( \REA[1] \) in  \( \REset{e} \) but not of \re degree.  This establishes that every properly \( \REA[1] \) set  can be extended to a properly \( \REA[2] \) set.  %Using the following definitions we can recast this result in terms of properly \( \REA[n] \) degrees.  
This result was extended to \( n=2 \) by Cholak and Hinman when they established the following result (recast using the above definitions) in \cite{Cholak1994Iterated}.

\begin{theorem}\label{thm:extend-two-REA}
If \( X \) is properly \( \REA[2] \) then there is a \( \REA(X) \) set \( \REset(X){e} \Tplus X \) which is properly \( \REA[3] \).  
\end{theorem}

In other words any \( \REA[2] \) set which isn't of \re degree can be extended to a \( \REA[3] \) set not of \( \REA[2] \) degree.   However, despite the suggestive evidence we show that the above attractive hypothesis in fact fails by proving the following theorem.

\begin{restatable}{theorem}{mainthm}\label{thm:nonextendable-three-REA}
    There is a properly \( \REA[3] \) set \( A \) which can't be extended to a properly \( \REA[4] \) set.  The set \( A \) can also be taken to be \( \Delta^0_2 \).  
\end{restatable}    

Note that, if \( A \) isn't properly \( \REA[3] \) then it is evident that no set of the form \( A \Tplus \REset(A){e} \) will be properly \( \REA[4] \).  Our strategy, roughly speaking, will be to build a \( \REA[3] \) set \( A \) such that enumeration into the first and second components of \( A \) (where the first component of \( A \) is the \re part of \( A \) and the second component is the part of $A$ \re in the first component) doesn't result in further changes to \( \REset(A){i} \).  The construction takes the form of a finite injury argument, with the complexity arising from the difficulty of showing that the requirements eventually succeed.  However, before we present the main construction, we first review some notation and define the \( \REA[n] \)  sets in \cref{sec:background}.  Note that in this section we also introduce the idea that \( \REA[n] \) sets can be viewed as \( n+1 \) column sets (with the \( 0 \)-th column empty) produced by an \re set of axioms. 

Readers familiar with \( \REA[n] \) sets may wish to jump ahead to \cref{sec:overview}, where we describe the requirements that our construction will meet and delineate some basic conditions that our construction will satisfy.  

\section{Background}\label{sec:background}
\subsection{Notational Conventions}\label{sec:background:conv}

We largely adopt the  standard notation seen in \cite{Odifreddi1992Classical} which we briefly review.  The use of \( \recfnl[s]{i}{X}{y}   \) is denoted by \( \use{\recfnl[s]{i}{X}{y}} \), the \( e \)-th set \re in \( X \) by  \( \REset(X){e} \) and we write \( y \entersat{s} X \) (\( y \entersat{s} \setcmp{X} \)) to indicate \( y \) enters (leaves) \( X \) at stage \( s \).      We let \( \code{x_0, \ldots x_n} \) (\( \pair{x}{y} \eqdef \frac{1}{2}(x+y)(x+y+1)+y  \) ) denote a canonical bijection of \( \omega^{< \omega} \) (\( \omega^2 \)) with \( \omega \) and define \( A \Tplus B \),  \( \TPlus_{n \in S} X_n  \) and \( \setcol{X}{n} \)/\( \setcol{X}{< n}  \)  standardly\footnote{That is, \( A \Tplus B \eqdef \set{y \st y = 2x \land x \in A \lor y = 2x+1 \land x \in B} \) and \( \TPlus_{n \in S} X_n \eqdef \set{\pair{n}{x}}{n \in S \land x \in X_n}   \)}.   We write \( \sigma \incompat \tau \) (\( \sigma \compat \tau \)) to indicate that two strings/partial functions/etc. are incompatible (compatible).  %We \textbf{extend} this common notation by writing  \( \chi \incompat[k] \xi \) (\( \chi \compat[k] \xi \)) to indicate  disagreement (agreement) on the first \( k \) columns\footnote{That is, \( \displaystyle \chi \incompat[k] \xi \iffdef \exists(k' \leq k)\exists(x)\left(\chi(\pair{k'}{x})\conv \neq \xi(\pair{k'}{x})\conv  \right) \)}.

We also adopt some less common notational conventions and assumptions.   We let the variables \( \chi, \xi, \eta \) range over \( \bpfuncs \) (binary partial functions with finite domain)  with extension denoted by \( \supfun \).  By identifying sets with their characteristic functions (so \( A_s \) is really a \( \bpfuncs \)) we gain the ability for our approximations to a set \( A \)  to take no position on whether \( x \in A \).  This turns out to be helpful in giving well-behaved stagewise approximations to \( \REA[n] \) sets.

%We also define \( A\colrestr{s}  \eqdef \set{\pair{k}{x}}{k \leq n \land x < s} \) where this will be used in place of \( A\restr{s} \) when we wish to treat each column similarly. %\footnote{Specifically, we can cut back the region of a set \( A \) which we are approximating at something like true stages so our approximation for an element \( x \) isn't always wrong.  After all, there might be a constant stream of axioms all enumerating an element \( x \)  before being canceled in turn.}.  

With this notation in place, we briefly review \( \REA[n] \) sets and describe how we identify such sets with \( n+1 \) (non-trivial) column sets (and whose remaining columns are empty).  The reader familiar with \( \REA[n] \) sets may wish to skip ahead to \cref{sec:overview} after familiarizing themselves with the notion of an \( A \) supported approximation in \cref{def:supported}.% (roughly speaking an approximation \( \chi \)  is \( A \) supported if \( \chi \) reflects all axioms in \( A \) \( \chi \) can guarantee are applicable).

\subsection{\( \REA[\alpha] \) Sets}\label{sec:background:rea}

In \cite{Jockusch1984PseudoJump} Jockusch and Shore introduce the \( \REA[\alpha] \) sets for any ordinal notation \( \alpha \in \kleeneO \).    In this paper, we will only be concerned with finite values of \( \alpha \) so we adopt the following definition (equivalent up to \( 1 \)-degree for finite \( \alpha \)).

\begin{definition}\label{def:rea-op}
Given a computable function \( f \),  we define the \( \REA[n] \) operator \( \REAop(X){f}{n}  \) (or \( \REAop(X){e}{n}  \) where \( e \) is an index for \( f \)),  \( n \in \omega \)  via the following inductive definition (where we let  \( \REset(Z){f(m)}  = \eset \) if \( f(m)\diverge \)).    %(where \( \beta, \lambda \) are assumed to be ordinal notations with \( \lambda \) a limit notation effectively given as the limit of \( \beta_n, n < \omega \))

\begin{equation}\label{eq:rea-def}
\begin{aligned}
\REAop*(X){f}{0} &= X\\
\REAop*(X){f}{m + 1} &= \REset({\REAop(X){f}{m}}){f(m)} \\
\REAop(X){f}{n} &= \TPlus_{m < n + 1} \REAop*(X){f}{m} 
\end{aligned}
\end{equation}
Furthermore, \( C \) is an \( \REA(X)[n] \) set just if \( C = \REAop(X){f}{n}  \) for some computable function \( f \) and \( C \) is an \( \REA[n] \) set if it is an \( \REA(\eset)[n] \) set.   We define \( X^{n}_e \) (written just as \( X_e \) when \( n \) is understood) to be \( \REAop(\eset){e}{n}  \).    
\end{definition}

Note that, because of how we define \( \TPlus_{m < n + 1} \) an \( \REA(X)[n] \) set \( Z \)  will have \( \setcol{Z}{m} = \eset \) for \( m > n \).  We further note that we can describe the construction of \( \REA(X)[n] \) with \( n \in \omega \) sets via the enumeration of `axioms' defined as follows.  

% As we aren't concerned with the particular details about exactly what kind of computable join operation \( \Tplus \) is applied we note that we may regard an \( \REA(X)[\alpha] \) set \( C \) as built from columns with \( \setcol{C}{0} = X \) and \( \setcol{C}{\beta \kleenePlus 1} = \REset({\REAop(X){f}{\beta}}){f(\beta)} \) for \( \beta \kleenePlus 1 \kleeneleq \alpha \). All other columns are left empty including those columns corresponding to limit notations \( \lambda \kleeneleq \alpha \) (since \( C \) already includes columns whose effective join is \( \REAop(X){f}{\lambda} \)).

% We abbreviate \( \setcol{C}{n}  \) by  \( C^{n}  \).  Moreover, when we are concerned with \( \REA(X)[\alpha] \) for \( \alpha \kleeneleq \omega \) we will identify the unique notation with height \( n \) with \( n \).  Thus, for \( m \in \omega \) a  \( \REA(X)[m] \) set \( C \) is a set with \( m+1 \) columns satisfying \( \pair{i}{x} \in  C  \) for \( i \leq m \) iff \( x \in C^{i} = \setcol{C}{i} \).

\begin{definition}\label{def:re-axiom}
    An \defword{axiom} is a pair \( \laxiom{\sigma}{y} \) with \( \sigma \in \bpfuncs, y \in \omega \).  An \defword{REA axiom} is an axiom that further satisfies \( y = \pair{m}{z} \) with \( m > 0 \) and  \( \dom \sigma  \subset \setcol{\omega}{< m} \).  Finally, an \defword{\( n \)-REA axiom } is an REA axiom with \( m \leq n \).
\end{definition}

We think of the axiom \( \laxiom{\sigma}{y} \) as an instruction to put \( y \) into a set \( Z \) provided \( \sigma \subfun Y \).  Thus, regarding \( y \) as coding an element of \( 2^{< \omega} \), an \re set of axioms defines (if compatible) a computable functional.     An REA axiom \( \laxiom{\sigma}{\pair{m}{z}}  \)  is then an instruction to put \( z \) into  \( \setcol{X}{m} \) provided \( \sigma \subfun \setcol{X}{< m} \) (as this is equivalent to \( \sigma \subfun X \)).

We now argue that for \(  n \in \omega \) we may identify \( \REA[n] \) operators \( \REAop{e}{n} \) (and hence \( \REA(X)[n] \) sets) with \re sets of \( n \)-REA axioms and effectively translate between \re indices for sets of \( n \)-REA axioms  and indices for \( \REA[n] \) operators.

\begin{lemma}\label{lem:build-rea}
    If \( \axset \) is a \re set of \( \REA[n] \) axioms (\(  n \in \omega \)) then the operator \( \REAop{A}{n} \) defined by
    \begin{equation}\label{eq:rea-axiom-def}
        \pair{l}{y} \in \REAop{A}{n}(X) \iff \left(l = 0 \land y \in X \right) \lor \exists[\sigma \subfun \REAop{A}{n}(X) ]\Bigl[ \laxiom{\sigma}{\pair{l}{y}} \in A \Bigr]
    \end{equation}
     is an \( \REA[n] \) operator.  Conversely, given an \( \REA[n] \) operator \( \REAop{f}{n} \)  there is a \( \re \) set \( A_f \) such that \( \REAop{A_f}{n} = \REAop{e}{n}  \).       Furthermore, we can effectively translate between \re indices of sets of \( \REA[n] \)  axioms and indices of \( \REA[n] \) operators. 
\end{lemma}
\begin{proof}
    We first prove that \( \REAop{A}{n} \) is an \( \REA[n] \) operator.  Note that, for \( 0 < l \leq n \), \( \setcol{\REAop{A}{n}(X)}{l} \) is determined by \( \setcol{\REAop{A}{n}(X)}{<l}  \) and \( X \) so \( J(X) \) is well defined.  Furthermore, the above equation explicitly defines \( \setcol{\REAop{A}{n}(X)}{l} \) from \( \setcol{\REAop{A}{n}(X)}{<l} \)  and \(  l  \) via a (uniformly) \( \sigmazi \) formula.  Thus, by an application of the s-m-n theorem \cite{Rogers1987Theory} there is a computable function \( f \) satisfying \cref{def:rea-op}.  

    Given \( \REAop{e}{n} \) we simply enumerate axioms \( \laxiom{\sigma}{\pair{m+1}{x}} \) into \( \axset \)  when \( \recfnl{f(m)}{\sigma}{x}\conv \) for \( m < n \).  The uniformity claim is evident from the proof.     
\end{proof}

In light of this result we adopt the following notation.  
\begin{notation}\label{def:axioms-for}
 \( \axset(J)  \) denotes a canonical enumeration of an \re set of axioms corresponding to the \( \REA[n] \) operator \( J \).   When \( X_e \) is an \( \REA[n] \) set we write \( \axset(X_e) \) for \( \axset(\REAop{e}{n}) \) (remember \( X_e \eqdef \REAop(\eset){e}{n} \) )   
 \end{notation} 

 In various cases we'll further specify which of the many potential enumerations of axioms we mean to specify with \( \axset(X_e) \).  For instance, we'll explicitly define the enumeration of axioms for the \REA[n] sets we explicitly construct and rely on \cref{lem:build-rea} to produce the corresponding \REA[n] set.   We adopt the following definition to assist us in defining the axioms we enumerate in these construction.

% This result justifies our talk of axioms being enumerated into \( \REA[n] \) sets/operators even when we haven't constructed those sets.   However, when we are constructing an \( \REA[n] \) set the following definition will be useful. 

% We denote the set of axioms for an \( \REA[n] \)  operator \( J \) by \( \axset(J) \) and we indicate the \( \REA[n] \) operator produced by applying the set of \( \leq n \)-axioms \( A \) to \( X \) by \( A(X) \).   %We will assume that if  \( \laxiom{\sigma}{y}  \) enters \( \axset{\REAop{e}{n}} \) at stage \( s \) then \( \dom \sigma \subseteq \set{\pair{m}{x}}{x < s}  \) (note that we need not obey this restriction in the construction of our own \( \REA[n] \) operators/sets as the stages there aren't the same as those in our uniform enumeration of all \re sets)
% During our constructions we won't specifically identify the axioms to be enumerated.  Rather, we will identify elements to be enumerated and then enumerate those elements dependent on the current state of earlier columns.  The following definition will help us formalize this practice.

\begin{definition}\label{def:dependent-on}
    The axiom \( \laxiom{\sigma'}{y} \) \defword{depends} on \( \sigma \) if \( \sigma \subfun \sigma' \).  We say the axiom \( \laxiom{\sigma}{y} \) (element \( y \) ) is enumerated dependent on \( \delta \) to mean we enumerate \( \laxiom{ \sigma \union \delta }{y} \) \( \laxiom{  \delta }{y}  \)  into \( \axset \).
\end{definition}

%We will also speak of an element \( x \) being enumerated as shorthand for the axiom \( \laxiom{\eset}{x} \) being enumerated.  During the construction of an \( \REA[n] \) set \( C \) we will also say that an axiom is enumerated dependent on \( C(n)=i \) when  the axiom is enumerated dependent on the partial function \( \sigma \) defined by  \( \sigma(n)= i \). 
We will call an axiom  \( \laxiom{\sigma}{y} \) compatible (incompatible) with a partial function \( \tau \) just if the partial function \( \sigma \)  determining when the axiom  applies is compatible (incompatible) with \( \tau \). We now argue that we can use an axiom set to approximate an \( \REA[n] \) set in well-behaved ways.  

\begin{definition}\label{def:supported}
    If \( \mathcal{A} \) is a set of \( n \)-REA axioms we say that \( \chi \in \bpfuncs \) is \( \mathcal{A} \) supported just if for all \( m,x \) with \( \chi( \pair{m}{x})\conv \) all of the following hold  
    \begin{itemize}
        \item If \( m = 0 \) then \( \chi( \pair{m}{x}) = X(x) \).
        \item If \( m > n \) then \( \chi( \pair{m}{x}) = 0 \).
        \item If \( 0 < m \leq n \) then \(\chi( \pair{m}{x}) = 1 \iff \exists(\sigma) \left[ \laxiom{\sigma}{\pair{m}{x}} \in A \land \sigma \subfun \chi\right] \)
    \end{itemize}

    If \( \chi \) is an \( \axset[s](J)  \) supported approximation we say that \( \chi \) is an  \( s \)-supported\footnote{Technically, being \( s \)-supported is relative to the choice of an enumeration of axioms but we will only use this notation when there is no ambiguity about the enumeration of the set  \( \axset(J) \). } approximation to \( J(X) \).
\end{definition}
Note that the notion of being \( s \)-supported is relative to a choice of a canonical enumeration of axioms for \( J(X) \).  Also, it's worth observing the following point.

\begin{remark}\label{rmk:unique-supported}
         Given any set of axioms \( \axset \) and finite domain \( D \subset \omega \) there is a unique \( \axset \)-supported   \( \chi \) with domain \( D \) .         
\end{remark}

 It turns out that for any \( \REA[n] \) operator \( J \) we can effectively find an enumeration of axioms \( \axset[s] \) and a computable functional  \( Y^{X}_s \) with \( \lim_{s \to \infty} \dom Y_s = \omega \),  such that \( Y^{X}_s \) is always an \( \axset[s] \) supported approximation to \( J(X) \)  and infinitely often \( Y^{X}_s \subfun J(X) \).  A proof of this claim can be found in the appendix to give the reader an idea how this can work.  However, we don't actually need the full strength of this result and will instead directly specify a stagewise enumeration of axioms and a \( \axset[s] \)  approximation \( A_s \)  to the \( \REA[n] \)  sets we are constructing.   We will ensure this approximation has the following property (where \( A \) here is the  \( \REA[n] \) set being constructed).

 \begin{condition}\label{cond:building-A:infinite-correctness} 
\( \existsinf(s) A_s \subfun A \). 
\end{condition}

In fact, our approximation to \( A \) (which we define below) will turn out to satisfy the stronger condition that there are infinitely many stages such that \( \forall(t > s) A_s \subfun A_t \).  Note that, this implies that our approximation \( A_s \) is a \( \Delta^0_2 \) approximation to \( A \).

 % \( A, Y_i \) to have the stronger property of being equal to the limit of their approximations, e.g., \( A = \lim_{s\to\infty} A_s \).

%  \begin{lemma}\label{lem:correct-supported-approx}
%      Given a finite set \( D \subset \omega \), \( \REA[n] \) operator \( J \) and base \( X \) there is a \( \axset(J) \) supported approximation \( \chi \) to \( J(X) \) with \( \chi \subfun J(X) \).   
%  \end{lemma}

% Note that for all sufficiently large \( s \) it will also be the case that \( \chi \) is \( s \)-supported.
% \begin{proof}
%         We seek to build a set \( D^{n} \) such that \( \chi \) is the unique \( \axset(J) \) supported approximation with domain \( D^{n} \).

%         Consider the case where \( n = 1 \).  It would be enough to identify every element \( y \in \setcol{J(X)}{1} \) and add the domain of the first axiom \( \laxiom{\sigma}{\pair{1}{y}} \) compatible with \( J(X) \) to \( D \) to build \( D^{1} \).  In the general case we simply build \( D^{k+1} \) (with \( D^{0} = D \) ) by adding the domain of any axioms needed to get \( J(X) \) correct on the \( n-k \)-th column of \( D^{k} \).  It is evident that \( \chi \) has the desired property.             
%   \end{proof}   

\section{Construction Overview}
\label{sec:overview}

We adopt the following notation for easier manipulation of \( \REA[n] \) sets defined in terms of columns pursuant to \cref{def:rea-op}.

\begin{definition}\label{not:column-rea-notation} \hfill
\begin{itemize}

    \item \( A \Tplus+ B \eqdef \TPlus_{n \in \omega} \setcol{A}{n} \Tplus \setcol{B}{n} \).
    \item For \( A \) an \REA(X)[n] set we let \( A\colrestr{s} = \set{\pair{k}{x}}{k \leq n \land x < s} \). 

    \item\label{cond:building-A:large} \( l_s > 4 \) is a number chosen large\footnote{Specifically, we will need that \( l_s \) is large enough that if \( x \) is mentioned at or before stage \( s \) then \( l_s > \pair{k}{x+2}, 0 \leq k \leq 3 \) so that \( l_s \) is large enough to see the \textit{next} two elements in each column.} at the \textit{end} of stage \( s \).
    % \item \( X_e \) denotes the \( e \)-th \( 2 \)-REA set, i.e.,  \(X_e = \REAop(\eset){e}{2} \)
    % \item For partial functions \( \chi, \xi \) we denote the fact that \( \chi, \xi \) disagree on the first \( k \) columns, i.e., \( \exists(k' \leq k)\exists(x)\left(\chi(\pair{k'}{x})\conv \neq \xi(\pair{k'}{x})\conv  \right) \), by \( \chi \incompat[k] \xi \).
\end{itemize}
\end{definition}

Note that if both \( A \) and \( B \) are both \( \REA[n] \) then so is \( A \Tplus+ B \).  Using these conventions we can now describe the broad outline of our proof of \cref{thm:nonextendable-three-REA}.  

To establish the desired claim we'll need to build a properly \( \REA[3] \) set \( A \) that can't be extended to a properly \( \REA[4] \) set.  In other words, for every \( i \in \omega  \) the set \( A \Tplus \REset(A){i} \) must not be properly \( \REA[4] \).  So in addition to the set \( A \) we'll build a sequence of  \( \REA[3] \) sets \( A \Tplus+ Y_i \) Turing equivalent to the sets \( A \Tplus \REset(A){i} \) (note that by using \( A \Tplus+ Y_i \) rather than just \( Y_i \)  we avoid the need for unnecessarily copying changes to \( A \) over to \( Y_i \)).  To verify these equivalences we'll also need computable functionals so, in addition to the \( \REA[3] \) set \( A \) we'll also build sets \( Y_i \) and functionals \( \Gfunc{i}{}, \Tfunc{} \) to satisfy all of the following requirements where \( X_e \) denotes the \( e \)-th \REA[2] set.

\begin{requirements}
\require{P}{i}  \Gfunc{i}{A \Tplus \REset(A){i}} = Y_i \land \Tfunc{Y_i} = \REset(A){i}  \\
\require{R}{j,e} \recfnl{j}{A}{} \neq X_e \lor \recfnl{j}{X_e}{} \neq A 
\end{requirements}

We now  observe that satisfying the above requirements is sufficient to prove the claimed theorem.  However, the impatient reader may wish to jump ahead to the informal description of how \req{R}{j,e} is met in \cref{sec:overview:req-R}.  Note that the general method there will be familiar to anyone familiar with the methods in \cite{Soare1982Relative,Cholak1994Iterated} and broadly resembles the construction of properly \( n \)-\re sets (with some extra bookkeeping).  Readers eager to jump ahead to the unique challenges and features of this construction should turn to \cref{sec:overview:interaction} for an informal discussion or to the full construction in \cref{sec:full}.

First we observe that our choice to avoid index profusion by using \( j \) as the index for both computations in  \req{R}{j,e} is harmless.

\begin{lemma}\label{lem:reqs-enough}
    Suppose \( A  \) and \(  Y_i, i \in \omega \) satisfy the requirements \req{P}{i} and \req{R}{j,e} for all \( i,j,e \) then \( A \) is properly \( \REA[3] \) but can't be extended to a properly \( \REA[4] \) set.     
\end{lemma}
\begin{proof}
 If \req{P}{i} is satisfied for all \(i \) it follows that the \( \REA[3] \) set \( A \Tplus+ Y_i  \) is Turing equivalent to  \( A \Tplus \REset(A){i} \).  It only remains to demonstrate that \( A \) is properly \( \REA[3] \).

Suppose, for a contradiction, that \( A \Tequiv X_e \) for some \( \REA[2] \) set \( X_e \). Then either \( \setcol{A}{3} = \eset \), making \( A \) literally \( \REA[2] \), or \( \setcol{A}{3} \neq \eset \).  In the first case  \req{R}{j,e} would fail when \( j \) is an index for the identity function and \( e \) an index for the empty set.  In the second case suppose that the Turing equivalence is witnessed by \( \recfnl{j_0}{A}{} = X_e \) and \( \recfnl{j_1}{X_e}{} = A \) and \( z \) is the least element in \( \setcol{A}{3} \).   Now let \( \recfnl{j}{Z}{} \) be the functional which asks if \( \pair{3}{z} \in Z \) and if not computes \( \recfnl{j_1}{Z}{} \) and if so computes \( \recfnl{j_0}{Z}{}  \).  Hence, if \( A \) is equivalent to some \( \REA[2] \) set \( X_e \) it follows that some \req{R}{j,e} isn't satisfied.      
\end{proof}

We will build  \( A \) or \( Y_i \)  by enumerating axioms into \( \axset \) and \(\axset(Y_i) \) respectively which, by the remarks above, uniquely determines the corresponding sets.  We will define approximations \( A_s, Y_{i,s} \) to the sets \( A, Y_i \) and adopt the convention that an attempt to enumerate \( x \) into \( A \) (or \( Y_i \)) at stage \( s \) means enumerating \( x \) dependent on a large initial segment of the prior columns as defined by \( A_s, Y_{i,s} \).  

We now formally specify our approximation \( A_s \) and specify some properties it will be constructed to have.

\begin{property}\label{cond:building-A}\leavevmode
\begin{propenum}
     \item\label{cond:building-A:unique} \( A_s \) is the unique \( s \)-supported approximation (i.e., \( \axset[s] \)-supported) with \( \dom A_s = \omega\colrestr{l_s} \). %\set{\pair{k}{x}}{k,x < l_s} \).

     \item \label{cond:building-A:largeuse} If \( x \) is enumerated into \( \setcol{A}{k} \) at stage \( s \) then \( x \nin \setcol{A_{s-1}}{k}  \) and \( x \) is enumerated dependent on \( \setcol{A_s}{k'} \) for \( k' < k \) (hence on prior columns of  \( A \) up to height \( l_s \)). %\todo{Don't thinkweneed the assumption that x nin A}.

    % \item\label{cond:building-A:is-limit} \( \lim_{s \to \infty} A_s = A \).
    \item \label{cond:building-A:even-stages} At odd stages no axioms are enumerated into \( \axset \) and at most one axiom is enumerated into \( \axset \) at even stages.  
    % \item \label{cond:building-A:reset-i} Without loss of generality we can assume that \( x \) doesn't enter \( \REset(X){i} \) at stage \( s \) unless \( s \) is odd and \( x, i < s \).  
    % \item \label{cond:building-A:only-one} At most one axiom is enumerated into \( \axset \) at any stage.
    % \item\label{cond:building-A:axioms-over-A}  Furthermore, we assume  that axioms for 
    % \item\label{cond:building-A:acceptable} \( A_s, \axset[s] \) form an acceptable approximation to \( A \).
\end{propenum}
\end{property}

To avoid the potential circularity induced by the interaction of \cref{cond:building-A:large,cond:building-A:largeuse} we queue any request made to enumerate an element made during stage \( s \)  until after \( l_s \) is chosen at the end of the stage.  Using \( \setcol{A_s}{< n} \) and \( l_s \)  we can build  an axiom satisfying \cref{cond:building-A:largeuse} for an element queued for enumeration into \( \setcol{A}{n} \) which, in turn, lets us determine \( \setcol{A_s}{< n + 1} \).  Working inductively, this lets us identify an axiom for the element (if any) enumerated during stage \( s \) which satisfies both constraints. 

% In fact, we will infinitely often restrain  the enumeration of any elements of the form \( \pair{k}{x} \) with \( x < l_s \) from being enumerated into \( A \) ensuring that \( A \) satisfies the following condition. 
% \begin{condition}\label{cond:building-A:infinite-correctness} 
% \( \existsinf(s) A_s \subfun A \). 
% \end{condition}

We will also treat \( \REset(A){i} \) as the result of an enumeration of axioms \( \axset<i>[s] \) and define an approximation \( \REset[s](A_s){i} \) that satisfies the following properties.   %Note that, by a slight abuse of notation, we will talk about \( \REset[s](A_s){i} \) being \( \axset<i>[s] \)-supported or \( \REset[s](A_s){i} \) being built by the axioms \( \axset<i> \) when technically it's  

\begin{property}\label{cond:building-re} \leavevmode
\begin{propenum}
\item \label{cond:building-re:odd-stages} No axioms are enumerated into \( \axset<i> \) at even stages. 
\item \label{cond:building-re:one-small} At any stage \( s \) there is at most one \( i \) and \( x \) such that an axiom enumerating some \( x \) into  \( \REset(A){i} \) is enumerated  and \( i, x < s \). 
\item \label{cond:building-re:use} If an axiom enumerating \( x \) into \( \REset(A){i} \) is enumerated at stage \( s \) then \( x \) is enumerated dependent on \( A_s\colrestr{l_s} \).   
\item \label{cond:building-re:approx} \( \REset[s](A_s){i} \) denotes the unique finite partial function with domain \( \omega\restr{l_s} \) that is \( \axset<i>[s] \)-supported.   %\( \REset[s](A_s){i} \) denotes the unique finite partial function with domain \( \omega\restr{l_s} \) such that the partial function whose first three columns are given by \( \colrestr{A_s}{l_s}\) and whose fourth column is given by  \( \REset[s](A_s){i} \)  is \( \axset<i>[s] \)-supported.
\item  \label{cond:building-re:true}  \( \REset(A){i} \) is the set built by applying the axioms  \( \axset<i> \) to the set \( A \).  
\end{propenum}
\end{property}

\Cref{cond:building-re:use} ensures that only axioms which agree with our current approximation \( A_s \) to \( A \)  are enumerated at stage \( s \) and are canceled by any change in \( A_s \).  Note, these properties are only possible to meet, as we will now verify, because of \cref{cond:building-A:infinite-correctness}.

\begin{lemma}\label{lem:re-enum-exists}
    For all \(i \) there is an  (uniformly specified) enumeration \( \axset[s]<i> \) and an effective approximation \( \REset[s](A_s){i} \) satisfying \cref{cond:building-re}.
\end{lemma}
\begin{proof}
We can assume we start with an enumeration of axioms \(  \widehat{\axset[s]<i>} \) such that \( x \in \REset(A){i} \) iff \( \exists(\laxiom{\sigma}{x} \in \axset<i>)\left( \sigma \subfun A \right) \)  and satisfies the usual rules for the enumeration of an \( A \)-\re set such as the enumeration of at most one element a stage etc.   We now define  \( \axset[s]<i> \) in terms of the sets \( \axset[s - 1]<i> \), \(i < s \).    

If \( s \) is even we do nothing. So assume \( s \) is odd and let \( t < s \) be the minimal value such that we can find \( i, y < s, \lh{\sigma} < s \) such that  \( \laxiom{\sigma}{y} \in \widehat{\axset[t]<i>} \), \( \sigma \subfun A_{s-1}  \) and \(  \REset[s-1](A_{s-1}){i}(y) = 0 \).  If no such values can be found do nothing.  If they are found enumerate the axiom \( \laxiom{A_s}{y} \) into \( \axset[s]<i> \).   By construction, we clearly satisfy \cref{cond:building-re:odd-stages,cond:building-re:one-small,cond:building-re:use,cond:building-re:approx}.  It remains to show that \cref{cond:building-re:true} holds.  By \cref{cond:building-A:infinite-correctness,cond:building-A:even-stages} we can assume that infinitely often we have an odd stage \( s \) with \( A_{s-1} \subfun A_s \subfun A \), ensuring that eventually every element in  \( \REset(A){i} \) gets an axiom enumerating it.  
\end{proof}

We now give an overview of how each requirement operates.  

\subsection{Overview of \( \req*{R}{j,e} \)}
\label{sec:overview:req-R}

The basic approach to meeting  \req{R}{j,e} is to use the fact that \( A \) is \( \REA[3] \) to change our minds about the behavior of some initial segment of \( A \) more times than \( X_e \) is able to track. We now give an informal description of the process which we depict graphically in \cref{fig:single-R-req}.  In the figure stages at which a well-behaved back and forth computation exists are indicated by shaded/hatched regions. The region of \( A \) which is either hatched or shaded indicates the use of a computation of the shaded region of \( X_e \), which in turn extends the use of a computation of the shaded region of \( A \).  We describe such computations as well-behaved when (at stage \( s \)) an \( s \)-supported region of  \( A \) (hatched and shaded) computes (via \( \recfnl{j}{}{} \))  \( X_e \) on (at least) an  \( s \)-supported region (shaded region) which in turn computes (via \( \recfnl{j}{}{} \))  \( A(\pair{3}{c}) \).  We leave cells blank if empty --- unless they've just been canceled by an enumeration into a prior column, in which case we place a \( 0 \) in the cell.  Finally, we circle locations which conflict with a commitment made by \( \recfnl{j}{}{} \) at a prior stage.  

\begin{figure}
    \centering
    \begin{tikzpicture}
        \REAmatrix(x)<{\( X_e \) }>{4}{2}
        % \node[right=.05cm of x] {\( \stackrel[{{ \recfnl{j}{X_{e}}{} }}]{ \recfnl{j}{A}{} }{\leftrightarrows} \) };
        \REAmatrix(m)<{\( A \) }>[right = .05cm of x]{4}{3}
        % \REAmatrixEnum{4-3}
        % \REAmatrixCircle(x-circle){4-3}
        \LineLabelCell[below right]{4-3}{\( c \) }
        \node [below=1cm,xshift=1cm] {Stage \( s_{-1} \)};
    \end{tikzpicture} 
    \begin{tikzpicture}
        \REAmatrix(x)<{\( X_e \) }>{4}{2}
        \REAmatrixFillRect<fill=LightGray>{4-1}{4-2}
        \REAmatrixFillRect<fill=LightGray>{3-1}{3-1}
        % \REAmatrixEnum{4-2}
        % \node[right=.05cm of x] {\( \stackrel[{{ \recfnl{j}{X_{e}}{} }}]{ \recfnl{j}{A}{} }{\leftrightarrows} \) };
        \REAmatrix(m)<{\( A \) }>[right = .05cm of x]{4}{3}
        \REAmatrixFillRect<fill=LightGray>{4-1}{4-3}
        \REAmatrixFillRect<pattern=north east lines>{3-1}{3-3}
        % \REAmatrixEnum{4-3}
        % \REAmatrixCircle(x-circle){4-3}
        % \LineLabelCell[below right]{4-3}{x}
        \node [below=1cm,xshift=1cm] {Stage \( s_{0} \) };
    \end{tikzpicture} \,
        \begin{tikzpicture}
        \REAmatrix(x)<{\( X_e \) }>{4}{2}
        \REAmatrixEnum{4-2}
        % \REAmatrixFillRect<pattern=north east lines>{3-1}{4-2}
        % \REAmatrixFillRect<fill=LightGray>{4-1}{4-2}
        % \node[right=.05cm of x] {\( \stackrel[{{ \recfnl{j}{X_{e}}{} }}]{ \recfnl{j}{A}{} }{\leftrightarrows} \) };
        \REAmatrix(m)<{\( A \) }>[right = .05cm of x]{4}{3}
        % \REAmatrixFillRect<fill=LightGray>{4-1}{4-3}
        \REAmatrixEnum{4-3}
        \REAmatrixCircle{4-3}
        % \REAmatrixCircle(x-circle){4-3}
        % \LineLabelCell[below right]{4-3}{x}
        \node [below=1cm,xshift=1cm] {Stage \( s_{0} +1 \)};
    \end{tikzpicture} \,
    \begin{tikzpicture}
        \REAmatrix(x)<{\( X_e \) }>{4}{2}
        \REAmatrixFillRect<fill=LightGray>{3-1}{4-2}
        \REAmatrixFillRect<fill=LightGray>{2-1}{2-1}
        \REAmatrixEnum{3-2}
        \REAmatrixEnum{4-2}
        % \node[right=.05cm of x] {\( \stackrel[{{ \recfnl{j}{X_{e}}{} }}]{ \recfnl{j}{A}{} }{\leftrightarrows} \) };
        \REAmatrix(m)<{\( A \) }>[right = .05cm of x]{4}{3}
        % \REAmatrixFillRect<fill=LightGray>{4-1}{4-3}
        \REAmatrixEnum{4-3}
        % \REAmatrixCircle(x-circle){4-3}
        % \LineLabelCell[below right]{4-3}{x}
        \REAmatrixFillRect<fill=LightGray>{4-1}{4-3}
        \REAmatrixFillRect<pattern=north east lines>{2-1}{3-3}
        \node [below=1cm,xshift=1cm] {Stage \( s_{1} \)};
    \end{tikzpicture}\\
        \begin{tikzpicture}
        \REAmatrix(x)<{\( X_e \) }>{4}{2}
        % \REAmatrixFillRect<fill=LightGray>{3-1}{4-2}
        % \REAmatrixFillRect<fill=LightGray>{2-1}{2-1}
        \REAmatrixEnum{3-2}
        \REAmatrixEnum{4-2}
        % \node[right=.05cm of x] {\( \stackrel[{{ \recfnl{j}{X_{e}}{} }}]{ \recfnl{j}{A}{} }{\leftrightarrows} \) };
        \REAmatrix(m)<{\( A \) }>[right = .05cm of x]{4}{3}
        % \REAmatrixFillRect<fill=LightGray>{4-1}{4-3}
        \REAmatrixEnum{2-2}
        \REAmatrixEnum[0]{4-3}
        \REAmatrixCircle{4-3}
        % \REAmatrixCircle(x-circle){4-3}
        % \LineLabelCell[below right]{4-3}{x}
        % \REAmatrixFillRect<fill=LightGray>{4-1}{4-3}
        % \REAmatrixFillRect<pattern=north east lines>{2-1}{3-3}
        \node [below=1cm,xshift=1cm] {Stage \( s_{1} +1 \)};
    \end{tikzpicture} \; \;
        \begin{tikzpicture}
        \REAmatrix(x)<{\( X_e \) }>{4}{2}
        \REAmatrixFillRect<fill=LightGray>{4-1}{4-2}
        \REAmatrixFillRect<fill=LightGray>{3-1}{3-1}
        \REAmatrixEnum{4-2}
        \REAmatrixEnum[0]{3-2}
        \REAmatrixEnum{2-1}

        % \node[right=.05cm of x] {\( \stackrel[{{ \recfnl{j}{X_{e}}{} }}]{ \recfnl{j}{A}{} }{\leftrightarrows} \) };
        \REAmatrix(m)<{\( A \) }>[right = .05cm of x]{4}{3}
        \REAmatrixFillRect<fill=LightGray>{4-1}{4-3}
        \REAmatrixFillRect<pattern=north east lines>{3-1}{3-3}
        \REAmatrixEnum{2-2}
        % \REAmatrixEnum{4-3}
        % \REAmatrixCircle(x-circle){4-3}
        % \LineLabelCell[below right]{4-3}{x}
        \node [below=1cm,xshift=1cm] {Stage \( s_{2} \) };
    \end{tikzpicture} \; \;
    \begin{tikzpicture}
          \REAmatrix(x)<{\( X_e \) }>{4}{2}
        % \REAmatrixFillRect<fill=LightGray>{3-1}{4-2}
        % \REAmatrixFillRect<fill=LightGray>{2-1}{2-1}
        \REAmatrixEnum{2-1}
        \REAmatrixCircle{2-1}
        \REAmatrixEnum{4-2}
        % \node[right=.05cm of x] {\( \stackrel[{{ \recfnl{j}{X_{e}}{} }}]{ \recfnl{j}{A}{} }{\leftrightarrows} \) };
        \REAmatrix(m)<{\( A \) }>[right = .05cm of x]{4}{3}
        % \REAmatrixFillRect<fill=LightGray>{4-1}{4-3}
        \REAmatrixEnum{4-3}
        \REAmatrixEnum[0]{2-2}
        \REAmatrixEnum{1-1}
        \LineLabelCell[right]{2-3}{\( s_1 \) comp applies}
        \LineLabelCell<white>[right]{3-3}{perm disagree w/}
        \LineLabelCell<white>[right]{4-3}{circled location}
        % \REAmatrixCircle(x-circle){4-3}
        % \LineLabelCell[below right]{4-3}{x}
        % \REAmatrixFillRect<fill=LightGray>{4-1}{4-3}
        % \REAmatrixFillRect<pattern=north east lines>{2-1}{3-3}
        \node [below=1cm,xshift=1cm] {Stages \( s > s_2 \)};
    \end{tikzpicture}

    \caption{Meeting \req*{R}{j,e}}
    \label{fig:single-R-req}
\end{figure}

  We pick a value \( c \) at some stage \( s_{-1} \)  which we hold out of \( \setcol{A}{3} \).  If we never see a stage \( s_0 \) and a well-behaved computation from  \( A \) to  \( X_e \) and back then we are done.   If we do see such a well-behaved computation then at stage \( s_0 +1 \)  we enumerate \( c \) into \( \setcol{A}{3} \) dependent on some large value staying out of \( \setcol{A}{2} \).      

If at some stage \( s_1 > s_0 +1 \) we again see a well-behaved back and forth computation   (meaning \( X_{e,s_1} \) disagrees with \( X_{e,s_0} \) on the region shaded at stage \( s_0 \)) we enumerate an element \( b \)  into \( \setcol{A}{2} \) canceling \( c \) from  \( A_{s_1 +1} \).  This restores the computation from \( A \) to \( X_e \) seen at stage \( s_0 \).  Note that the only way we can again see a well-behaved back and forth computation (given we restrain enumeration into \( A \)) is if our approximation to \( X_e \) restores the (used part of the) state it had at \( s_0 \)  by enumerating an element into \( \setcol{X_e}{1} \) to cancel any changes.          %The only way  \( X_e \) can agree with the computation from \( A \) is to  enumerate an element  into \( \setcol{X_e}{1} \) which cancels the (small) changes made to \( X_e \) since stage \( s_0 \) .

Now suppose that at some stage \( s_2 \) we again see a well-behaved back and forth computation.        We  respond by enumerating an element into \( \setcol{A}{1} \) at stage \( s_2 +1 \)  canceling the enumeration made at stage \( s_1 \) into \( \setcol{A}{2} \).  This restores the computation from \( A \) to \( X_e \) seen at stage \( s_1 \).   But now \( X_e \) is unable to match this change as it can't cancel the element enumerated into~\( \setcol{X_e}{1} \) and, as we will prove in \cref{sec:overview:req-R:positive-change-property}, the fact that \( A \) computed an \( s \)-supported approximation to \( X_e \) at \( s_1 \) means this element was predicted by \( A \) to be out of \( X_e \) at \( s_1 \).  This ensures that, provided we restrain modifications to \( A \), we never again see a well-behaved \( \recfnl{j}{} \) computation from \( A \) to \( X_e \) and back, guaranteeing that \( \req{R}{j,e} \) is satisfied.

We now formalize the argument made in this sketch and demonstrate that if the equivalence of \( A \) and \( X_e \) is witnessed by \( \recfnl{j}{} \) then we will see infinitely many computations we've been calling well-behaved.    %in the sense that the only way to remove an element a well-behaved computation places in \( \setcol{X_{e,s}}{n+1} \) is to add an element the computation placed out of \( \setcol{X_{e,s}}{\leq n} \).  

\subsubsection{Positive Change Property}
\label{sec:overview:req-R:positive-change-property}

We now show that \( \REA \) sets have the following positive change property.  That is, the only way an approximation to an \( \REA \) set can change is by enumerating a new element.  We prove the results in this section for  \( \REA[n] \) sets for arbitrary \( n \), as this result will help us understand the strategy used to satisfy  \( \req{P}{i} \). 

\begin{lemma}\label{lem:positive-change-property}
    Suppose \( \chi, \chi' \) are \( s, s' \)-supported approximations to an \( \REA(Z)[n] \) set \( Y \) with \( s < s' \) and \( \dom \chi' \supseteq \dom \chi \).  If \( \chi(\pair{m}{x}) = 1 \) but \( \chi'(\pair{m}{x}) =0 \) then \( \chi \incompat[m-1] \chi' \).  Moreover, there is a \( y \) and  \( m' < m \) with \( \chi(\pair{m'}{y}) = 0 \) but \( \chi'(\pair{m'}{y}) = 1 \).
\end{lemma}
\begin{proof}
To establish the first claim suppose that the hypotheses of the lemma are satisfied. By \cref{def:supported} there is some axiom \( \laxiom{\sigma}{y} \in \axset[s](Y) \) with \( y = \pair{m}{x} \) and \( \sigma  \subfun \chi \).  As \( \dom \sigma \subset \setcol{\omega}{m-1} \) if \( \chi' \compat[m-1] \chi  \) then \( \sigma \subfun \chi' \) as \( \dom \chi' \supseteq \dom \chi \).  But then by \cref{def:supported} we would have  \( \chi'(y) = 1 \), contrary to the assumptions of the lemma.

We now prove the second claim by induction.  As \( \setcol{Y}{1} \) is \re, the claim obviously holds for \( m=2 \).  Now suppose the claim holds for any \( m' \leq m \) and that the hypotheses of the lemma are satisfied for \( m+1 \).  Thus, \( \chi \incompat[m] \chi' \).  If there is some \( y = \pair{m'}{x} \) with \( m' < m+1 \) and \( \chi(y) =0 \) and \( \chi'(y) = 1 \), we are done.  If not there must be a \( y \) with \( \chi(y) =1 \) and \( \chi'(y) = 0 \).  The result now follows by the inductive hypothesis applied to \( y = \pair{m'}{x} \).
\end{proof}

We now give a slightly modified version of the above lemma that is specifically phrased in terms of undoing a previous change.  

\begin{lemma}\label{lem:return-forces-change}
    Suppose \( \chi, \chi', \chi'' \) are \( s, s', s'' \)-supported approximations to \( \REAop(Z){e}{n} \) with \( s < s' < s'' \), \( \chi'' \supfun \chi \) and \( \chi' \incompat[m] \chi  \) then \( \chi'' \incompat[m-1] \chi' \).  
\end{lemma}
\begin{proof}
Note that WLOG we may assume \( \dom \chi'' \supset \dom \chi' \) by restricting \( \chi' \) to the domain of \( \chi \).   With this assumption made, the pair \( \chi, \chi' \) satisfy the hypotheses of \cref{lem:positive-change-property} or \( \chi', \chi'' \) do.  As \( \chi'' \supfun \chi \) whichever way we have \( \chi'' \incompat[m-1] \chi' \).
\end{proof}

We now define a predicate which holds at \( x,s \) if it appears that \( A_s \) computes enough of \( X_e \) for \( X_e \) to correctly compute the value of \( A_s(x) \).  Note that, in what follows, we will make use of an approximation \( X_{e,s} \) to \( X_e \).  However, somewhat surprisingly, all that matters about this approximation is that \( X_{e,s} \) be \( s \)-supported and that, infinitely often, it's domain includes any finite subset.  As such, we can simply take \( X_{e,s} \) to be the unique \( s \)-supported approximation such that \( \dom \setcol{X_{e,s}}{1} = \dom \setcol{X_{e,s}}{2} = l_s  \) even if this approximation isn't even guaranteed to be (pointwise) infinitely often correct.

\begin{definition}\label{def:backforth-witness}
    We define \( \SelfComp[j,e]{x}{s} \) (written \( \SelfComp{x}{s} \) when \( j,e \) is clear from context) to hold just when there is some \( s \)-supported approximation \( \xi  \) to \( X_e \) satisfying
    \begin{itemize}
        % \item \label{def:backforth-witness:dom} \( \dom \xi \subset  \set{\pair{l}{y}}{l, y \leq s}  \)
        \item  \( \recfnl[s]{j}{A_s}{} \supfun \xi \) 
        \item \(\recfnl[s]{j}{\xi}{} \supfun A_s\restr{x+1} \)
    \end{itemize}
\end{definition}

We note that if \( \req{R}{j,e} \) fails we can always wait for a stage at which our approximation witnesses this failure.

\begin{lemma}\label{lem:inf-many-good-comps}
   If \( \recfnl{j}{A}{} = X_e \land \recfnl{j}{X_e}{} = A \),  \( A_t \subfun A \)  then for any \( x \) there are infinitely many stages at which \( \SelfComp{x}{s} \) holds and is witnessed by \( A_s, \xi \) where  \(A \supfun A_s \supfun A_t \).
\end{lemma}
\begin{proof}
Suppose \( A_t \) is as in the lemma.  We first note that it is enough to show one such \( s > t \) exists, since further witnesses may be generated by applying this result.  Moreover,  we may presume that \( x \in \dom A_t \) since, as the approximation to \( A \) is infinitely often correct,  we can simply wait for some later stage at which \( x \in \dom A_t \) and \( A_t \subfun A \) then argue as below.  

Since the reductions are total we can find \( \xi, \chi \) such that \( \xi \subfun X_e \) is an \( \axset(X_e) \) supported approximation to \( X_e \) satisfying \( \forall(y \leq x)\left(\recfnl{j}{\xi}{y} = A(y) = A_t(y)\right) \) and  \( \chi  \) is an \( \axset \) supported approximation to \( A \) satisfying  \( A \supfun \chi  \supfun A_t \) with \( \recfnl{j}{\chi}{} \supfun \xi \).  Now let \( t' > t \) be large enough that every axiom needed to ensure \( \xi, \chi \) are \( \axset(X_e), \axset \) supported has been enumerated by stage \( t' \) , \( \recfnl[t']{j}{\xi}{x} = A(x) = A_t(x) \), \( \recfnl[t']{j}{\chi}{} \supfun \xi \)  and the domain of \( \xi \) is contained in \( \set{\pair{l}{x}}{l, x \leq t'} \).  Finally, %(by the remark following \cref{def:acceptable-enumeration}) 
choose \( s > t' \) so that \( A_s \subfun A \) and \( \dom A_s \supseteq \dom \chi \union \dom A_t \) which guarantees \( A_s \supfun \chi \).  But now note that all the conjuncts in  \cref{def:backforth-witness} are satisfied so \( \SelfComp{x}{t} \) holds.
\end{proof}

Using this definition we can make precise the idea of an approximation changing its mind.

\begin{definition}\label{def:k-flipflop}
    An increasing sequence of \( k+1 \)  stages \( s_0, s_1, \ldots, s_k \) is \( k \)-flipflopping (for \( \recfnl{j}{}{}, X_e \)) at \( x \)  denoted \( \SelfComp<k>[j,e]{x}{s_0, s_1, \ldots, s_k} \)  just if
    \begin{enumerate}
        % \item \( s_l \) is strictly increasing in \( l \).
        % \item\label{def:k-flipflop:alternation} \( \forall(l \leq k - 2)\left( A_{s_l} \subfun A_{s_{l+2}} \right) \)
        \item\label{def:k-flipflop:change} \( \forall(l \leq k -1 )\left( A_{s_l}(x) \neq A_{s_{l+1}}(x) \right) \)
         \item\label{def:k-flipflop:computation} \( \forall(l \leq k -1)\left( \SelfComp[j,e]{x}{s_l} \right) \) is witnessed by some approximation \( \xi_l \) to \( X_e \).
        \item\label{def:k-flipflop:alternation} \( \forall(l \leq k - 2)\left( A_{s_l} \subfun A_{s_{l+2}} \right) \)
    \end{enumerate}

    An element \( x \) is \( k \)-flipflopping at stage \( s \) (for \( X_e, \recfnl{j}{}{} \)), denoted  \( \SelfComp<k>[j,e]{x}{s} \) just if there is an increasing sequence of stages \( s_0, s_1, \ldots, s_{k-1}, s \) such that  \( \SelfComp<k>{x}{s_0, s_1, \ldots, s_{k-1}, s} \).  %Finally, \( x \) is \( k \)-flipflopping (for \( X_e, \recfnl{j}{}{} \)) if there there is a stage \( s \) at which \( x \) is \( k \)-flipflopping at stage \( s \) and \( A \supfun A_s \).      

    When \( j,e \) are clear from context we will omit mentioning them.  
\end{definition}

We can now provide the framework for diagonalizing against \( X_e \).

\begin{lemma}\label{lem:max-flip-flops}
If \( \SelfComp<3>[j,e]{x}{s} \) then \( \SelfComp[j,e]{x}{s} \) fails to hold.    
\end{lemma}
\begin{proof}
Let \( s_3 = s \) and  suppose, for a contradiction, that both \( \SelfComp<3>[j,e]{x}{s_0, s_1, s_2, s_3} \)  and \( \SelfComp[j,e]{x}{s_3} \) hold and are witnessed by the approximations \( \xi_l, l < 4\) to \( X_e \).  By part \ref{def:k-flipflop:alternation} of \cref{def:k-flipflop} and \cref{def:backforth-witness} it follows that \( \forall(l \leq k - 2)\left( \xi_{l} \subfun \xi_{l+2} \right) \).  

As \( A_{s_0}(x)\conv \neq A_{s_{1}}(x)\conv \) we have \( \xi_0 \incompat[2] \xi_1 \)  by \cref{lem:return-forces-change} we have \( \xi_2 \incompat[1] \xi_1 \) and as  \( A_{s_1}(x)\conv \neq A_{s_{2}}(x)\conv \) we can apply  \cref{lem:return-forces-change} again to infer that \( \xi_3 \incompat[0] \xi_2 \).  But \( s \)-supported approximations to an \( \REA(Z)[n] \)  can't disagree on column \( 0 \).  Contradiction. 
\end{proof}

With this in mind we can now reiterate the basic strategy  we would use to meet \req{R}{j,e} if
there were no other requirements.  Choose some \( x \)  and hold it out of      \( \setcol{A}{3}
\) until we observe a computation witnessing \( \SelfComp{x}{s} \), i.e., a computation of an \( s
\)-supported approximation \( \xi \) of \( X_e \) which computes \( x \nin \setcol{A}{3} \).  This
stage becomes \( s_0 \) in the lemma above and \( \xi  \) becomes \( \xi_0 \).    Now enumerate \( x
\) into \( \setcol{A}{3} \) dependent on a very large initial segment of  \( \setcol{A}{2} \) and
wait until we again see \( \SelfComp{x}{s} \) with a witness whose domain contains the domain of \(
\xi_0 \).  This new stage becomes \( s_1 \) and this new witness becomes \( \xi_1 \).  Enumerate an
element into \( \setcol{A}{2} \) depending on a large initial segment of \( \setcol{A}{1} \)
canceling the enumeration of \( x \) and now wait until we again see  \( \SelfComp{x}{s} \) with \(
A_s \supfun A_{s_0} \) and a witness whose domain extends that of \( \xi_1 \).  Repeat the
cancellation one last time via an enumeration into \( \setcol{A}{1} \) for a guaranteed win.  At
each point at which we wait on a computation we restrain any elements from being enumerated into any
initial segment appearing in a computation we've used in this process.  The lemma above ensures that
this process ends in a victory.

\subsection{Overview of \( \req*{P}{i} \)}
\label{sec:overview:req-P}

To meet \( \req{P}{i} \) we must construct sets \( Y_i \) and functionals that witness  \( Y_i \Tplus A \Tequiv A \Tplus \REset(A){i}\).  As we control the construction of the sets \( Y_i \) we will simply settle on a particular way of coding \( \REset(A){i} \) into \( Y_i \) so we can share a single functional \( \Tfunc{} \) computing \( \REset(A){i} \) from \( Y_i \).  As elements can both enter and leave (our approximation to) \( \REset(A){i} \) during the construction our coding mechanism must allow \( Y_i \) to change it's mind about the value of \( \REset(A){i}(x) \).

To this end we use the \( x \)-th column of \( \setcol{Y_i}{3} \) to encode whether or not \( x \) is in  \( \REset(A){i} \).  To guess that \( x \in \REset(A){i}  \) we place \( \pair{x}{0} \) into \( \setcol{Y_i}{3} \), to revoke that guess we place \( \pair{x}{1} \) into \( \setcol{Y_i}{3} \), to guess \( x \) in \( \REset(A){i} \) we place \( \pair{x}{2} \) into \( \setcol{Y_i}{3} \) and so on.  To avoid any need to memorize the particular coding convention we define the following notation.

\begin{definition}\label{def:code-location}
\[ \codeY<i>[s]{x} = \pair{x}{k} \text{ where } k  \text{ is the least element of } \omega \text{ such that }  \setcol{Y_{i, s - 1}}{3}(\pair{x}{k})\conv  = 0 \]
When used without stage it indicates the least element of the form \( \pair{x}{k} \) not yet enumerated into \( \setcol{Y_{i}}{3}\).  
% We let \( \codeY<i>{x} = \lim_{s \to \infty} \codeY[s]<i>{x} = \codeY[\infty]<i>{x} \).  
\end{definition}

On the other hand, as we don't control the set \( \REset(A){i} \), the only way we can take full advantage of any change in \( \REset(A){i} \) is to build functionals \( \Gfunc{i}{} \) so that it always tries to map \( A_s = A_s\restr{l_s} \) to \( Y_{i,s}\restr{s} \) whenever \( i < s \) but is reset whenever \req{P}{i} is reinitialized.

\begin{definition}\label{def:funcs} \hfil \\
    \begin{enumerate}
\item \label{def:funcs:T} \(  \Tfunc[s]{Z}(x) = \begin{cases}  
                                                    1 & \text{if } \max \set{k}{Z(\pair{3}{\pair{x}{k}})\conv = 1 \lor k = -1} \equiv 0 \!\pmod{2} \\
                                                    0 & \text{otherwise} \\
                                                        \end{cases}
        \)
        \item \label{def:funcs:G} Let \( s' \) be the last stage with \(  s' \leq s \) at which \module{P}{i} is reinitialized (injured).  Then \( \Gfunc[s]{i}{X \Tplus Z}(y)\conv = Y_{i,t}(y) \) where \( t \) is the least stage satisfying  \( s \geq t \geq s' \), \( t > y \),  \( Y_{i,t}(y)\conv \),  \( X \supfun  A_t \) and  \( Z \supfun \REset[t](A_t){i}\restr{y+1} \).  If there is no such \( t \) then \( \Gfunc[s]{i}{X \Tplus Z}(y)\diverge \)  
          % \item \label{def:funcs:G} If \( s' < s\) is the last stage at which\module{P}{i} is reinitialized (injured).  Then \( \Gfunc[s]{i}{X \Tplus Z}(y)\conv = Y_{i,t}(y) \) where \( t \) is the least stage satisfying  \( s > t > y \), \( t > s' \)  \( Y_{i,t}(y)\conv \),  \( X \supfun  A_t \) and  \( Z \supfun \REset[t](A_t){i}\restr{y+1} \).  If there is no such \( t \) then \( \Gfunc[s]{i}{X \Tplus Z}(y)\diverge \)  
    \end{enumerate}
\end{definition}

Note that \cref{cond:building-A:infinite-correctness} ensures that \( \Gfunc{i}{A \Tplus \REset(A){i}} \) is total as for every \( y \) there is a stage  \( s > y \) at which \( A_s \subfun A \) and \( \REset[s+1](A){i}\restr{y+1} = \REset(A){i}\restr{y+1} \).   Also, observe that if \( x > s \) then \( \Gfunc{i}{} \) makes no commitments about \( x \) prior to stage \( s \) and thus any commitments made about \( x \) will be dependent on at least the initial \( l_s \) elements of each column of \( A \).    
% Note that \( \Tfunc{Y_{i,s}}(x)\conv \iff \REset[s](A_s){i}(x)\conv \).

% In effect \ref{def:funcs:G} says that at stage \( t \) we if \( y < t \) we enumerate an axiom asserting that \( Y_i(y) = Y_{i,t}(y) \) dependent on \( A_t \) and \( \REset[t](A_t){i}\restr{y+1} \).  Note that our axiom can depend on all of \( A_t \) (all \( l_t \) many elements in each column) but only the first \( y+1 \) elements of  \(  \REset[t](A_t){i} \) since if the \( \REset(A){i} \) use of \( \Gfunc{i}{} \) depended on the stage the computation was first observed at we would have no guarantee of convergence.     

% Note that \( \Tfunc[s]{\chi}{x}  \) treats \( \chi \) as if any element not in \( \dom \chi \)     

We build \( Y_i \) by enumerating axioms much like we do for \( A \) but, remember it is only the columnwise sum \( Y_i \Tplus+ A \) which we build to be \( \REA[3] \) not \( Y_i \) itself (this merely saves us the trouble of copying every enumeration into \( A \) over to \( Y_i \)) 
%\todo{You had red splotch here...did this need explanation?  I added a remark here dunno if it's what you wanted.}.  
Thus, the axioms enumerating elements into \( Y_i \) can depend not only on prior columns of \( Y_i \) but on prior columns of \( A \) as well.  We omit the straightforward modifications to the definition of axiom to allow for this dependence.  As with \( A \) we will simply indicate which elements we wish to enumerate into \( Y_i \) and rely on the following properties to uniquely define what axiom is enumerated and how our approximation to \( Y_i \) is affected.

% Instead of directly specifying the axioms that are enumerated into \( Y_i \) we will specify elements to be added to \( Y_i \) which gives rise to an enumeration of axioms satisfying the following properties. 

\begin{property}\label{cond:Y-approx} \leavevmode
\begin{propenum}
    \item\label{cond:Y-approx:s-supported} \( Y_{i,s} \) is the unique partial function with \( \dom Y_{i,s} = \set{\pair{k}{x}}{k < l_s +4 \land x < l_s} \)  such that \( Y_{i,s} \Tplus+ A_s \) is an \( s \)-supported approximation (i.e., \( \axset[s]({Y_i \Tplus+ A})\)-supported) (hence, \( k > 3, k < l_s +4, x < l_s \implies  Y_{i,s}(\pair{k}{x})\conv =0  \) ).
    % \item \label{cond:Y-approx:finite-nontrivial} If \( i \geq s \) then \( Y_{i,s} \subfun \eset \).   
    \item \label{cond:Y-approx:largeuse} If \( x \) is enumerated into \( \setcol{Y_i}{k} \) at stage \( s \) then \( x \) is enumerated dependent on \( \setcol{A_s}{k'} \) and  \( \setcol{Y_{i,s}}{k'} \) for \( k' < k \) (hence on prior columns of both \( A \) and \( Y_i \) up to height \( l_s \)).     
    % \item \label{cond:Y-approx:limit}\( Y_i = \lim_{s \to \infty} Y_{i,s} \)
\end{propenum}
\end{property}

Our construction will ensure that the following condition is satisfied.

\begin{condition} 
\label{cond:Y-approx:compute-right} For all \( s \),  \( \Tfunc[s]{Y_{i,s}}(x) = \REset[s]({A_s}){i}(x)  \) and  \( \Gfunc[s]{i}{A_s \Tplus \REset[s](A_s){i}}(x) = Y_{i,s}(x)  \) whenever both sides are defined.  Moreover, \( \Gfunc{i}{} \) is well-defined. 
\end{condition}

To this end, unlike the functional \( \Gfunc{i}{} \), we don't reset the set of axioms enumerated into \( \axset(Y_i) \) when the module implementing \( \req{P}{i} \) is reinitialized since \( \REset(A){i} \) isn't reinitialized.

%by \cref{cond:Y-approx:s-supported} we see that as \( x \) can only enter \( \REset{i}{A} \) at a stage \( s \geq x \) and provided we only change our mind about whether \( x \in \REset(A){i} \) at most once a stage; we need not do any extra work to ensure convergence of \( \Tfunc{Y_{i,s}}(x) \) only enumerate axioms for \( Y_i \) to ensure it has the correct value.    Furthermore, 

% We observe that, \( \Gfunc{i}{A \Tplus \REset(A){i}}(y) \) will always be defined as infinitely often \( A_s \subfun A \) by \cref{cond:building-A:infinite-correctness} and by \cref{cond:building-A:unique}. 

% Of course, if any of the first four (excluding \( 0 \)) columns of \( A \Tplus \REset(A){i} \) were computable in  prior column  \req{P}{i} would be trivial to meet.  The danger all arises from the interaction between the requirements.

\subsection{Requirement Interaction}
\label{sec:overview:interaction}

Note that we can regard \( A \Tplus \REset(A){i} \) as (equivalent to) a \( \REA[4] \) set so the same considerations about an \( \REA[n+1] \) set avoiding equivalence with an \( \REA[n] \) from \cref{sec:overview:req-R} apply but now it's our job to ensure equivalence.  Obviously, if we could leave \( \setcol{A}{1} \) empty (or even computable) then we could trivially meet \req{P}{i} simply by enumerating \( \codeY<i>[s]{x} \) into \( \setcol{Y_i}{3} \) whenever we see \( z \) enter \( \REset(A){i} \) (as we could wait until \( \setcol{A}{1} \) had settled) and simply copying \(  \setcol{A}{2}, \setcol{A}{3} \) into \( \setcol{Y_i}{1} \) and \( \setcol{Y_i}{2} \) respectively.  However, meeting requirements of the form \req{R}{j,e} forces us to make (non-computable) enumerations into all three columns of \( A \).

Specifically, an opponent building \( \REset(A){i} \) could try to duplicate the kind of argument we gave in \cref{sec:overview:req-R}.  If such an opponent could arrange for the approximation to \( A \Tplus \REset(A){i} \)  to flip-flop on some value \( z \) \( 4 \) times (waiting for new \( \Gfunc{i}{}, \Tfunc{} \) computations each time) then we would have no way to change \( Y_i \Tplus+ A \) to meet \req{P}{i}.  This situation is loosely depicted in \cref{fig:piggybacking-off-R} where the shaded region is meant to represent the region below \( l_s \) (and thus use and domain of \( \Tfunc{} \) and \( \Gfunc{i}{} \)) while the hatched region in stage \( t_4 \)  represents the use of the \( \Gfunc{i}{} \) computation which disagrees with \( Y_i \).

{
 \newcommand{\fillAllToLevel}[1]{
 \REAmatrixFillRect<fill=LightGray>(a){#1-1}{4-3}
  \REAmatrixFillRect<fill=LightGray>(w){#1-1}{4-1}
  \REAmatrixFillRect<fill=LightGray>(y){#1-1}{4-3}
 }

\newcommand{\divider}{\, \rule[.7cm]{1pt}{1.5cm} \,}
 \begin{figure}
    \centering
    \begin{tikzpicture}
        \REAmatrix(a)<{\( A \) }>{4}{3}
        \LineLabelCell[right][5ex]{4-3}{\( c \) }
        \REAmatrix(w)<{\( \REset(A){i} \) }>[right = 0cm of a]{4}{1}
        \LineLabelCell[above right][1.5ex]{4-1}{\( z \) }
        \node[right=.08cm of w,yshift=.5cm] (funcs) {\( \stackrel[{ \Gfunc{i}{A \Tplus \REset(A){i}} }]{ \Tfunc{Y_i} }{\leftrightarrows} \) } ;
         \REAmatrix(y)<{\( Y_i \) }>[right=.05cm of funcs,yshift=-.5cm]{4}{3}
         \LineLabelCell[right]{4-3}{\( \codeY<i>{z} \) }
         \fillAllToLevel{4}
        \node [below=1cm,xshift=1cm] {Stage \( t_{0} \)};
    \end{tikzpicture} \divider
      \begin{tikzpicture}
        \REAmatrix(a)<{\( A \) }>{4}{3}
        \REAmatrix(w)<{\( \REset(A){i} \) }>[right = 0cm of a]{4}{1}
        \REAmatrixEnum{4-1}
        % \node[right=.05cm of w] (funcs) {\( \stackrel[{ \Gfunc{i}{A \Tplus \REset(A){i}} }]{ \Tfunc{Y_i} }{\leftrightarrows} \) } ;
         \REAmatrix(y)<{\( Y_i \) }>[right=.05cm of w]{4}{3}
        \REAmatrixEnum{4-3}
         \fillAllToLevel{3}
        \node [below=1cm,xshift=1cm] {Stage \( t_{1} \)};
    \end{tikzpicture} \\
      \begin{tikzpicture}
        \REAmatrix(a)<{\( A \) }>{4}{3}
        \REAmatrixEnum{4-3}
        \REAmatrix(w)<{\( \REset(A){i} \) }>[right = 0cm of a]{4}{1}
        % \REAmatrixEnum{4-1}
        \REAmatrixEnum[0]{4-1}
        % \node[right=.05cm of w] (funcs) {\( \stackrel[{ \Gfunc{i}{A \Tplus \REset(A){i}} }]{ \Tfunc{Y_i} }{\leftrightarrows} \) } ;
         \REAmatrix(y)<{\( Y_i \) }>[right=.05cm of w]{4}{3}
        \REAmatrixEnum{3-2}
        \REAmatrixEnum[0]{4-3}
         \fillAllToLevel{2}
        \node [below=1cm,xshift=1cm] {Stage \( t_{2} \)};
    \end{tikzpicture} \divider % \parbox[b][2cm][c]{1cm}{\( \ldots \)} 
        \begin{tikzpicture}
        \REAmatrix(a)<{\( A \) }>{4}{3}
        \REAmatrixEnum[0]{4-3}
          \REAmatrixEnum{2-2}
        \REAmatrix(w)<{\( \REset(A){i} \) }>[right = 0cm of a]{4}{1}
        % \REAmatrixEnum{4-1}
        \REAmatrixEnum{4-1}
        % \node[right=.05cm of w] (funcs) {\( \stackrel[{ \Gfunc{i}{A \Tplus \REset(A){i}} }]{ \Tfunc{Y_i} }{\leftrightarrows} \) } ;
         \REAmatrix(y)<{\( Y_i \) }>[right=.05cm of w]{4}{3}
         \REAmatrixEnum{2-1}
        \REAmatrixEnum[0]{3-2}
        \REAmatrixEnum{4-3}
         \fillAllToLevel{1}
        \node [below=1cm,xshift=1cm] {Stage \( t_{3} \)};
    \end{tikzpicture} \divider
        \begin{tikzpicture}
        \REAmatrix(a)<{\( A \) }>{4}{3}
        \REAmatrixEnum{4-3}
          \REAmatrixEnum[0]{2-2}
          \REAmatrixEnum{1-1}
        \REAmatrix(w)<{\( \REset(A){i} \) }>[right = 0cm of a]{4}{1}
        % \REAmatrixEnum{4-1}
        \REAmatrixEnum[0]{4-1}
        % \node[right=.05cm of w] (funcs) {\( \stackrel[{ \Gfunc{i}{A \Tplus \REset(A){i}} }]{ \Tfunc{Y_i} }{\leftrightarrows} \) } ;
         \REAmatrix(y)<{\( Y_i \) }>[right=.05cm of w]{4}{3}
         \REAmatrixEnum{2-1}
        % \REAmatrixEnum[0]{3-2}
        \REAmatrixCircle{2-1}
        \REAmatrixEnum{4-3}        % \node [below=1cm,xshift=1cm] {Stage \( s_{-1} \)};
         \REAmatrixFillRect<pattern=crosshatch>(a){2-1}{4-3}
  \REAmatrixFillRect<pattern=crosshatch>(w){2-1}{4-1}
   \node [below=1cm,xshift=1cm] {Stage \( t_{4} \)};
    \end{tikzpicture} 
    % \begin{tikzpicture}
    % \node[yshift=1cm] {\( \ldots \) } ;
    % \end{tikzpicture}
   \caption{Piggybacking off \req*{R}{j,e}}
    \label{fig:piggybacking-off-R}
\end{figure}

}

For instance, our opponent might start by reserving some value \( z \) to be kept out of \( \REset(A){i} \).  By waiting until we've committed to the behavior of our functionals on the appropriate initial segments our opponent could force a change to \( Y_i \) by enumerating \( z \) into  \( \REset(A){i} \) with some large use and wait until some stage at which our functionals again witness (enough of) the equivalence \( A \Tplus \REset(A){i} \Tequiv Y_i \Tplus+ A \).  The proof of \cref{thm:extend-two-REA} from \cite{Cholak1994Iterated} shows that our opponent can ensure that sometimes we're forced to enumerate some \( y \) into \( \setcol{Y_i}{2} \) (to cancel a previous enumeration into \( \setcol{Y_i}{3} \)) in response to  some requirement \req{R}{j,e} enumerating \( c \) into  \( \setcol{A}{3} \).  That is, we can't avoid getting to stage \( t_2 \) as depicted in  \cref{fig:piggybacking-off-R}.

\subsection{Safety Via Agreement}
\label{sec:overview:agreement}

If \req{P}{i}  is of higher priority than \req{R}{j,e} we need to be able to meet \req{R}{j,e} in a way that doesn't allow an opponent to force an injury to \req{P}{i}. Our strategy will be to modify the way (the module responsible for meeting) \req{R}{j,e} operates to ensure that when \req{R}{j,e} enumerates \( b \) into \( \setcol{A}{2} \) we aren't forced to enumerate any element into \( \setcol{Y_i}{1} \), i.e., we avoid ever reaching the stage labeled \( t_3 \)  in \cref{fig:piggybacking-off-R}.  If we can arrange this, it removes our opponents \( 1 \) column advantage on us which will prevent \req{R}{j,e} from injuring \req{P}{i} (the final enumeration of \( a \) into \( \setcol{A}{1} \) will thus keep \( A \Tplus+  Y_i \) and \( A \Tplus \REset(A){i} \) in lock step).

The key idea here is that  \( b \)'s entry into \( \setcol{A}{2} \) can only force us to enumerate an element into \( \setcol{Y_i}{1} \) if there is some (small) value \( y \) with \( y \nin Y_{i,t_1 - 1} \) but  \( y \in Y_{i,t_3 -1} \).  In other words we will wait to enumerate \( b \) until the second (and first) column of \( Y_i \) are the same as they were immediately before \( c \) was enumerated. 

As the enumeration of \( c \) into \( \setcol{A}{3} \) at stage \( s_0 \)  will generally force a change in \( \setcol{Y_i}{2} \) and since we can't control \( \REset(A){i} \)  we can't ensure this condition is met for any particular value \( c \).  Instead, we make multiple attempts to meet \req{R}{j,e} with the first attempt enumerating \( c_0 \) into \( \setcol{A}{3} \), the second attempt \( c_1 \) and so on and argue that for some \(c =  c_k \) we either meet \req{R}{j,e} without canceling \( c \)  or that we eventually see the agreement needed to allow safe enumeration of \( b \) into \( \setcol{A}{2} \).  

Specifically, at stage \( s_{-1} \) we initialize (the module for) \req{R}{j,e} choosing \( c_0 \) large and setting \( c_k = c_0 + k \) (in the full construction \( \setcol{A}{3} \) will be partitioned between the various requirements to avoid collision).      At stages \( t_k =  s_k -1 \) we observe an active stage for \req{R}{j,e}  and, if we haven't yet found agreement (i.e. some \( c_n \) which we can now cancel without changing the first two columns of \textit{any} set \( Y_i \)), we respond at stage \( s_k \) by enumerating (only) \( c_k \) into \( \setcol{A}{3} \).  

There are two critical aspects of this strategy.  First, we ensure that \( c_{k+1} \) is small relative to the stage at which \( c_k \) is first enumerated so that each time we enumerate some \( c_k \) we are extending the use of \( \Gfunc{i}{} \).  Hence, commitments made about \( \Gfunc{i}{} \) after stage \( s_{-1} \) are canceled when we enumerate any \( c_k \) into \( \setcol{A}{3} \).     Second, the fact that we choose \( c_0 \) large means that if \( z \in \REset[s_{-1}](A){i} \) then no enumeration of \( c_k \) will cancel \( z \) (and thus no reason to change how \( Y_i \) codes \( z \)'s membership in \( \REset(A){i} \) ).  Note that it will also be the case that each time we enumerate some \( c_k \) we will remove\footnote{Remember, that we've used the fact that we will ensure \( A_s \subfun A \) infinitely often to allow us to set the `use' of a computation placing  \(z \in \REset(A){i} \) at stage \( s \) to \( l_s \).} any \( z \) enumerated into \( \REset(A){i} \) since stage \( s_{-1} \).      Thus, the only real constraint prior commitments impose on us while trying to produce the desired agreement is that we must remove \( \codeY[s_{-1}]<i>{z} \) from \( \setcol{Y_i}{3} \) \textit{if present} whenever we enumerate any \( c_k \) into \( \setcol{A}{3} \).

The main difficulty in the construction will be to argue that we can always arrange our enumerations into \( Y_i \) during the intervals \( (s_{k-1}, s_{k}) \) so that we eventually find some \( c_n \) we can cancel without changing the first two columns of any \( Y_i \).   Stated in terms of the following definition we need to show that we can find \( k',k \) such that  \( s_{k'} - 1 \Yagree s_k - 1 \).  That is,  there are active stages for \( \req{R}{j,e} \) (or there are only finitely many such stages)  \( s_{k'} - 1 < s_k -1 \) such that canceling \( c_{k'} \) by enumerating \( b_{k'} \)  at stage\footnote{Note that in full construction we assume that no elements enter any \( \REset(A){i} \) at stages \( s_k \) so we need not worry about changes between \( s_k -1 \) and \( s_k \).} \( s_k \)  (which requires us to return all \( Y_i \) to a state compatible with that at \( s_{k'} - 1 \)) won't require changes in the first two columns of \( Y_i \) (the third column is returned to it's earlier state for free by the enumeration of \( b_{k'} \)).

\begin{definition}\label{def:Y-agree}
We define 
\begin{align*}  
s \Yagree<i> t & \iffdef \setcol{Y_{i,s}}{\leq 2}\restr{l_s} \compat \setcol{Y_{i,t}}{\leq 2}\restr{l_t} \\
s \Yagree t & \iffdef \forall(i)\left(s \Yagree<i> t\right) \\
\end{align*}
We say that \( s, t \) \defword{agree} just when  \( s \Yagree t \).
\end{definition}

For future use we also adopt the following terminology (reflecting the fact that if \( t \) is accessible at stage \( s \) then there is some possible enumeration of elements into the sets \( Y_i \) that would allow some later stage \( s' > s \) to satisfy \( t \Yagree s' \)).

\begin{definition}\label{def:accessible-stage}

    A stage \( t \) is \textbf{accessible} at a stage \( s \geq t \) just if  for all \( i \) 
        \[ \setcol{Y_{i,s}}{\leq 1}\restr{l_s} \supfun \setcol{Y_{i,t}}{\leq 1}\restr{l_t} \]

    A stage is (\( i \)-\textbf{canceled}) \textbf{canceled} just if it isn't (\( i \)-accessible) accessible.
\end{definition}

We also adopt the definition below.  Note that active stages are the stages at which \module{R}{j,e} takes action in response to seeing a longer \( \recfnl{j}{} \)  back and forth computation on it's own accord rather than merely responding to some enumeration into \( \REset(A){i} \).

\begin{definition}\label{def:expansionary-stage}
    A stage \( s_m \) at which  \module{R}{j,e} enumerates an element into \( \setcol{A}{3} \) or \( \setcol{A}{2} \)  is a (\req{R}{j,e}) \textbf{active stage}.   A stage \( t_k \) is called a (\req{R}{j,e}) \textbf{preactive} stage just if \( t_{k} + 1 \) is an active stage.    
\end{definition}

\subsection{Simplified Agreement}
\label{sec:overview:simplified-agreement}

Let's consider how (the module for) \req{R}{j,e} might enact this strategy with respect to a single requirement \req{P}{i} assuming only a single value  \( z \) enters or leaves \( \REset(A){i} \).  On this assumption, we'll describe a winning strategy for \req{R}{j,e} that enumerates at most two elements \( c_0, c_1 \) into \( \setcol{A}{3} \).   This strategy is visualized in \cref{fig:safe-satisfying,fig:safe-satisfying-two}. Note that, while the figures depict \( c_0 \) and  \( \codeY[s_{-1}]<i>{z} \) as being on the same level but in reality \( \codeY[s_{-1}]<i>{z} \) is smaller than \( c_0 \) which is why we shade a larger region of \( Y_i \) (where the shaded region indicates the region below \( l_s \) at the given stage).

   {

\newcommand{\divider}{\, \rule[1cm]{1pt}{2cm} \,}
\newcommand{\mylabel}{\node [below=1.5cm,xshift=1cm]}
 \NewDocumentCommand{\fillAllToLevel}{O{#2}m}{
 \REAmatrixFillRect<fill=LightGray>(a){#2-1}{6-3}
  \REAmatrixFillRect<fill=LightGray>(w){#2-1}{6-1}
  \REAmatrixFillRect<fill=LightGray>(y){#1-1}{6-3}}

 \begin{figure}
    \centering
    \begin{tikzpicture}
        \REAmatrix(a)<{\( A \) }>{6}{3}
        \LineLabelCell[right][1ex]{6-3}{\( c_0 \) }
        \LineLabelCell[right][1ex]{5-3}{\( c_1 \) }
        % \LineLabelCell[right][1ex]{4-3}{\( c_2 \) }
        \REAmatrix(w)<{\( \REset(A){i} \) }>[right = .4cm of a]{6}{1}
         \REAmatrixEnum{6-1}
         \LineLabelCell[right][1ex]{6-1}{\( z \) }
         \REAmatrix(y)<{\( Y_i \) }>[right=.4cm of w]{6}{3}
         \REAmatrixEnum{6-3}
         % \LineLabelCell[right][1ex]{6-3}{\( \codeY[s_{-1}]<i>{z} \) }
         \fillAllToLevel[5]{6}
        \node [below=1.5cm,xshift=2cm] {\( s_{0} - 1 \)};
    \end{tikzpicture} \divider
        \begin{tikzpicture}
        \REAmatrix(a)<{\( A \) }>{6}{3}
        \REAmatrixEnum{6-3}
        \REAmatrix(w)<{\( \REset(A){i} \) }>[right = 0cm of a]{6}{1}
         \REAmatrixEnum[0]{6-1}
         \REAmatrix(y)<{\( Y_i \) }>[right=0cm of w]{6}{3}
         \REAmatrixEnum{5-2}
         \REAmatrixCircle{5-2}
         \REAmatrixEnum[0]{6-3}
         \fillAllToLevel[4]{5}
        \mylabel {Stage \( s_{0} \) and \( s_{1} - 1 \)};
    \end{tikzpicture} \\
    \begin{tikzpicture}
        \REAmatrix(a)<{\( A \) }>{6}{3}
        \REAmatrixEnum{6-3}
        \REAmatrixEnum{5-3}
        \REAmatrix(w)<{\( \REset(A){i} \) }>[right = 0cm of a]{6}{1}
         % \REAmatrixEnum[0]{6-1}
         \REAmatrix(y)<{\( Y_i \) }>[right=0cm of w]{6}{3}
         \REAmatrixEnum{5-2}
          
         % \REAmatrixEnum[0]{6-3}
         \fillAllToLevel[3]{4}
        \mylabel {Stage \( s_{1} \) };
    \end{tikzpicture} \divider
        \begin{tikzpicture}
        \REAmatrix(a)<{\( A \) }>{6}{3}
        \REAmatrixEnum{6-3}
        \REAmatrixEnum{5-3}
        \REAmatrix(w)<{\( \REset(A){i} \) }>[right = 0cm of a]{6}{1}
         \REAmatrixEnum{6-1}
         \REAmatrix(y)<{\( Y_i \) }>[right=0cm of w]{6}{3}
         \REAmatrixEnum{5-2}
          
         \REAmatrixEnum{6-3}
         \fillAllToLevel[2]{3}
        \mylabel {Stage \( s_{2} - 1 \Yagree<i> s_1 -1  \) };
    \end{tikzpicture} \divider
            \begin{tikzpicture}
        \REAmatrix(a)<{\( A \) }>{6}{3}
        \REAmatrixEnum{6-3}
        \REAmatrixEnum[0]{5-3}
        \REAmatrixEnum{4-2}
        \REAmatrix(w)<{\( \REset(A){i} \) }>[right = 0cm of a]{6}{1}
         \REAmatrixEnum[0]{6-1}
         \REAmatrix(y)<{\( Y_i \) }>[right=0cm of w]{6}{3}
         \REAmatrixEnum{5-2}
         \REAmatrixEnum[0]{6-3}
         \fillAllToLevel[1]{2}
        \mylabel {Enumerating \( b_1 \) };
    \end{tikzpicture}
 
   \caption{Satisfying \req{R}{j,e} respecting \req{P}{i} on \( z \)}
    \label{fig:safe-satisfying}
\end{figure}

}

  If\( z \) doesn't enter \( \REset(A){i} \) in the interval \( [s_{-1},  s_0) \) when we enumerate \( c_0 \) we get \( s_{0} -1 \Yagree<i> s_1 - 1 \) for free since there was never any need to cancel any element in \( \setcol{Y_i}{3} \) (this case isn't depicted).

{

\newcommand{\divider}{\, \rule[1cm]{1pt}{2cm} \,}
\newcommand{\mylabel}{\node [below=1.5cm,xshift=1cm]}
 \NewDocumentCommand{\fillAllToLevel}{O{#2}m}{
 \REAmatrixFillRect<fill=LightGray>(a){#2-1}{6-3}
  \REAmatrixFillRect<fill=LightGray>(w){#2-1}{6-1}
  \REAmatrixFillRect<fill=LightGray>(y){#1-1}{6-3}}

 \begin{figure}
    \centering
        \begin{tikzpicture}
        \REAmatrix(a)<{\( A \) }>{6}{3}
        \REAmatrixEnum{6-3}
        \REAmatrix(w)<{\( \REset(A){i} \) }>[right = 0cm of a]{6}{1}
         \REAmatrixEnum[0]{6-1}
         \REAmatrix(y)<{\( Y_i \) }>[right=0cm of w]{6}{3}
         \REAmatrixEnum{5-2}
         % \REAmatrixCircle{5-2}
         \REAmatrixEnum[0]{6-3}
         \fillAllToLevel[4]{5}
        \mylabel {Stage \( s_0 \)};
    \end{tikzpicture} \divider
            \begin{tikzpicture}
        \REAmatrix(a)<{\( A \) }>{6}{3}
        \REAmatrixEnum{6-3}
        \REAmatrix(w)<{\( \REset(A){i} \) }>[right = 0cm of a]{6}{1}
         \REAmatrixEnum{6-1}
         \REAmatrix(y)<{\( Y_i \) }>[right=0cm of w]{6}{3}
         \REAmatrixEnum[0]{5-2}
      \REAmatrixEnum{4-1}
         \REAmatrixEnum{6-3}
         \fillAllToLevel[4]{5}
        \mylabel {Stage \( s'_0 \) and \( s_1 - 1 \Yagree<i> s_0 - 1 \) };
    \end{tikzpicture} \divider
    \begin{tikzpicture}
        \REAmatrix(a)<{\( A \) }>{6}{3}
         \REAmatrixEnum[0]{6-3}
         \REAmatrixEnum{5-2}
        \REAmatrix(w)<{\( \REset(A){i} \) }>[right = 0cm of a]{6}{1}
         \REAmatrixEnum{6-1}
         \REAmatrix(y)<{\( Y_i \) }>[right=0cm of w]{6}{3}
         \REAmatrixEnum{6-3}
          \REAmatrixEnum{4-1}
         \fillAllToLevel[3]{4}
        \mylabel {Enumerating \( b_0 \) };
    \end{tikzpicture}
 
   \caption{Alternate behavior after stage \( s_0 \)}
    \label{fig:safe-satisfying-two}
\end{figure}
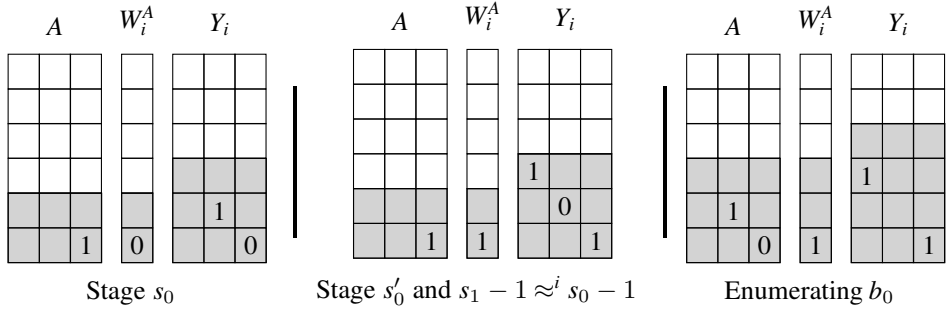

}

%\Cref{fig:safe-satisfying} goes on to illustrate that this now lets us cancel \( c_0 \).  
In \cref{fig:safe-satisfying} we see that \( z \) enters \( \REset(A){i} \) by stage \( s_{0} - 1 \) forcing us to cancel \( \codeY[s_{-1}]<i>{z} \) at stage \( s_0 \).  Note that when \( z \) first enters \( \REset(A){i} \) we become free to code that entry into \( Y_i \) since the enumeration changes \( \REset(A){i} \) on the use of  \( \Gfunc{i}{A \Tplus \REset(A){i}}(\codeY[s_{-1}]<i>{z}) \).  However, in the figure \( z \) isn't enumerated into \( \REset(A){i} \) during \( [s_0, s_1) \).  We've circled the value which blocks agreement at stage \( s_1 - 1 \) with the earlier stage \( s_0 -1 \) .  But now observe that at stage \( s_2 - 1 \)   agreement is achieved for free since we had no need to cancel any elements at stage \( s_1 \) (with \( c_1 \) playing the role of \( c \)).  As we will see later, even if we'd had other intervening stages as long as they didn't enumerate any elements into \( \setcol{Y_i}{1} \)  small at stage \( s_1 - 1 \) we could cancel their contributions.    Note that when we enumerate \( b_1 \) into \( \setcol{A}{2} \) it automatically removes \( \codeY[s_{-1}]<i>{z} \) from \( \setcol{Y_i}{3} \) as this was enumerated during the interval \( (s_1, s_2) \) (remember \( A \Tplus+ Y_i \) is \( \REA[3] \) not \( Y_i \)).

In \cref{fig:safe-satisfying-two} we depict the case where \( z \) is enumerated into \( \REset(A){i} \) at some stage \( s'_0 \in (s_0, s_1) \).  At stage \( s'_0 \) (we can't wait until \( s_1 -1 \) ) we are faced with a choice.  We could simply enumerate  \( \codeY[s_{-1}]<i>{z} \) into \( \setcol{Y_i}{3} \) again but then the value we enumerated into \( \setcol{Y_i}{2} \) at stage \( s_0 \) would prevent us from achieving agreement between stages \( s_0 -1 \) and \( s_1 - 1 \).  So instead we choose to cancel the value enumerated into  \( \setcol{Y_i}{2} \) and restore the state of \( Y_i \) at stage \( s_0 - 1 \) producing agreement.  

So far our examples have always shown \( z \) entering \( \REset(A){i} \) before the enumeration of \( c_0 \) as this is the most challenging case to achieve agreement.  However, it's useful to consider the case to illustrate two important points.  First, to demonstrate why it's only agreement in the first two columns of \( Y_i \) which matters not agreement about \( z \in \REset(A){i}  \)  and why later enumeration into \( \REset(A){i} \) doesn't threaten agreement once it's been achieved.  This situation is depicted in \cref{fig:z-removing-victory}.

   {

\newcommand{\divider}{\, \rule[1cm]{1pt}{2cm} \,}
\newcommand{\mylabel}{\node [below=1.5cm,xshift=1cm]}
 \NewDocumentCommand{\fillAllToLevel}{O{#2}m}{
 \REAmatrixFillRect<fill=LightGray>(a){#2-1}{6-3}
  \REAmatrixFillRect<fill=LightGray>(w){#2-1}{6-1}
  \REAmatrixFillRect<fill=LightGray>(y){#1-1}{6-3}}

 \begin{figure}
    \centering
    \begin{tikzpicture}
        \REAmatrix(a)<{\( A \) }>{6}{3}
        % \LineLabelCell[right][1ex]{6-3}{\( c_0 \) }
        % \LineLabelCell[right][1ex]{5-3}{\( c_1 \) }
        % \LineLabelCell[right][1ex]{4-3}{\( c_2 \) }
        \REAmatrix(w)<{\( \REset(A){i} \) }>[right = 0cm of a]{6}{1}
         % \REAmatrixEnum{6-1}
         % \LineLabelCell[right][1ex]{6-1}{\( z \) }
         \REAmatrix(y)<{\( Y_i \) }>[right=0cm of w]{6}{3}
         % \REAmatrixEnum{6-3}
         % \LineLabelCell[right][1ex]{6-3}{\( \codeY[s_{-1}]<i>{z} \) }
         \fillAllToLevel[5]{6}
         \mylabel {\( s_{0} - 1 \)};
    \end{tikzpicture} \divider
        \begin{tikzpicture}
        \REAmatrix(a)<{\( A \) }>{6}{3}
        \REAmatrixEnum{6-3}
        \REAmatrix(w)<{\( \REset(A){i} \) }>[right = 0cm of a]{6}{1}
         % \REAmatrixEnum[0]{6-1}
         \REAmatrix(y)<{\( Y_i \) }>[right=0cm of w]{6}{3}
         % \REAmatrixEnum{5-2}
         % \REAmatrixCircle{5-2}
         % \REAmatrixEnum[0]{6-3}
         \fillAllToLevel[4]{5}
        \mylabel {Stage \( s_{0} \Yagree<i> s_0 - 1  \)};
    \end{tikzpicture} \divider
                \begin{tikzpicture}
        \REAmatrix(a)<{\( A \) }>{6}{3}
        \REAmatrixEnum{6-3}
        \REAmatrix(w)<{\( \REset(A){i} \) }>[right = 0cm of a]{6}{1}
         \REAmatrixEnum{6-1}
         \REAmatrix(y)<{\( Y_i \) }>[right=0cm of w]{6}{3}
         % \REAmatrixEnum[0]{5-2}
      % \REAmatrixEnum{4-1}
         \REAmatrixEnum{6-3}
         \fillAllToLevel[4]{5}
        \mylabel {Stage \( s_1 - 1 \Yagree<i> s_0 - 1 \) };
    \end{tikzpicture} \\
        \begin{tikzpicture}
        \REAmatrix(a)<{\( A \) }>{6}{3}
         \REAmatrixEnum[0]{6-3}
         \REAmatrixEnum{5-2}
        \REAmatrix(w)<{\( \REset(A){i} \) }>[right = 0cm of a]{6}{1}
         \REAmatrixEnum[0]{6-1}
         \REAmatrix(y)<{\( Y_i \) }>[right=0cm of w]{6}{3}
         \REAmatrixEnum[0]{6-3}
          % \REAmatrixEnum{4-1}
         \fillAllToLevel[3]{4}
        \mylabel {Enumerating \( b_0 \) at \( s_1 \)  };
    \end{tikzpicture} \divider
        \begin{tikzpicture}
        \REAmatrix(a)<{\( A \) }>{6}{3}
         \REAmatrixEnum[0]{6-3}
         \REAmatrixEnum{5-2}
        \REAmatrix(w)<{\( \REset(A){i} \) }>[right = 0cm of a]{6}{1}
         \REAmatrixEnum[1]{6-1}
         \REAmatrix(y)<{\( Y_i \) }>[right=0cm of w]{6}{3}
         \REAmatrixEnum[1]{6-3}
          % \REAmatrixEnum{4-1}
         \fillAllToLevel[2]{3}
        \mylabel {\( z \) may still enter \( \REset(A){i} \)   };
    \end{tikzpicture}  \divider
 \begin{tikzpicture}
        \REAmatrix(a)<{\( A \) }>{6}{3}
        \REAmatrixEnum{4-1}
         \REAmatrixEnum[0]{5-2}
        \REAmatrixEnum{6-3}
        \REAmatrix(w)<{\( \REset(A){i} \) }>[right = 0cm of a]{6}{1}
         \REAmatrixEnum{6-1}
         \REAmatrix(y)<{\( Y_i \) }>[right=0cm of w]{6}{3}
         % \REAmatrixEnum[0]{5-2}
      % \REAmatrixEnum{4-1}
         \REAmatrixEnum[0]{6-3}
         \fillAllToLevel[1]{2}
        \mylabel {Enumerating \( a \) to cancel \( b_0 \)  };
    \end{tikzpicture}

   \caption{Satisfying \req{R}{j,e}, \req{P}{i} when \( b_0 \) removes \( z \) from \( \REset(A){i} \)   }
    \label{fig:z-removing-victory}
\end{figure}

}

In this scenario \( z \) doesn't enter \( \REset(A){i} \) before, at \( s_0 \), we enumerate \( c_0 \) into \( \setcol{A}{3} \).  As such, we don't enumerate \( \codeY[s_{-1}]<i>{z} = \codeY[s_{0}]<i>{z} \) before \( s_0 \) either so it need not be canceled at stage \( s_0 \).  Hence, at stage \( s_0 \) we already have agreement with stage \( s_0 -1 \), i.e.,  \( s_0 \Yagree<i> s_0 -1 \).  However, we may have to wait to see a \req{R}{j,e} expansionary stage and during that time we may see \( z \) enter \( \REset(A){i} \).  But, note that, once we've achieved agreement we can always respond to such an entry by enumerating \( \codeY<i>{z} \) (labeled above as \( \codeY[s_{-1}]<i>{z} \)) into \( \setcol{Y_i}{3} \) without jeopardizing that agreement.  Thus, at stage \( s_1 -1 \) when we finally see a \req{R}{j,e} expansionary stage we have \( s_1 -1 \Yagree<i> s_0 -1 \).  

Now, remember, that it's not actually \( Y_i \) that we've committed to building as a \REA[3] set but  \( A \Tplus+ Y_i \) (the columwise sum).  Hence, the enumeration of \( b_0 \) into \( \setcol{A}{2} \) doesn't just cancel \( c_0 \) but also  \( \codeY[s_{-1}]<i>{z} \) (in \cref{fig:safe-satisfying-two} we glossed over need to reenumerate this element when \( b_0 \) was enumerated).   Thus, at \( s_1 \) we comply with the commitments made for \req{P}{i} at \( s_0 -1 \).  Of course, \( z \) could still decide to finally enter \( \REset(A){i} \) at this point but, since \( z \) wasn't in  \( \REset(A){i} \)  at stage \( s_0 -1 \), we are no longer committed to keeping \( \codeY[s_{-1}]<i>{z} \) out of \( \setcol{Y_i}{3} \) allowing us to preserve the computation required by \req{P}{i}.  Finally, if we enumerate \( a \), we again rely on the fact that it's  \( A \Tplus+ Y_i \) that's \REA[3] to allow \( a \) to cancel both  \( \codeY[s_{-1}]<i>{z} \) and \( b_0 \) returning both sides to stages compatible with those they had at stage \( s_1 -1 \).  Note that if, instead, \( z \) had been out of \( \REset(A){i} \) at stage \( s_1 -1 \)  and then entered \( \REset(A){i} \) after the enumeration of \( a \) we could respond by enumerating \( \codeY[s_{-1}]<i>{z} \) just as we did when \( z \) entered  \( \REset(A){i} \) after the enumeration of \( b_0 \) but before \( a \).

\subsection{General Agreement}
\label{sec:overview:general-agreement}

In \cref{sec:overview:simplified-agreement} we demonstrated how to meet \( \req{R}{j,e} \) without injuring a single requirement \req{P}{i} on the assumption that only a single value \( z \) ever enters or leaves \( \REset(A){i} \).  However, in the general case we must meet \( \req{R}{j,e} \) without injuring any of the finitely many higher priority requirements \( \req{P}{i_0}, \req{P}{i_1}, \req{P}{i_2}, \ldots \).  We must also accommodate arbitrarily many elements entering and leaving the sets \( \REset(A){i_n} \).

\subsubsection{Multiple elements entering some \( \REset(A){i_0} \)}

First we explain the relatively easy task of dealing with multiple elements entering and leaving the sets \( \REset(A){i_n} \).  By the discussion in \cref{sec:overview:agreement} the enumeration of \( c_k \) cancels any commitments we've made about \( \Gfunc{i}{} \) after stage \( s_{-1} \).  Hence, if \( z \geq s_{-1} \) then \cref{def:funcs} means that at \( s_{-1} \) we've yet to make any commitment about \( z \) (i.e. about the value of \( \Gfunc{i}{} \) on  \( \codeY[t]<i>{z} \)).  This leaves us with finitely many values we must be concerned with entering or leaving any \( \REset(A){i} \) when trying to achieve agreement.  Note that if we ever see \( z < z' \) enter \( \REset(A){i} \) after stage \( s_{-1} \) this cancels any commitments we've made regarding \( z' \) allowing us to freely code changes in \( \REset(A){i}(z') \) by enumerating elements into \( \setcol{Y_i}{3} \).  Morally speaking, this means it's only the least value of \( z \) entering \( \REset(A){i} \) in the interval \( [s_k, s_{k+1} -1] \) which matters so we need only tweak the strategy above to deal with multiple values entering/leaving some \( \REset(A){i} \).    

 More specifically, we eventually find \( k' < k \) such that  the least  \( z \nin \REset[s_{-1}](A){i} \) which is in some \( \REset[s_{k''}-1](A){i}, k' \leq k'' < k \) also is in \( \REset[s_{k}-1](A){i} \) or eventually all \( z \nin \REset[s_{-1}](A){i} \) aren't in \( \REset[s_{k}-1](A){i}  \).  In the later case we have \( s_k - 1 \Yagree<i> s_{k+1} -1 \)  as we are in the same situation depicted in \cref{fig:safe-satisfying}.  In the former case, as in \cref{fig:safe-satisfying-two}, we ensure \( s_{k'} - 1 \Yagree<i> s_{k} -1 \) by making enumerations into \( \setcol{Y_i}{1} \) when we see a new least \( z \nin  \REset[s_{-1}](A){i} \) enter \( \REset(A){i} \)  that cancels all changes since the last stage \( s_{k''} - 1 \) with \( z \in  \REset[s_{k''}-1](A){i} \) and making whatever corrections needed to code the state of \( z' > z \) by enumeration into \( \setcol{Y_i}{3} \).  Note that no problem is introduced if, in response to the entry of \( z'  \) into \( \REset(A){i} \) at stage \( s \in (s_{k- 1}, s_k) \)  we make an enumeration into \( \setcol{Y_i}{1} \)  to cancel all changes between \( s_{k'} - 1 \) and \( s \).  Such enumerations may prevent ever achieving agreement with some \( s_{k''} - 1, k' < k'' < k \) (as we can never remove an element from \( \setcol{Y_i}{1} \)) but that's not a problem since the canceled stages are ones in which no \( z'' < z \) enters \( \REset(A){i} \).

  % even if we later see some \( z < z' \) enter \( \REset(A){i} \) since no \( s_{k''} - 1, k' < k'' < k \) (which this enumeration blocks future agreement with) as the fact that \( z \in \REset[s_{k}-1](A){i}  \)  means that we don't rule out any potential pairs \( s_k -1, s_{k'} -1 \) which might satisfy the former property.    
\subsubsection{Multiple \req{P}{i} requirements}

 {

\newcommand{\divider}{\, \rule[1cm]{1pt}{2cm} \,}
\newcommand{\mylabel}{\node [below=1.5cm,xshift=2cm]}
 \NewDocumentCommand{\fillAllToLevel}{O{#2}m}{
 \REAmatrixFillRect<fill=LightGray>(a){#2-1}{6-3}
  \REAmatrixFillRect<fill=LightGray>(w1){#2-1}{6-1}
  \REAmatrixFillRect<fill=LightGray>(y1){#1-1}{6-3}
    \REAmatrixFillRect<fill=LightGray>(w0){#2-1}{6-1}
  \REAmatrixFillRect<fill=LightGray>(y0){#1-1}{6-3}
  }

 \begin{figure}
    \centering
        \fbox{\fbox{\begin{tikzpicture}
        \REAmatrix(a)<{\( A \) }>{6}{3}
        \REAmatrix(w0)<{\( \REset(A){i_0} \) }>[right = 0cm of a]{6}{1}
        \REAmatrixEnum{6-1}
         \REAmatrix(w1)<{\( \REset(A){i_1} \) }>[right = 0cm of w0]{6}{1}
          \REAmatrixEnum{6-1}
         \REAmatrix(y0)<{\( Y_{i_0} \) }>[right=0cm of w1]{6}{3}
         \REAmatrixEnum{6-3}
         \REAmatrix(y1)<{\( Y_{i_1} \) }>[right=0cm of y0]{6}{3}
         \REAmatrixEnum{6-3}
         \fillAllToLevel[5]{6}
        \mylabel {Stage \( s_0 - 1 \), \( z \in \REset(A){i_0} \) first. };
    \end{tikzpicture}}} \divider
    \begin{tikzpicture}
        \REAmatrix(a)<{\( A \) }>{6}{3}
        \REAmatrixEnum{6-3}
        \REAmatrix(w0)<{\( \REset(A){i_0} \) }>[right = 0cm of a]{6}{1}
        \REAmatrixEnum{6-1}
         \REAmatrix(w1)<{\( \REset(A){i_1} \) }>[right = 0cm of w0]{6}{1}
         \REAmatrixEnum{6-1}
         \REAmatrix(y0)<{\( Y_{i_0} \) }>[right=0cm of w1]{6}{3}
         \REAmatrixEnum{6-3}
         \REAmatrixEnum{5-2}
         \REAmatrix(y1)<{\( Y_{i_1} \) }>[right=0cm of y0]{6}{3}
        \REAmatrixEnum{6-3}
         \REAmatrixEnum{5-2}
         \fillAllToLevel[4]{5}
        \mylabel {Stage \( s_1 - 1 \), \( z \in \REset(A){i_1} \) first. };
    \end{tikzpicture} \\
        \begin{tikzpicture}
        \REAmatrix(a)<{\( A \) }>{6}{3}
        \REAmatrixEnum{6-3}
         \REAmatrixEnum{5-3}
        \REAmatrix(w0)<{\( \REset(A){i_0} \) }>[right = 0cm of a]{6}{1}
        % \REAmatrixEnum[0]{6-1}
         \REAmatrix(w1)<{\( \REset(A){i_1} \) }>[right = 0cm of w0]{6}{1}
         \REAmatrixEnum{6-1}
         \REAmatrix(y0)<{\( Y_{i_0} \) }>[right=0cm of w1]{6}{3}
         \REAmatrixEnum[0]{6-3}
         \REAmatrixEnum{5-2}
         \REAmatrixEnum{4-2}
         \REAmatrix(y1)<{\( Y_{i_1} \) }>[right=0cm of y0]{6}{3}
        \REAmatrixEnum{3-1}
         \REAmatrixEnum{5-2}
         \REAmatrixEnum[0]{4-2}
          \REAmatrixEnum{6-3}
         \fillAllToLevel[3]{4}
        \mylabel {Stage \( s_2 - 1 \), \( z \in \REset(A){i_1} \) first. };
    \end{tikzpicture} \divider
    \fbox{\begin{tikzpicture}
        \REAmatrix(a)<{\( A \) }>{6}{3}
        \REAmatrixEnum{6-3}
         \REAmatrixEnum{5-3}
         \REAmatrixEnum{4-3}
        \REAmatrix(w0)<{\( \REset(A){i_0} \) }>[right = 0cm of a]{6}{1}
        \REAmatrixEnum{6-1}
         \REAmatrix(w1)<{\( \REset(A){i_1} \) }>[right = 0cm of w0]{6}{1}
         % \REAmatrixEnum{6-1}
         \REAmatrix(y0)<{\( Y_{i_0} \) }>[right=0cm of w1]{6}{3}
         \REAmatrixEnum{6-3}
         \REAmatrixEnum[0]{5-2}
         \REAmatrixEnum[0]{4-2}
          % \REAmatrixEnum[0]{3-2}
         \REAmatrixEnum{4-1}
         \REAmatrix(y1)<{\( Y_{i_1} \) }>[right=0cm of y0]{6}{3}
        \REAmatrixEnum[0]{6-3}
         \REAmatrixEnum{3-1}
         \REAmatrixEnum{5-2}
         % \REAmatrixEnum{4-2}
          \REAmatrixEnum{3-2}
         \fillAllToLevel[2]{3}
        \mylabel {Stage \( s_3 - 1 \), \( z \in \REset(A){i_0} \) first. };
    \end{tikzpicture}} \\ { %\renewcommand{\mylabel}{\node [below=1.5cm,xshift=1cm]}%
     \RenewDocumentCommand{\fillAllToLevel}{O{#2}m}{
 \REAmatrixFillRect<fill=LightGray>(a){#2-1}{7-3}
  \REAmatrixFillRect<fill=LightGray>(w1){#2-1}{7-1}
  \REAmatrixFillRect<fill=LightGray>(y1){#1-1}{7-3}
    \REAmatrixFillRect<fill=LightGray>(w0){#2-1}{7-1}
  \REAmatrixFillRect<fill=LightGray>(y0){#1-1}{7-3}
  }%
    \begin{tikzpicture}
        \REAmatrix(a)<{\( A \) }>{7}{3}
        \REAmatrixEnum{7-3}
         \REAmatrixEnum{6-3}
         \REAmatrixEnum{5-3}
         \REAmatrixEnum{4-3}
        \REAmatrix(w0)<{\( \REset(A){i_0} \) }>[right = 0cm of a]{7}{1}
        \REAmatrixEnum{7-1}
         \REAmatrix(w1)<{\( \REset(A){i_1} \) }>[right = 0cm of w0]{7}{1}
         \REAmatrixEnum{7-1}
         \REAmatrix(y0)<{\( Y_{i_0} \) }>[right=0cm of w1]{7}{3}
         \REAmatrixEnum{7-3}
         % \REAmatrixEnum[0]{5-2}
         % \REAmatrixEnum[0]{4-2}
         %  \REAmatrixEnum[0]{3-2}
         \REAmatrixEnum{3-2}
         \REAmatrixEnum{5-1}
         % \REAmatrixEnum{2-2}
         \REAmatrix(y1)<{\( Y_{i_1} \) }>[right=0cm of y0]{7}{3}
        \REAmatrixEnum{7-3}
         \REAmatrixEnum{6-2}
         % \REAmatrixEnum{4-2}
          \REAmatrixEnum{4-1}
          \REAmatrixEnum{4-2}
          % \REAmatrixEnum{3-2}
         \fillAllToLevel[2]{3}
        \node [below=1.7cm,xshift=2cm]  {Stage \( s_4 - 1  \), \( z \in \REset(A){i_1} \) first. };
    \end{tikzpicture} \divider \hspace{.15em}
        \fbox{\begin{tikzpicture}
        \REAmatrix(a)<{\( A \) }>{7}{3}
        \REAmatrixEnum{7-3}
         \REAmatrixEnum{6-3}
         \REAmatrixEnum{5-3}
         \REAmatrixEnum{4-3}
         \REAmatrixEnum{3-3}
        \REAmatrix(w0)<{\( \REset(A){i_0} \) }>[right = 0cm of a]{7}{1}
        \REAmatrixEnum{7-1}
         \REAmatrix(w1)<{\( \REset(A){i_1} \) }>[right = 0cm of w0]{7}{1}
         \REAmatrixEnum{7-1}
         \REAmatrix(y0)<{\( Y_{i_0} \) }>[right=0cm of w1]{7}{3}
         \REAmatrixEnum{7-3}
         % \REAmatrixEnum[0]{5-2}
         % \REAmatrixEnum[0]{4-2}
         %  \REAmatrixEnum[0]{3-2}
         \REAmatrixEnum[0]{3-2}
         \REAmatrixEnum{5-1}
         \REAmatrixEnum{2-1}
         \REAmatrix(y1)<{\( Y_{i_1} \) }>[right=0cm of y0]{7}{3}
        \REAmatrixEnum{7-3}
         \REAmatrixEnum{6-2}
          \REAmatrixEnum{4-1}
         \REAmatrixEnum{2-2}
          \REAmatrixEnum{4-2}
        \fillAllToLevel[1]{2}
        \node [below=1.7cm,xshift=2.1cm] {Stage \( s_5 - 1 \Yagree s_3 -1  \), \( z \in \REset(A){i_0} \) first. };
    \end{tikzpicture}}
}
       \caption{Agreement Respecting \( \req{P}{i_0} \) and \( \req{P}{i_1} \) }
    \label{fig:two-Pi-requirement-interaction}
\end{figure}
}

 {

\newcommand{\divider}{\, \rule[1cm]{1pt}{2cm} \,}
\newcommand{\mylabel}{\node [below=1.5cm,xshift=2cm]}
 \NewDocumentCommand{\fillAllToLevel}{O{#2}m}{
 \REAmatrixFillRect<fill=LightGray>(a){#2-1}{6-3}
  \REAmatrixFillRect<fill=LightGray>(w1){#2-1}{6-1}
  \REAmatrixFillRect<fill=LightGray>(y1){#1-1}{6-3}
    \REAmatrixFillRect<fill=LightGray>(w0){#2-1}{6-1}
  \REAmatrixFillRect<fill=LightGray>(y0){#1-1}{6-3}
  }

 \begin{figure}
    \centering
\begin{tikzpicture}
        \REAmatrix(a)<{\( A \) }>{6}{3}
        \REAmatrixEnum{6-3}
         \REAmatrixEnum{5-3}
         \REAmatrixEnum{4-3}
        \REAmatrix(w0)<{\( \REset(A){i_0} \) }>[right = 0cm of a]{6}{1}
        \REAmatrixEnum{6-1}
         \REAmatrix(w1)<{\( \REset(A){i_1} \) }>[right = 0cm of w0]{6}{1}
         \REAmatrixEnum{6-1}
         \REAmatrix(y0)<{\( Y_{i_0} \) }>[right=0cm of w1]{6}{3}
         \REAmatrixEnum{6-3}
         \REAmatrixEnum{5-2}
         \REAmatrixEnum{4-2}
          % \REAmatrixEnum[0]{3-2}
         \REAmatrix(y1)<{\( Y_{i_1} \) }>[right=0cm of y0]{6}{3}
        \REAmatrixEnum{2-1}
         \REAmatrixEnum{3-1}
         \REAmatrixEnum{5-2}
          \REAmatrixEnum{6-3}
          \REAmatrixEnum[0]{3-2}
         \fillAllToLevel[2]{3}
        \mylabel {Stage \( s_3 - 1 \Yagree s_2 -1 \), \( z \in \REset(A){i_1} \) first. };
    \end{tikzpicture} \divider
        \fbox{\fbox{\begin{tikzpicture}
        \REAmatrix(a)<{\( A \) }>{6}{3}
        \REAmatrixEnum{6-3}
         \REAmatrixEnum{5-3}
         \REAmatrixEnum{4-3}
        \REAmatrix(w0)<{\( \REset(A){i_0} \) }>[right = 0cm of a]{6}{1}
        \REAmatrixEnum{6-1}
         \REAmatrix(w1)<{\( \REset(A){i_1} \) }>[right = 0cm of w0]{6}{1}
         \REAmatrixEnum{6-1}
         \REAmatrix(y0)<{\( Y_{i_0} \) }>[right=0cm of w1]{6}{3}
         \REAmatrixEnum{6-3}
         \REAmatrixEnum[0]{5-2}
         \REAmatrixEnum[0]{4-2}
          % \REAmatrixEnum[0]{3-2}
         \REAmatrixEnum{4-1}
         \REAmatrix(y1)<{\( Y_{i_1} \) }>[right=0cm of y0]{6}{3}
        \REAmatrixEnum{6-3}
         \REAmatrixEnum{3-1}
         \REAmatrixEnum[0]{5-2}
         \REAmatrixEnum{4-1}
          \REAmatrixEnum{3-2}
         \fillAllToLevel[2]{3}
        \mylabel {Stage \( s_3 - 1 \), \( z \in \REset(A){i_0} \) first. };
    \end{tikzpicture}}}

       \caption{Alternative agreements respecting \( \req{P}{i_0} \) and \( \req{P}{i_1} \) }
    \label{fig:alt-two-Pi-requirement-interaction}
\end{figure}
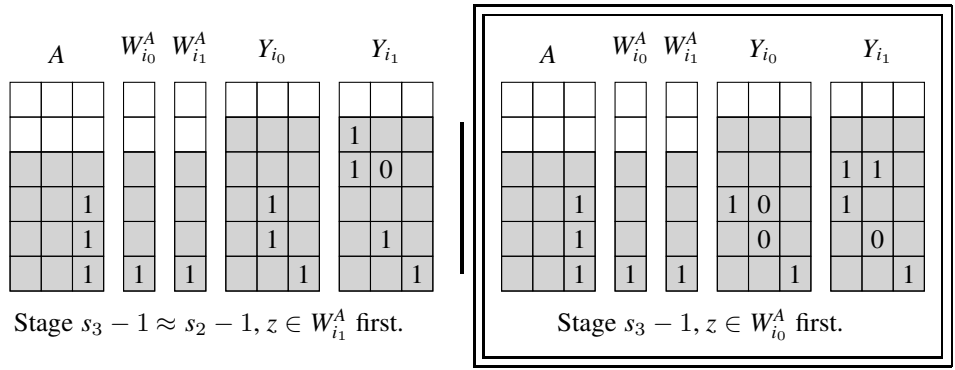
}

This leaves us with the harder problem of dealing with multiple requirements \( \req{P}{0}, \req{P}{1}, \ldots, \req{P}{\pair{j}{e}} \) at the same time.    

Unlike the case with multiple values entering a single set, \( \REset(A){i} \) here entry into \( \REset(A){i} \) doesn't cancel any commitments made about \( \Gfunc{i'}{}, i' \neq i\).  Thus, when we see \( z \) enter \( \REset(A){i} \) we must immediately decide how to code this change in \( Y_i \) (by enumerating into \( \setcol{Y_i}{3} \) or  cancelling a removal from it by enumerating into \( \setcol{Y_i}{1} \))  without yet knowing what elements may later enter some \( \REset(A){i'} \).  We now describe how we handle this problem under the simplifying assumption that there is only a single value \( z \) which may enter any set \( \REset(A){i} \) and \( z \nin \REset[s_{-1}](A){i} \)\footnote{If \( z \in \REset[s_{-1}](A){i} \) then \( z \) won't leave \( \REset(A){i} \) as a result of any enumeration we make in trying to meet \req{R}{j,e} ensuring that coding the status of \( z \) isn't in tension with our attempt to meet \req{R}{j,e}.} where \(i \leq \pair{j}{e} \).   As there is only a single value of \( z \) in play we will only ever need to make use of a single location  \( \codeY<i>{z} \) to code whether or not \( z \in \REset(A){i}  \) (i.e. if we see \( z \) leave \( \REset(A){i}   \) we always remove  \( \codeY<i>{z} \) from \( \setcol{Y_i}{3} \)) and WLOG we can assume that value is independent of \( i \) and just write \( \codeY{z} \).

We formulate the problem slightly more generally as trying to meet \( \req{R}{j,e} \) while respecting the \( n + 1 \) requirements \( \req{P}{i_0}, \req{P}{i_1}, \ldots, \req{P}{i_n} \) (i.e., assuming that agreement is automatic for \( i \) not in \( \set{i_0, \ldots i_n}{} \))   and then argue that we can reduce this to the problem of respecting \( n \) requirements.   Since we've assumed that \( z \nin \REset[s_{-1}](A){i_j}, j \leq n \) at each stage \( s_k \) where we enumerate \( c_k \) into \( \setcol{A}{3} \) this has the effect of removing \( z \) from \( \REset[s_{k}](A){i_j}, j \leq n \).  To ensure that the sets \( Y_{i_j} \) reflect this we enumerate \( l_{t} \) into \( \setcol{Y_{i_j}}{2} \) if \( z \in \REset[s_k -1](A_{s_k -1}){i_j} \) where \( t \) is the stage at which \( \codeY{z} \) entered \( \setcol{Y_{i_j}}{1} \) in the interval \( (s_{k-1}, s_k) \) .  Thus, at each stage \( s_k \) \( z \) isn't in any of the sets \( \REset(A){i_j} \) nor are the elements \( \codeY{z} \).    Note that, in the case where \( z \nin \REset[s_k -1](A_{s_k -1}){i_j} \) for all \( j \leq n \) we don't enumerate any elements into \( Y_{i_j} \) at stage \( s_k \) meaning that we automatically have \( s_k -1 \Yagree s_k \) and, by the remarks at the end of \cref{sec:overview:simplified-agreement}, once established we can maintain this agreement until stage \( s_{k+1} -1 \) as we only need to enumerate elements into the third column of the sets \( Y_{i_j} \) to respect the requirements \( \req{P}{i_0}, \req{P}{i_1}, \ldots, \req{P}{i_n} \). 

What we must do is describe a strategy for reacting to the enumeration of \( z \) into \( \REset(A){i_j} \) during the interval \( (s_{k}, s_{k+1}) \) (the interval between \( c_k \) and \( c_{k+1} \) entering \( \setcol{A}{3} \)) that ensures we will eventually produce agreement.    The key factor that determines how we act to try and produce agreement during the interval \( (s_{k}, s_{k+1}) \) is which of the sets \( \REset(A){i_0}, \REset(A){i_1}, \ldots, \REset(A){i_n} \) we see \( z \) enter first.   To this end, define \( u^j_m \) so that \( (s_{u^j_m}, s_{u^j_m + 1} \) is the \( m \)-th interval on which \( z \) first enters \( \REset(A){i_j} \).   

Our argument is essentially a priority argument in which we give highest priority to producing agreement between intervals \( (s_{k}, s_{k+1}) \) in which the first set  \( z \) enters is \( \REset(A){i_0} \).  We give second highest priority to producing agreement between intervals in which \( z \) first enters  \( \REset(A){i_1} \) and so on.   Specifically, during the interval  \( (s_{u^0_m}, s_{u^0_m + 1}) \) we will attempt to produce agreement with the stages \( s_{u^0_0 + 1} -1, s_{u^0_1 + 1} -1, \ldots, s_{u^0_{m -1} + 1} -1  \) ignoring any intervening intervals in which \( z \) first entered some other \re set in our list.  However, as when we service a higher priority requirement in a finite injury argument, when we first see \( z \) enter  \( \REset(A){i_0} \) we abandon our attempts to produce agreement on intervals where \( z \) first enters some other \re set.  Thus, during the interval \( (s_{u^1_m}, s_{u^1_m + 1}) \) we will only attempt to produce agreement with the stages \( s_{u^1_{m'} + 1} -1, s_{u^1_{m' + 1} + 1} -1, \ldots, s_{u^1_{m -1} + 1} -1  \) where \( m' \) is the least value such that there is no \( k \) with  \( s_{u^1_{m'} + 1} -1 <   s_{u^0_k}  < s_{u^1_{m -1} + 1} -1   \).  Similarly, in the interval \( (s_{u^2_m}, s_{u^2_m + 1} \) we only try to produce agreement with the stages \( s_{u^2_{m''} + 1} -1 \) which have occurred since the last interval in which \( z \) first entered  \( \REset(A){i_0} \) or \( \REset(A){i_1} \) and so on.   Of course, if we already have agreement with some earlier stage \( s_k -1 \)  (e.g. because \( z \) didn't enter any of our \re sets in the previous interval) we simply maintain it regardless of which set is entered first.

Thus, if \( s \) is the first stage in \( (s_k, s_{k+1})  \)  at which we see \( z \) enter some \( \REset(A){i_j} \) with  there is a (potentially empty) list of earlier stages\footnote{For instance, if \( z \) first enters  \( \REset(A){i_0} \) and \( (s_k, s_{k+1}) =  (s_{u^0_m}, s_{u^0_m + 1})  \)  then \( k^{s}_q  = u^{0}_q, q < m \) and if \( z \) first enters  \( \REset(A){i_1} \) and \( (s_k, s_{k+1}) =  (s_{u^1_m}, s_{u^1_m + 1})  \)  then \( k^{s}_q  = u^{1}_{m' + q}, m' + q < m \). }   \( s_{k^{s}_0 + 1}  - 1, \ldots s_{k^{s}_r + 1} - 1 \) which we will try to produce agreement with.  If this list is non-empty then we immediately\footnote{Stage \( s+1 \) since we interleave stages at which enumeration into  \( A \) and sets \re in \( A \) are seen.} and enumerate \( l_{t}, t = s_{k^{s}_r + 1}  \) into \( \setcol{Y_{i_j}}{1} \) to cancel all elements enumerated into \( Y_{i_j} \) at intervening stages.  This enumeration also cancels the element enumerated into   \( \setcol{Y_{i_j}}{2} \) at stage \( t \) and thus returns \( \codeY{z} \) into   \( \setcol{Y_{i_j}}{3} \).

This has the effect of ensuring that \(s_{k^{s}_0 + 1}  - 1 \Yagree<i_j> s_{k^{s}_1 + 1}  - 1 \Yagree<i_j>   \ldots \Yagree<i_j>   s_{k^{s}_r + 1} - 1 \Yagree<i_j>  s \).   Note that if we are only trying to respect a single requirement  \req{P}{i_0} this strategy ensures we always achieve agreement.  Either \( z \) doesn't enter \( \REset(A){i_0} \) during the interval \( (s_{-1}, s_0) \) in which case we achieve the free win mentioned above or during the interval \( (s_0, s_1) \) then the list of prior stages we are attempting to agree with is non-empty ensuring that \( s_0 -1 \Yagree s_1 -1 \). 

Now we notice that the argument we just gave depends only on having a sequence of intervals \( (t'_q, t_{q}) \) with \( t_{q - 1} \leq t'_{q} \leq t_{q} \)  such that at every stage \( t_q \) we have \( z \nin \REset(A){i_j}, j \leq n \), \( \codeY{z} \nin \setcol{Y_{i_j}}{3} \), all elements enumerated into \( Y_{i_j} \) between  \( t_{q} \) and \( t'_{q+1} \) are larger \( l_{t_q} \) and that, if \( z \) enters \( \REset(A){i_j}  \) in \((t'_q, t_{q})  \)  for some \( j \leq n \) then enumerating \( l_{t_q} \) into \( \setcol{Y_{i_j}}{1} \) restores  \( Y_{i_j}\colrestr{l_{t_q -1}} \) to the state it had at stage \( t_q - 1 \).  Finally, we observe that if \( (t'_q, t_{q}), q \leq m \)  is such that   \( t_q = u^j_{q + m'}, q + m' \leq m \) where, as above, \( z \) doesn't first enter any \( \REset(A){i_e}, e < j \) during any interval between \( u^j_{m'} \) and \( u^j_{m} \) and \( t'_q \) is the stage at which \( z \) enters \( \REset(A){i_j} \) in the interval ending at stage \( t_q \) satisfies these properties.  That is, \( t_q -1, q < m - m' \) is the sequence of stages that we try to achieve agreement with when we see \( z \) first enter \( \REset(A){i_j} \) during the interval ending with \( t_{m - m'} \).  Thus, when see \( z \) first enter \( \REset(A){i_j} \) in an interval we can then run the same strategy on the indexes in \( \set{i_0, \ldots, i_n}{} - \set{i_j} \) on these reduced intervals.

The notation gets a bit heavy in that last discussion so to make this concrete consider \cref{fig:two-Pi-requirement-interaction} where we give an example of trying to produce agreement while respecting the two requirements \req{P}{i_0} and \req{P}{i_1}.  In the figure we've only shown the final stage of intervals and we've boxed (including the double box) the stages which correspond to intervals in which  \( z \) first enters \( \REset(A){i_0} \).  Note that, on the intervals ending at stage \( s_1, s_2 \) where \( z \) first enters \( \REset(A){i_1} \) we have \( s_1 -1 \Yagree<i_1> s_2 -1 \) and  the behavior of \( Y_{i_0} \) looks just like the strategy we would follow if we were dealing with only the requirement \req{P}{i_0}.  However, when we hit the interval ending at \( s_3 \) in which \( z \) first enters \( \REset(A){i_0} \) this attempt is injured so in the interval ending \( s_4 \), \( Y_{i_1} \) starts afresh and doesn't try to create agreement with stage \( s_2 -1 \).   

Furthermore, note that at the boxed stages (the last stages of intervals at which \( z \) first enters \( \REset(A){i_0} \)) agreement works as if the unboxed stages don't exist.  Thus, since we have \( s_0 -1 \Yagree<i_0> s_3 -1 \Yagree<i_0> s_5 -1 \) just as in the case where we dealt with only a single requirement \req{P}{i_1} we eventually achieve agreement on that index as well ensuring \( s_3 -1 \Yagree s_5 -1  \).  Note that in \cref{fig:alt-two-Pi-requirement-interaction} we give two alternate ways that the interval ending at stage \( s_3 \) could have played out.  On the left we see that if we had a long enough consecutive sequence of intervals in which \( z \) first entered \( Y_{i_1} \) that too would have produced agreement while, on the right, we see how we would have achieved agreement between the double boxed stages had \( z \) entered \( Y_{i_1} \) during the interval \( )s_2, s_3) \) after \( z \) entered \( Y_{i_0} \).

Finally, one might wonder how, even though a smaller value of \( z \) entering  \( \REset(A){i_j} \) subsumes larger values entering the same set how we can handle the case where we first see the larger values enter \( \REset(A){i_j} \) when we don't yet know if the smaller value will later enter.  We can handle this the same way we do with multiple sets where seeing \( 0 \) enter  \( \REset(A){i_0} \) first can injure seeing \( 1 \) first enter \( \REset(A){i_0} \) and so forth.  Though, by the discussion above, we need not worry about the case where the larger number enters during an interval after the smaller number.

Now that we've outlined the general approach we give the full construction.  The complexity involved in this construction largely reflects the difficulty involved in doing accurate bookkeeping for the approach described in this section in the case where we must deal both with achieving agreement with respect to multiple values of \(i \) and multiple elements entering and leaving the associated \re sets.

\section{Full Construction}\label{sec:full}

\subsection{General Framework}

The construction takes the form of a finite injury argument with the module \( \module{P}{i} \) tasked with meeting \req{P}{i} having priority \( 2i \) and the module \module{R}{j,e} tasked with meeting \req{R}{j,e} having priority \( 2\pair{j}{e} +1 \).  Modules of the form \( \module{P}{i} \) merely define \( \Gfunc{i}{} \)  as specified in \cref{def:funcs} and thus have no direct effect on any set or module.  In contrast, any time the module \module{R}{j,e} acts all lower priority modules are reinitialized.   

At any even stage \( s \)  at most one requirement of the form  \req{R}{j,e} with priority at most \( 2s - 1 \)  acts while (at most one per stage) enumerations into \( \REset(A){i} \) occur only at odd stages.   In response to an enumeration into \( \REset(A){i} \) we give each requirement (in decreasing order of priority) the chance to act and claim the enumeration (injuring lower priority requirements) and enumerate elements into \( Y_i \) in response.  We apply the following rule for unclaimed enumerations. 

\begin{crule}\label{rule:unclaimed}
 If  \( z  \) enters  \(  \REset(A){i} \) at stage \( s \) but goes unclaimed then  we enumerate \( \codeY<i>[s]{z} \) into \( \setcol{Y_i}{3} \). 
\end{crule}

\subsection{Module \( \module{R}{j,e} \) }\label{sec:full:Rje}

We break up the description of the module \module{R}{j,e} into two parts.  First, we describe the behavior of the module \module{R}{j,e} at even stages when it may act of it's own volition to modify \( A \) and possibly the sets \( Y_i \).  Then, later, we describe action of the module \module{R}{j,e} at odd stages in response to an enumeration into a set \( \REset(A){i} \).   

\newcommand*{\Rstate}[2]{\ensuremath{\mathcal{S}_{#1}(\module{R}{#2})}}

\subsubsection{Even Stages}\label{sec:full:even}

% If this is the first stage at which the module for \( \req{R}{j,e} \) is executed since it was last (re)initialized then 
The module \module{R}{j,e}  tasked with meeting \( \req{R}{j,e} \) has four \( 4 \) states, \( \diverge, 0,1,2,3 \) where \( \diverge \) indicates \module{R}{j,e} has yet to execute after (re)initialization and  the later numbers (roughly) indicate the number of times we've changed our mind about whether \( c \in \setcol{A}{3} \).   We denote the state of the module \module{R}{j,e}  at the end of stage \( s \) by  \( \Rstate{s}{j,e} \) and assume that the state remains the same at stage \( s+1 \) unless otherwise noted.  Note that, in state \( 1 \) we've enumerated at least one potential value for \( c_k \)  for \( c \) into \( \setcol{A}{3}  \) but we may not yet have enumerated the ultimate value \( c \) takes  into \( \setcol{A}{3}  \).   If \req{R}{j,e} is injured at stage \( s \) we set \( \Rstate{s}{j,e} = \diverge \). 

We describe the behavior \module{R}{j,e}  at stage \( s \) where we may assume that the following all hold.
\begin{itemize}
    \item \( s \) is even, \( \pair{j}{e} < s \) and no higher priority requirement has chosen to act (injuring \req{R}{j,e}) at stage \( s \), i.e., we execute this module at stage \( s \). 
    \item \( s_{-1} \leq s \) is the first even stage after the most recent reinitialization of \module{R}{j,e}.
    \item The element \( c_{k}, k > 0 \) is defined to be \( \pair{\pair{j}{e}}{v +k } \) if \( c_0 = \pair{\pair{j}{e}}{v} \).
    \item \( c_k \)  is enumerated into \( \setcol{A}{3} \) at stage \( s_{k} =t_k +1 > s_{-1} \) and \( m \) is the smallest value for which \( s_m \) is undefined and less than \( s \) (thus \( c_m \) is the next element waiting to be enumerated into \( \setcol{A}{3} \)).  
    \item \( b_k, k < m \) are values chosen large at stage \( s_{k} \) satisfying \( b_k \nin \setcol{A_{s_{k}}}{n} \) and \( b_k \nin \setcol{Y_i}{n} \) for any \( i \) and \( n \leq 2 \) .  Note that this implies \( l_{s_k - 1} < b_k < l_{s_k} \).        
\end{itemize}

\begin{procedure}[Even Stages of \( \module{R}{j,e} \)]  \label{proc:even}\hfil \\
\begin{pfcases*}
\case[\( \Rstate{s-1}{j,e} = \diverge \) ]  We act by choosing \( c_0 = \pair{\pair{j}{e}}{v}  \) where \( v \) is chosen large and setting  \( \Rstate{s}{j,e} = 0 \).

\case[\( \Rstate{s-1}{j,e} = 0  \) ]  If \( \SelfComp{c_0}{s-1} \) holds then we act by enumerating \( c_0 \) into  \( \setcol{A}{3} \), executing \cref{proc:expansionary}  and setting \( \Rstate{s}{j,e} = 1  \)

\case[\( \Rstate{s-1}{j,e} = 1  \)]\label{modR:state-one}  We execute the following steps.
\begin{steps}
    \step If \( \SelfComp{c_m}{s - 1} \) fails to hold end the stage without acting.  Otherwise the module acts by executing the subsequent steps.

    \step\label{proc:even:state-one:change-to-two} If there is \( k < m  \) with \(  s_{k} - 1 \Yagree s - 1  \) perform the following steps and end the stage. 
        \begin{steps}
            \step Choose \( a \nin \setcol{A}{1} \) large.
            \step Set \( c = c_k, b = b_k, \hat{s}_0 = s_{k} - 1, \hat{s}_1 = s - 1 \).
            \step Enumerate \( b \) into \( \setcol{A}{2} \) and set \( \Rstate{s}{j,e} = 2  \)
        \end{steps}
    \step If there is no such \( k \) we instead enumerate \( c_m \) into  \( \setcol{A}{3} \) and execute \cref{proc:expansionary}

\end{steps}

\case[\( \Rstate{s-1}{j,e} = 2  \)]  If \( \SelfComp{c}{s-1} \) the module acts by enumerating \( a \) into \( \setcol{A}{1} \) and setting \( \Rstate{s-1}{j,e} = 3, \hat{s}_2  = s - 1 \).  %Furthermore, the module enumerates any elements into any \( \setcol{Y_i}{3} \)  

\case[\( \Rstate{s-1}{j,e} = 3  \)] Once in state \( 3 \) the module never acts again. 

\end{pfcases*}
\end{procedure}

We now specify the procedure we referenced above when we enumerated \( c_k \).  As we want to maintain \cref{cond:Y-approx:compute-right} (correctness of our functionals) at all active stages we must return  \( Y_i \), by potentially enumerating an element into \( \setcol{Y_i}{2} \), to a state compatible with \( Y_{i,s_{-1}} \).  However, what, if any,  element must be enumerated into \( \setcol{Y_i}{2} \) at \( s_k \)  is determined by how we choose to respond to enumerations into \( \REset(A){i} \) so we hand that task of keeping track of that value to the machinery which responds to those enumerations.  Specifically,  for each \(i \in \omega \)  we define (relative to \req{R}{j,e}) a marker \( \bY{i} \in  \omega \union \set{\diverge}{} \)  with initial value \( \bY[s_{-1}-1]{i} = \diverge \).  At active stages we'll just trust that we've placed this marker on the right value and do the following.

\begin{procedure}[Resetting \( Y_i \)] \label{proc:expansionary}
Execute the following steps when called:
\begin{steps} 
\step \label{proc:expansionary:enum-bY}  For every \( i < \pair{j}{e} + 1\) enumerate \( \bY{i} \) into  \( \setcol{Y_i}{2} \) if \( \bY{i}\conv \) and update \( \bY{i}  \) to \(  \diverge \). 
\step \label{proc:expansionary:fix-misc} If \(i > \pair{j}{e} \) or \( x \geq s_{-1} \)  then   \textbf{at the end of the stage} enumerate \( \codeY<i>[s]{x} \) into \( \setcol{Y_i}{3} \) if we would otherwise have \( \Tfunc{Y_{i,{s}}}(x)\conv \neq  \REset[s](A_s){i}(x)\conv \).
\end{steps}
\end{procedure}

Note that unless \( i, x < s \) no elements have ever been enumerated into \( \REset(A){i} \) so we need take no action to maintain agreement between \( \REset[s](A_s){i} \) and \( \Tfunc{Y_{i,{s}}} \).

\subsubsection{Odd Stages}\label{sec:full:odd}

It's now our task to ensure that there are two preactive stages \( t, t' \) that agree in the sense \( t \Yagree t' \).  Since we may injure lower priority modules we have no obligation to undo enumerations into \( Y_i, i \geq \pair{j}{e} +1 \) at active stages as we do with \( Y_i, i < \pair{j}{e} +1 \).  Thus, we can simply leave enumerations into \( Y_i, i \geq \pair{j}{e} +1 \) to lower priority modules (or leave them unclaimed) without concern they will block us from achieving the desired agreement.  On the other hand, if \( z > s_{-1} \) then the use of \( \Gfunc{i}{} \) for any element coding the status of \( z \) is large enough that we can redefine it every time some \( c_m \) is enumerated into \( \setcol{A}{3} \) by \module{R}{j,e}.  Finally, once \module{R}{j,e} enters state \( 2 \) we are too far along in the process for changes in \( \setcol{Y_i}{3} \) to migrate all the way to the first column in time to cause a problem.  We only claim those finitely many enumerations which don't fall into these easy cases.

\begin{crule}\label{rule:claim}
If \( z \entersat{s} \REset(A_s){i} \), \(i < \pair{j}{e} +1 \) then \( \module{R}{j,e} \) claims \( z \) just if  \( \Rstate{s-1}{j,e} \in \set{0,1}{}  \) and \( z < s_{-1} \).  
\end{crule}

Our active management in these cases is complicated by the fact that we may only modify the elements in  \( Y_i \) coding the status of \( z < s_{-1} \)   immediately in response to the enumeration of \( z \).  Our basic approach to deal with this problem is to use the order in which elements are enumerated into the sets \( \REset(A){i} \) during an interval \( [s_k, s_{k+1}-1] \) to determine our response.  However, our approach is only easily described in a recursive fashion. To that end we abstract away from the particular rule we will use to respond to enumeration and define the notion of a strategy.  A strategy is  a procedure with persistent state (i.e., a coroutine) that tells us how best to respond to an enumeration of \( z \) into the set \( \REset(A){i} \) given some finite set \( S \) of stages we are trying to produce agreement with  (or force future agreement).  
 %This abstraction will try to achieve agreement with previous preexpnsionary stages so we will have to give an explicit must first specify a simple response to enumerations before stage \( s_0 \). 

\NewDocumentCommand{\Strat}{D<>{}O{}mmm}{\mathscr{S}_{#2}^{#1}\ifthenelse{\isempty{#3#4#5}}{}{(#3, #4, #5)}}

We now define the notion of a procedure whose task is to tell us how to respond to an enumeration of \( z \) into \( \REset(A){i} \) at stage \( s \).    
\begin{definition}\label{def:explicit-strategy}
    A strategy \( \Strat{}{}{} \) is a computable function \( \Strat{s}{z}{i} \) (subroutine) with persistent state which enumerates a finite set of elements for entry into \( Y_i \) and recommends an update to \( \bY{i} \).  The persistent state of \( \Strat{s}{z}{i} \) consists of a number of variables which \( \Strat{s}{z}{i} \) updates and whose values persist across calls.  In addition to the explicitly given arguments and persistent state we also assume that \( \Strat{s}{z}{i} \) has access to a complete history of enumerations/recommendations made \footnote{Including enumerations/updates that some calling strategy has already committed to recommending but haven't yet been made.} prior to it's execution.        
\end{definition}

We will describe the behavior of a strategy by giving a procedure by which it updates it's persistent internal variables (initially set to \( \diverge \)) and leave it to the pedantic reader to translate such a description into a fully formal object.  We also need terminology for the stage (if any) in the intervals \( (s_{k-1}, s_k) \)  at which control is first passed to a particular strategy as follow.

\NewDocumentCommand{\ienum}{om}{#2^{\#\IfValueTF{#1}{(#1)}{}}}

\begin{definition}\label{def:initial-enum}
   A stage \( s \)  is an \textbf{initial enumeration} for \( \Strat{}{}{} \) (relative to \module{R}{j,e})  if  \( \Strat{}{}{} \) is executed at stage \( s \) in response to the enumeration of \( z \) into \( \REset(A){i} \) and if there is some \( k  \) such that \( s_{k-1} < s < s_{k} \) and \( \Strat{}{}{} \) was not executed during the interval \( (s_{k-1},s) \).  In such cases we also say that the initial  enumeration for \( \Strat{}{}{} \) during the interval  \( (s_{k-1}, s_{k}) \) was into \( \REset(A){i} \).  
\end{definition}

Our goal is to prove that, given any value of \( s_{-1} \) we can effectively produce a strategy \( \Strat{}{}{} \)   which, if followed whenever any \( z < s_{-1} \) enters some \( \REset(A){i}, i \leq \pair{j}{e} \), guarantees that after some finite number of preactive stages we'll succeed in producing agreement.

Specifically, we'll show that for any \( h:\omega \mapsto \omega \) with compact support there is a strategy \( \Strat<h>{}{}{} \) such that \( \Strat<h>{}{}{} \) is capable of producing agreement while respecting the requirements \req{P}{i} for \(i \in \supp h \) provided that the only elements \( z \)  entering or leaving \( \REset(A){i} \) satisfy \( h(i)> z \).  We will then argue that if for every \( h' \) with \( \int h' < \int h \) (relative to the counting measure) \( \Strat<h'>{}{}{} \) is guaranteed to win then so is  \( \Strat<h>{}{}{} \).  As when \( h \) is the zero function, there are no values of \( i, z \) with \( h(i)> z \)   \( \Strat<0>{}{}{} \) clearly satisfies the above assumption the desired result will follow by induction on \( \int h \).

First, however, we describe the procedure \( \Strat<h>{}{}{} \).  In this description, we style variables which persist across calls in Fraktur, e.g., \( \mathfrak{r},  \), and adopt the computer science convention that we specify temporary values via `let' while we update a variable in the strategy's internal state with `set'.  

\NewDocumentCommand{\Sb}{D<>{}O{}m}{\ensuremath{\mathfrak{b}^{#1}_{#3\IfValueTF{#2}{, #2}{}}}}

\begin{definition}\label{def:strat-h}
    The strategy \( \Strat<h>{}{}{} \) for \( h \in \omega^{\omega}  \) with \( \supp{h} = I = \set{i_0, \ldots, i_{N-1}}{} \) is defined to be the strategy with the following persistent internal variables (all of which are initialized to \( \diverge \)) which executes \cref{proc:odd} when called as\footnote{That is, \( \Strat<\text{arg}_0>{\text{arg}_1}{\text{arg}_2}{\text{arg}_3} \) follows the steps in \cref{proc:odd} with \( h = \text{arg}_0,  s = \text{arg}_1, z = \text{arg}_2, i_p = \text{arg}_3 \).  } \( \Strat<h>{s}{z}{i_p} \).
    \begin{itemize}
        \item \( \mathfrak{m}^{h}_r \in \omega, r < N \) which records the greatest non-canceled (i.e. accessible) preactive stage \( t_k \) at which initial enumeration for \( \Strat<h>{}{}{} \) was made into \( \REset(A){i_r} \).  Note that we may assume that \( \mathfrak{m}_r  \) actually records the value of \( k \) rather than \( t_k \) behind the seems and uses \( k \) to return \( t_k \) so that we may assign  \( t_{m+1} \) to  \( \mathfrak{m}_r  \) even when that stage hasn't yet happened.   
        \item  \( \mathfrak{S}^{h}_r, r < N  \) holds the current state of the substrategy tasked with producing agreement in response to initial enumerations made into \( \REset(A){i_r} \).
        \item \( \Sb<h>{i_r}, r < N \) identifies the value, if any, needed to cancel the values in \( Y_{i_r} \) coding the status of \( h(i_r) -1 \) (the only value  \( \Strat<h>{}{}{} \) is directly responsible for) at the next active stage.  \( \Strat<h>{}{}{} \) will recommend that \( \bY{i} \) be set to the least of \( \Sb<h>{i} \) and the recommendations of any sub-strategies.   
    \end{itemize}
\end{definition}

In what follows we drop the superscript \( h \) from persistent variables when clear from context.

\begin{procedure}  \label{proc:odd}

    Suppose that,
\begin{itemize}
    \item \( h \in \omega^{\omega} \land \supp{h} = I = \set{i_0 < i_1 < \ldots < i_N}{} \).
    \item \( m \) is largest with  \( t_m + 1 = s_m < s \).
    \item \( z< h(i_p), p \leq N \)
    \item \( z \entersat{s} \REset(A){i_p} \)
\end{itemize}
then \( \Strat<h>{s}{z}{i_p} \) behaves as follows when executed at stage \( s \).   Note that \( \Strat<h>{s}{z}{i_p} \) always returns the current value of \( \Sb{i_p} \) when it ceases execution as it's recommendation for \( \bY{i_p} \). 

\begin{steps}

    \step\label{proc:odd:if-agree} \fbox{If we project\footnotemark \( t_m \Yagree s \)}   Enumerate \( \codeY<i_p>[s]{z} \) into \( Y_{i_p} \) and exit any strategy (even parent strategies calling this one).  There is no need to recommend an update to \( \bY{i_p}\) since we are guaranteed agreement.  %Jump directly to \cref{step:enum-resp:cleanup}.

\footnotetext{Projecting on the (valid) assumption that all recommendations for enumeration committed to by strategies executing before this one during stage \( s \) will be made.}

 \step\label{proc:odd:initial}  \fbox{If \( s \) is an initial enumeration for \( \Strat<h>{}{}{} \)} set \( \Sb{i_r}, r < N \) to be undefined and execute the following steps:
\begin{steps}
    \step Set \( \mathfrak{r} = p \) and for all \( r, N > r > p \) set  \( \mathfrak{m}_r \) and  \( \mathfrak{S}_r  \) to \( \diverge \). 

    \step Let \( n \) be such that  \( t_{n} = \mathfrak{m}_{p} \) and set \( \mathfrak{m}_p  \) equal to \(  t_{m+1} \) (yes, in the future).  For all \( r < p \) if \( \mathfrak{m}_r = \diverge \) set \( \mathfrak{m}_r = t_{m+1} \).   

    \item If \( \mathfrak{S}_p \) is undefined set \( \mathfrak{S}_{p} =  \Strat<h_{p}>{}{}{} \) with all persistent variables initialized to \( \diverge \) where  
    \begin{equation}\label{eq:hr} 
    h_{p}(i) \eqdef \begin{cases}
                                                h(i) -1  & \text{if } i = i_{p} \\
                                                h(i)  & \text{otherwise.}
                                           \end{cases} 
    \end{equation} 

    \item \label{proc:odd:initial:restart}  \fbox{If \( \mathfrak{m}_p  \) was undefined at the start of execution} then do the following. If \( z = h(i_p) - 1 \)  mark \( \codeY<i_p>[s]{z} \) for enumeration into \( Y_{i_p} \) and set \( \Sb{i_p} \) to a large value returning that as our recommendation for \( \bY{i_p} \).  If \( z < h(i_p) - 1 \) then execute \( \mathfrak{S}_p(s,z,i_p) = \Strat<h_{p}>{s}{z}{i_p} \) and let our recommendation for \( \bY{i_p} \) be whatever was returned by \( \mathfrak{S}_p(s,z,i_p) \).  In either case, cease all further execution and return to caller.  

    Only if both \( \mathfrak{S}_{p} \) and \( \mathfrak{m}_{p}  \) were defined at the start of execution will we pass beyond this point.

    \item\label{proc:odd:jump-back}  Enumerate the least \( q \nin \setcol{Y_{i_p}}{1} \) satisfying \( l_{t_n} < q < l_{t_{n}+1} \) into  \( \setcol{Y_{i_p}}{1} \).  This has the effect of returning \( Y_{i_p} \)  to the state it was in at stage \( \mathfrak{m}_{p} \).  

    Note that, if we've previously enumerated a \( q' \in (l_{t_n}, l_{t_{n}+1}) \) at some stage \( t \geq t_{n} + 1 \) it is enough to enumerate \( q < l_t \) so we need not worry about running out of values but for simplicity we'll assume\footnote{An assumption that will be vindicated by the calculation of an explicit bound on the number of preactive stages before agreement and thus on the total number of strategy executions which in turn gives us an explicit bound on the number of values we may need to enumerate between \( l_{t_n}  \) and \( l_{t_{n}+1} \).} that we always have \( l_{t_n} < q < l_{t_{n}+1} \). 

    \step\label{proc:odd:restore-state-for-sub}  \fbox{If \( \exists(x < h(i_p)- 1 )\left( x \in \REset[t_n](A_{t_n}){i_p} \setminus \REset[s_{-1}](A_{s_{-1}}){i_p} \right) \)}  place  \( \bY[t_n]{i_p}  \) into \( \setcol{Y_{i_p}}{2} \).  %Skip to \cref{proc:odd:inductive}

    This has the effect of canceling (as if we were the active stage \( s_n \)) the elements in \( \setcol{Y_{i_p}}{3} \) which code the membership of those \( x  < h(i_p)- 1 \).  Note that this enumeration may be canceled by a sub-strategy at this very stage.

    \step\label{proc:odd:infer-b}  \fbox{If \( h(i_p)- 1  \in \REset[t_n](A_{t_n}){i_p} \setminus \REset[s_{-1}](A_{s_{-1}}){i_p}  \)} set \( \Sb<h>{i_p} \) to \( \Sb<h>[t_n]{i_p}  \) as at the next active stage \( s_{m+1} \) we must be prepared to cancel the enumerations into \( \setcol{Y_{i_p}}{2} \) we restored by the enumeration of \( q \).

\end{steps}

\step\label{proc:odd:inductive} \fbox{If \( p \neq \mathfrak{r} \lor z +1 < h(i_p) \)} Execute  \( \mathfrak{S}_{\mathfrak{r}}(z,i_p) \) (updating the internal state) and pass on the set of elements to be enumerated unmodified.  Let \( b \) be the recommendation for \( \bY{i_p} \)  returned by the sub-strategy just executed.

 \step \label{proc:odd:cleanup-bY} Let \( V = \set{ \laxiom{\sigma_0}{\pair{3}{y_0}}, \laxiom{\sigma_1}{\pair{3}{y_1}}, \ldots  } \) be the set of axioms projected to be active at the end of stage \( s \) such that \( y_k \)  has the form \( \pair{h(i_p) -1}{x} \) but \( y_k \nin \setcol{Y_{i, s_{-1}}}{3} \).  If \( V \) is empty then set \( \Sb<h>{i_p} \) to \( \diverge \).  Otherwise, let \( t \) be the least stage at which one these axioms is enumerated and \( t_p \) the least preactive stage after \( t \).  If there is no such \( t_p \), i.e., \( t_p = t_{m+1} \), then leave \( \Sb<h>{i_p} \) unchanged.  Otherwise, set \( \Sb<h>{i_p} \) to \( \Sb<h>[t_p]{i_p} \).  

\step\label{proc:odd:cleanup} \fbox{If \( p = \mathfrak{r} \land z = h(i_p) -1   \)} Check if, absent this step we are on track to have  \( \Tfunc{Y_{i_p,{s}}}(h(i_p)-1)\conv \neq  \REset[s](A_s){i_p}(h(i_p)-1)\conv \).  If not enumerate  \( \codeY<i_p>[s]{h(i_p)-1} \) into \( Y_{i_p}  \) and if \( \Sb<h>{i_p} \) is undefined set  \( \Sb<h>{i_p} \) to be large.  

 \step Return to caller with recommendation for \( \bY{i_p} \) equal to the minimum of \( b \) and  \( \Sb<h>{i_p} \) (always passing through any requests for enumeration).

\end{steps}

\end{procedure}

We can now give the rule for responding to a claimed enumeration using the following specification

\begin{equation*}\label{eq:h-for-enum-response}
    h^{j,e}(i) \eqdef \begin{cases}
                        s_{-1} & \text{if } i < \pair{j}{e} + 1\\
                        0 & \text{otherwise}\\
                        \end{cases}
\end{equation*}
  
We presume that we initialize a strategy of the form \( \Strat<h^{j,e}>{}{}{} \) at stage \( s_{-1} \) and adopt the following rule.

\begin{crule}\label{rule:enum-response}
If \( z \entersat{s} \REset(A_s){i} \), and \req{R}{j,e} claims \( z \) then execute \( \Strat<h^{j,e}>{s}{z}{i}  \) and follow it's recommendations.
\end{crule}

Note that, whenever we say things like ``execute \( \Strat<h^{j,e}>{S}{z}{i}  \)'' there is an ambiguity between the procedure which might be instantiated multiple times and the particular instance.  However, we'll see in part \ref{lem:descent:unique-h} of \cref{lem:descent} that in any context each strategy will only ever be instantiated once so there is no risk of confusion.

\section{Verification}\label{sec:verification}

We now demonstrate that the construction given above satisfies the requirements.  In what follows we will frequently prove results about the strategy \( \Strat<h>{}{}{} \) by either via induction or descent on strategies so we remind the reader that relation between strategies which holds whenever  \( \Strat<h>{}{}{} \) calls a strategy \(  \Strat<h_r>{}{}{} \) is a well-founded relation.  To see this note that under the counting measure on \( \omega \) , the integral is a functional from \( \omega^\omega \) to \( \omega \) and for each \( r \) we have \( \int h > \int h_r \).

To facilitate these arguments we've defined \( \Strat<h>{}{}{} \) so that when we initially execute \( \Strat<h_r>{}{}{} \) the configuration \( \Strat<h_r>{}{}{} \) sees (i.e., \( Y_i \)) resembles that seen by \( \Strat<h>{}{}{} \) (and indeed present at stage \( s_{-1} \)).  However, to make use of this we need yet more notation.

\NewDocumentCommand{\Y}{t-omD<>{}}{Y_{#3\IfValueTF{#2}{, #2}{}}^{#4\,\IfBooleanTF{#1}{-}{} }}
\begin{definition}\label{def:state-at-initial-enum}
    If the strategy \( \mathfrak{S} \) is executed at stage \( s \) then let \( \Y-[s]{i}<\mathfrak{S}> \) denote the state of \( Y_i \) at the start of execution of \( \mathfrak{S} \) (including any enumerations scheduled by strategies prior to this point) and let \( \Y[s]{i}<\mathfrak{S}> \) denote the state of \( Y_i \) at the end of execution of \( \mathfrak{S} \).  When \( \mathfrak{S}  \)  instantiates \(  \Strat<h>{}{}{} \) we also denote these values by \( \Y-[s]{i}<h> \) and \( \Y[s]{i}<h> \) respectively.
\end{definition}

Note that we'll continue to use notation like \( t_k \) which is defined relative to a requirement \req{R}{j,e} without making the dependence on \( j,e \) explicit when it is obvious from context which requirement \req{R}{j,e} is relevant. 

\subsection{Correct Computations}\label{sec:verification:correct}%\label{sec:full:ver-strat}

Our main goal in this \namecref{sec:verification:correct} is to prove that \cref{cond:Y-approx:compute-right} holds.  However, we first must prove that \( Y_{i, s_k} \supfun Y_{i,s_{-1}\colrestr{s_{-1}}} \).  But we need a number of utility results first, some of which depend on the very fact to be proved.  We therefore present the lemma below but hold off on a proof until we can gather up all the inductive hypotheses needed.

\begin{restatable}{lemma}{activestagescompat}\label{lem:bY-h-cancels}
Suppose that \req{R}{j,e} isn't initialized between \( s_{-1} \) and \( s_k \) and \( \Rstate{s_k}{j,e} < 2 \)   then
\begin{enumerate}
    \item \label{lem:bY-h-cancels:active-stages} If \( i \leq \pair{j}{e} \) and \( x < s_{-1} \)  we have \( \setcol{\setcol{Y_{i,s_{-1}}}{3}}{x} \subfun \setcol{\setcol{Y_{i,s_{k}}}{3}}{x} \). 
    \item \label{lem:bY-h-cancels:prepped} If \( s_k \) is the greatest active stage less than \( s \) and \( Y'_{i,s} \) is the result of enumerating \( \bY[s]{i} \) into \( \Y[s]{i} \), \( i \leq \pair{j}{e}, x < s_{-1} \)  then \( \setcol{\setcol{Y'_{i,s}}{3}}{x} \supfun \setcol{\setcol{Y_{i,s_{-1}}}{3}}{x} \).  %Moreover,  \( Y_{i,s_k}\colrestr{s_{-1}} = Y_{i, s_{-1}}\colrestr{s_{k}} \).
\end{enumerate}
\end{restatable}

We now prove our utility results, some of which assume that \cref{lem:bY-h-cancels} holds at earlier stages.

\begin{lemma}\label{lem:descent}
    Suppose that \( \mathfrak{S}  \) instantiates \(  \Strat<h>{}{}{} \) and receives an initial enumeration of \( z \) into \( \REset(A){i} \)  at stage \( t \in (s_k, s_{k+1}) \) (note that \( s_{k+1} \) is possibility infinite) then
    \begin{enumerate}
        % \item\label{lem:descent:initial-enum-for-h} If \cref{lem:bY-h-cancels} holds at stage \( s_k \) then  \( \setcol{\setcol{Y_{i,s_{-1}}}{3}}{z} \subfun \setcol{\setcol{\Y-[t]{i}<h>}{3}}{z}  \).  In other words, initial enumerations to strategies start with a blank slate.  
        % \item     
        \item\label{lem:descent:handle-comp} If \( x \) is enumerated into \( \REset(A){i'} \) at stage \( s \in (s_k, s_{k+1}) \) then \( \mathfrak{S} \) is executed to handle this enumeration iff \( x < h(i')  \).  
        \item \label{lem:descent:straat-only-resp-for-one} At no stage in \( (s_k, s_{k+1}) \) does any strategy other than \( \mathfrak{S} \) enumerate elements into \( \setcol{\setcol{Y_{i}}{3}}{h(i) - 1}  \).
        \item \label{lem:descent:straat-controls-eligible-coders} At no stage in \( (s_k, s_{k+1}) \) does any strategy not called from \( \mathfrak{S} \) enumerate elements into \( \setcol{\setcol{Y_{i}}{3}}{z}  \).
        \item\label{lem:descent:unique-h} No other strategy implementing \(  \Strat<h>{}{}{} \) executes during  \( (s_k, s_{k+1}) \).
        \item\label{lem:descent:q-small}  If \( \mathfrak{S} \) executes stage \cref{proc:odd:jump-back} enumerating \( q \) with \( l_{t_n} < q < l_{t_n + 1} \) into \( \setcol{Y_i}{1} \)  then no \( q' < l_{t_n} \) has been enumerated since stage \( t_n \).    
    \end{enumerate}
    If we further assume that \cref{lem:bY-h-cancels} holds for \( s_{k'}, k' \leq k  \)  and \( x < h(i') \) then we may also conclude.
    \begin{enumerate}[resume*]
        % \item\label{lem:descent:bound-enum}
        \item\label{lem:descent:out-until-seen}  For all \( k' \leq k, x < s_{-1} \) and \( s \in (s_{k'}, s_{k'+1}) \) if   \( \forall(x' \leq x)\left(x' \nin \REset[s](A){i'} \setminus \REset[s_{-1}](A){i'}\right) \) then \( \setcol{\setcol{Y_{i',s_{-1}}}{3}}{x} \subfun \setcol{\setcol{Y_{i', s}}{3}}{x}  \). 
         \item\label{lem:descent:initial-enum-for-h} For all \(i', x \) if \( x < h(i') \) then  \( \setcol{\setcol{Y_{i',s_{-1}}}{3}}{x} \subfun \setcol{\setcol{\Y-[t]{i'}<h>}{3}}{x}  \).  In other words, initial enumerations to strategies start with a blank slate.  
     \end{enumerate}
\end{lemma}

\begin{proof}
We prove these claims by descent.   For claims \ref{lem:descent:handle-comp}, \ref{lem:descent:straat-only-resp-for-one}, \ref{lem:descent:straat-controls-eligible-coders} and \ref{lem:descent:unique-h} we observe they clearly hold for \( h = h^{j,e} \) and we now prove that these claims also hold for the strategies \( \mathfrak{S}_{\mathfrak{r}} = \Strat<h_{\mathfrak{r}}>{}{}{}  \) called from \( \mathfrak{S}  \).  
 
% To see that \ref{lem:descent:unique-h} holds note that only \( \mathfrak{S}_{\mathfrak{r}} \) is ever executed by  \( \mathfrak{S}  \) during the interval.  Note that as the lemma is only concerned with those strategies that are executed during the interval this means we can limit ourselves only to showing the claim holds for  

To see that \ref{lem:descent:handle-comp} holds note that the only time \( \mathfrak{S}  \) doesn't pass on \( z \) is when \( z = h(i) - 1 \) and that by \cref{eq:hr} this is the only element which isn't below \( h_{\mathfrak{r}} \).     To see that \ref{lem:descent:straat-only-resp-for-one} (and, thus, \ref{lem:descent:straat-controls-eligible-coders}) holds note that only \cref{proc:odd:cleanup} ever enumerates elements and then only if \( z = h(i_p) - 1 \).      \( \mathfrak{S}_{\mathfrak{r}} \) is passed any element \( x < h_r(i') \) as the only time \( x \)   received by \( \mathfrak{S} \) and \( \mathfrak{S} \) only enumerates into \( \setcol{Y_i}{\leq 2} \) before executing \( \mathfrak{S}_{\mathfrak{r}} \) for the first time.

To see that \ref{lem:descent:q-small} holds it is enough to observe that if the first time after reinitialization that \( \mathfrak{S} \) runs is \( s  \in (s_{k'}, s_{k'+1}) \) (with \( s_{k' + 1} \) possibly infinite) then neither \( \mathfrak{S} \) nor any strategy called from it will enumerate elements below \( l_{s_{k'}} \).  Then observe that if we ever run \( \mathfrak{S}_r  \) with \( r < p \) we reinitialize \( \mathfrak{S}_r \) so that if we had run some strategy that wasn't called by \( \mathfrak{S} \) but enumerated a value below \( l_{t_n} \) we would have reinitialized \( \mathfrak{S} \) after \( t_n \) contradicting the choice of \( q \).   The conclusion follows by observing that, after stage \( t_n \), the first time \( \mathfrak{S} \) is run it enumerates \( q  \) before any sub-strategy it calls enumerates anything.

 %Without loss of generality we may assume \( x \nin \REset[s_{-1}](A){i'} \) or the conclusion would hold trivially.

To prove \ref{lem:descent:out-until-seen} note that the base case holds by the assumption of \cref{lem:bY-h-cancels} and we can ignore any stages at which no enumeration into \( \REset(A){i'} \) is claimed by \module{R}{j,e}.  If \(  \Strat<h'>{s}{y}{i'} \) is executed at \( s \)  then \( h'(i') - 1 \geq y > x  \) so direct enumeration can't cause disagreement.  This leaves only enumeration into \( \setcol{Y_{i'}}{\leq 2} \) as a threat.  However, when this happens \cref{proc:odd:jump-back} restores the state as of \( t_n \) and by the inductive assumption (with \( t_{n} \) replacing \( s \) and \(   s_{n - 1} \) replacing \( s_{k'} \))  we can conclude that if no \( x' \leq x  \) are in \(  \REset[s](A){i'} \setminus \REset[s_{-1}](A){i'} \) then \( \setcol{\setcol{Y_{i',s_{-1}}}{3}}{x} \subfun \setcol{\setcol{Y_{i', s}}{3}}{x}  \).  Conversely, if there is such an \( x' \) then \cref{proc:odd:restore-state-for-sub} enumerates \( \bY[t_n]{i'}  \) into \( \setcol{Y_{i'}}{2} \) which, by \ref{lem:descent:q-small} ensures that \( \setcol{\setcol{Y_{i',s_{-1}}}{3}}{x} \subfun \setcol{\setcol{Y_{i', s}}{3}}{x}  \) by the assumption of  \cref{lem:bY-h-cancels}.  

With \cref{lem:descent:out-until-seen} established we can now proceed to prove \ref{lem:descent:initial-enum-for-h} in exactly the same manner.  Note that  \ref{lem:descent:initial-enum-for-h}  clearly holds if \( h = h^{j,e} \) by our assumption that \cref{lem:bY-h-cancels}  holds at \( s_k \).  For the inductive case we need merely replace the assumption that no \( x' \leq x \) are in \(  \REset[s](A){i'} \setminus \REset[s_{-1}](A){i'} \) with the assumption that \(  \Strat<h'>{}{}{} \) is executing before the initial enumeration for \(  \Strat<h>{}{}{} \) noting that by \ref{lem:descent:handle-comp} if \( \Strat<h'>{s}{y}{i'} \) executes before the initial enumeration for \(  \Strat<h>{}{}{} \) in \( (s_k, s_{k+1}) \)  then  \( h'(i') - 1 > x \).   
\end{proof}

We now return and provide the promised proof.

\activestagescompat*
\begin{proof}
We first prove \ref{lem:bY-h-cancels:prepped} on the assumption that \ref{lem:bY-h-cancels:active-stages} holds at each \( s_{k'}, k' \leq k \).

By \cref{lem:descent}  it is enough to argue that if \( z = h(i) - 1 \) then \(  \Sb<h>[s]{i}   \) (and thus \( \bY[s]{i} \leq \Sb<h>[s]{i} \)  ) cancels any changes made by \( \Strat<h>{s}{z'}{i}, z' \leq z \) to \( \setcol{\setcol{Y'_{i,s}}{3}}{z} \).  After all, \cref{lem:descent:initial-enum-for-h} ensures that the claim holds true at all stages prior to such an enumeration, \cref{lem:descent:initial-enum-for-h} ensures that it is true when \( \Strat<h>{}{}{} \) receives it's initial enumeration, stages at which enumeration occurs into some \( \REset(A){i'}, i' \neq i\) don't alter  \( \setcol{Y'_{i,s}}{3}, \Sb<h>{i} \) or \(  \bY{i}  \).  But the combination of \cref{proc:odd:cleanup-bY,proc:odd:cleanup} and out assumption that \ref{lem:bY-h-cancels:active-stages} holds at all \( s_{k'}, k' \leq k \)   ensure that \(  \Sb<h>[s]{i} \) is small enough to cancel any active axioms causing \( \setcol{\setcol{Y'_{i,s}}{3}}{z} \) to disagree with \( \setcol{\setcol{Y_{i,s_{-1}}}{3}}{z} \) or \cref{proc:odd:initial:restart} when it is an initial enumeration and \( \mathfrak{m}_i \)  is undefined.

We now note that \ref{lem:bY-h-cancels:prepped} is trivial if \( k = -1 \) and otherwise follows by induction on \( k \) by application of \ref{lem:bY-h-cancels:prepped} and the fact that \cref{proc:even} executes \cref{proc:expansionary} when it enumerates \( c_k \).  
\end{proof}

We are now in a position to prove \cref{cond:Y-approx:compute-right} holds.  We break the proof up into two lemmas.

\begin{lemma}\label{lem:maintain-computations-T}
    For all \( s,i \), \( \Tfunc{Y_{i,s}}  \) is compatible with \( \REset[s](A_s){i} \).  Moreover, if \( \REset[s](A_s){i}(x) \diverge \) then \( \Tfunc{Y_{i,s}}(x) = 0  \).  
\end{lemma}
\begin{proof}
The moreover claim is trivial since if \( x > l_s \) then we've taken no action regarding \( x \)  at or before stage \( s \).   We let \( i \) be arbitrary and proceed by induction to prove the main claim.  Trivially, the claim holds if \( s = 0 \) so we assume that the claim holds for all stages prior to \(  s \) and argue the \namecref{lem:maintain-computations-T} holds at \( s \).  If no enumeration happens at \( s \) into either \( A \) or  \( \REset(A){i} \)   the result is immediate. So we consider the following (disjoint) cases.

\begin{pfcases*}
\case[\( x \entersat{s} \REset(A){i} \) ]  This case can be broken into two subcases.
\begin{pfcases*}
\case[\( x  \) is claimed by some \module{R}{j,e}]  In this case, there is some \( \Strat<h>{}{}{} \) executing at \( s \) such that \( h(i) - 1 = x \).  \Cref{proc:odd:cleanup} plus the fact that no elements are enumerated into either \( \setcol{Y_i}{\leq 2} \) or \( \setcol{\setcol{Y_i}{3}}{x}  \) after \cref{proc:odd:cleanup} during stage \( s \) establish the claim.
\case[\( x  \) is unclaimed]  In this case, the result follows immediately by \cref{rule:unclaimed}.
\end{pfcases*}
 \case[\( c_k \entersat{s} \setcol{A}{3} \)] Suppose \module{R}{j,e} is responsible for the numeration of \( c_k \) and consider the following subcases.

 \begin{pfcases*}
    \case[\( x \in \REset[{s_{-1}}]({A_{s_{-1}}}){i} \) ] In this case, \( x \in \REset[s](A_{s}){i} \) as any enumeration below \( l_{s_{-1}} \) would have injured \module{R}{j,e}.  Similarly, \( \setcol{\setcol{Y_{i,s_{-1}}}{3}}{x} \subset  \setcol{\setcol{Y_{i,s}}{3}}{x} \).  But as we only enumerate elements into   \( \setcol{\setcol{Y_{i}}{3}}{x}  \) when it appears that the lemma would fail no such elements are ever enumerated at or after \( s_{-1} \) so the claim holds by virtue of the inductive hypothesis applied at stage \( s_{-1} \). 

    \case[\( x \nin \REset[s_{-1}](A_{s_{-1}}){i} \)] Note that if \( x \geq s_{-1} \lor i> \pair{j}{e} \) then the claim follows immediately via the action of \cref{proc:expansionary:fix-misc}.  So assume \( x < s_{-1} \land i \leq \pair{j}{e} \).   Since \( x \) is enumerated into \( \REset(A){i} \) after stage \( s_{-1} \) it is (by induction) enumerated dependent on \( \setcol{A}{3}(c_k) = 0 \) .  Thus, \( x \nin \REset[s](A_{s}){i} \) and the result follows by  \cref{lem:bY-h-cancels}. 
\end{pfcases*}

        \case[\( b=b_k \entersat{s} \setcol{A}{2} \)]  Suppose \module{R}{j,e} is responsible for the numeration of \( b \).  \Cref{proc:even:state-one:change-to-two} ensures that there is some \( k \) with \( s_k < s \)  such that \( s_{k} - 1 \Yagree s - 1 \).  But the enumeration of \( b \) (chosen large at stage \( s_k \) ) cancels all enumerations into \( \setcol{Y_i}{3} \) after stage \( s_{k} - 1 \) and as \( s_{k} - 1 \Yagree s - 1 \) this ensures that \( Y_{i,s} \supfun Y_{i,s_k} \).  As the enumeration of \( b \) also cancels any enumeration of \( x \) into \( \REset(A){i} \) after stage \( s_k -1 \) the result follows from the application of the inductive hypothesis to stage \( s_k \). 

        \case[\( a \entersat{s} \setcol{A}{1} \)] Suppose \module{R}{j,e} is responsible for the enumeration of \( a \).  The enumeration of \( a \) cancels all enumerations into \( \setcol{Y_i}{3} \) since the stage \( t \) at which  \( b \) was enumerated as well as any enumeration of \( x \) into \( \REset(A){i} \) after stage \( t -1 \).  As no enumerations were made into \( \setcol{Y_i}{\leq 2}\colrestr{l_t} \) since stage \( t - 1 \) the state at stage \( t - 1 \) is restored and the claim follows from the inductive hypothesis applied at stage \( t - 1 \).     
    \end{pfcases*}
\end{proof}

Now to prove the corresponding claim for \( \Gfunc{i}{} \). 

\begin{lemma}\label{lem:maintain-computations-G}
    For all \( s,i \),   \( \Gfunc{i}{A_s \Tplus \REset[s](A_s){i}} \subfun Y_{i, s} \) and \( \Gfunc{i}{} \) is well-defined. 
\end{lemma}
Note that here \( \Gfunc{i}{} \) really refers to whatever version of the functional is left uninjured at the end of stage \( s \) since we will reinitialize each \( \Gfunc{i}{} \) finitely many times.  
 
 \begin{proof}
 By \cref{def:funcs} it is enough to show that \( \Gfunc{i}{} \) is well-defined.  Specifically, we must show that if \( t < s \) and  \module{P}{i} is not injured during the interval \( (t,s] \) then
\begin{subequations}
    \begin{align}
    A_s \Tplus \REset[s](A_s){i} \supfun & A_t \Tplus \REset[t](A_t){i} \label{eq:lem:maintain-computations-G:inductive:antecedent} \\
         & \implies Y_{i,t}\restr{t} \subfun Y_{i,s} \label{eq:lem:maintain-computations-G:inductive:cons}
    \end{align}
    \label{eq:lem:maintain-computations-G:inductive}
\end{subequations}

 The claim is trivial if $ t = s$ so we assume the claim holds for all $ s' $ with $ t \leq s' < s $ and prove it also holds
at $ s $. This is evident if no axioms placing an element into $ A $ or $ W_{i}^{A} $ are enumerated at stage
$ s $.  So, suppose, for contradiction, that some element is enumerated
at stage $ s $ and $\module{P}{i}$ is not injured during the interval
$ (t,s] $ but that eq. \textup{(\ref
{eq:lem:maintain-computations-G:inductive})} fails.
We consider the following cases.
\begin{description}
\item[Case $ z \entersat{s} W^{A}_{i} $ and $ A_{s - 1} \subfun A_s $:]   For the claim to fail we must have $ z $ in (our approximation to) $ W^{A}_{i} $ at stages $ t $ and $ z < t $ but not $ s - 1 $.  Thus, we must have some $ q < l_t $ in $ A_{s - 1} \subfun A_s $ but not $ A_t $  rendering eq. \textup{(\ref
{eq:lem:maintain-computations-G:inductive:antecedent})} false.
\item[Case $ q \entersat{s} A $:] Suppose that $\mathcal{R}_{j,e}$ enumerates
$ q $. Note that if $i > \langle j,e\rangle $ and $\mathcal{R}_{j,e}$
takes any
action then $\mathcal{P}_{i}$ is injured giving us the desired
conclusion trivially.
Thus, we can assume $i \leq\langle j,e\rangle $ and consider the following
subcases.
\begin{description}
\item[Case $ q=c_{k} \entersat{s} A{}^{[3]} $:] We must have
$ t < s_{-1} $ or $ c_{k'} $ makes
eq. \textup{(\ref
{eq:lem:maintain-computations-G:inductive:antecedent})} false where
$ s_{k'} $ is the least active stage greater than $ t $. If
$ A_{t} \oplus W_{i,t}^{A_{t}} \nprec A_{s_{-1}} \oplus
W_{i,s_{-1}}^{A_{s_{-1}}} $ then that disagreement would persist
until stage $ s $ making
eq. \textup{(\ref
{eq:lem:maintain-computations-G:inductive:antecedent})} false. But
Lemma~\ref{lem:bY-h-cancels} lets us conclude that
$ Y_{i,s_{-1}}\restr{t} \subfun Y_{i,s} $ so
$ Y_{i,t}\restr{t} \subfun Y_{i,s} $.

\item[Case $ q \entersat{s} A{}^{[\leq2]} $:] This follows from the inductive
assumption plus Lemma~\ref{lem:bY-h-cancels} which ensures that if
$ t < s_{-1} $ then
eq. \textup{(\ref{eq:lem:maintain-computations-G:inductive:cons})}
holds. If
$ t \geq s_{-1} $ then the enumeration of some $ c_{k} $ or
$ q $ itself renders
eq. \textup{(\ref
{eq:lem:maintain-computations-G:inductive:antecedent})} false.\qedhere
\end{description}
\end{description}
\end{proof}
  
We can now prove the desired proposition.

\begin{restatable}{proposition}{maintaincomps}\label{prop:maintain-computations}
\Cref{cond:Y-approx:compute-right} is satisfied, i.e., \( \Tfunc{} \) and \( \Gfunc{i}{} \) are well-defined and correct at every stage.
\end{restatable}

 \begin{proof}
 Immediate from \cref{lem:maintain-computations-T,lem:maintain-computations-G}
 \end{proof}

\subsection{Winning Strategies}\label{sec:verification:winning}

In this \namecref{sec:verification:winning} we seek to show that \module{R}{j,e} only acts finitely many times.  To prove this we need to show that our strategies eventually guarantee we find agreement between stages.  But the strategy \( \Strat<h>{}{}{} \) doesn't execute  \( \Strat<h_r>{}{}{} \) on consecutive intervals \( [s_k, s_{k+1}] \) .  It may, instead, skip a number of active stages during which \( \Strat<h>{}{}{} \) executes some other sub-strategy only to return to \( \Strat<h_r>{}{}{} \) much later so we need some way to talk about intervals which behave as if we had run \( \Strat<h_r>{}{}{} \) on consecutive intervals.

To that end, we make the following definitions (these are implicitly relative to a particular module \( \module{R}{j,e} \)). 

% \NewDocumentCommand{\hequiv}{m}{\mathbin{\propto_{#1}}}

\begin{definition}\label{def:h-equiv}
    Say that  \(  t \) is an \( h \) successor of \( s < t \) (relative to the module \module{R}{j,e}), just if \( t \) is an initial enumeration for \( \Strat<h>{}{}{} \) and letting \( \hat{Y}_i = \Y-[s]{i}<h> \)   we have  
    \[ \forall(i \leq \pair{j}{e})\forall(x < h(i))\left(\setcol{Y_{i,s}}{\leq 2}\colrestr{s} \subfun \setcol{\hat{Y}_i}{\leq 2} \land \setcol{\setcol{Y_{i,s}}{3}}{x} \subfun \setcol{\setcol{\hat{Y}_i}{3}}{x} \right) \]     
\end{definition}

Informally speaking, \(  t \) is an \( h \) successor of \( s \) just if for all \(i \leq \pair{j}{e} \)  the state of \( Y_{i} \) when \( \Strat<h>{}{}{} \) receives it's initial enumeration during \( t \) agrees with that of \( Y_{i,s} \) excepting only elements \( \Strat<h>{}{}{} \) can ignore as too large for it's concern.  Using this notion we can now define a notion of a sequence of intervals that look as if they are consecutive as far as \( \Strat<h>{}{}{} \) is concerned.   

\begin{definition}\label{def:virtual-contiguous}
    Say that a sequence \( [\hat{t}_0, \hat{s}_0], [\hat{t}_1, \hat{s}_1] \ldots, [\hat{t}_n, \hat{s}_n] \) of intervals is \( \Strat<h>{}{}{} \)-\textbf{virtually consecutive} (or just \( h \)-virtually consecutive) if% \( \hat{t}_k \) is the \( k + 1 \)-st initial enumeration for \( \Strat<h>{}{}{} \) during  \( [\hat{t}_0, \hat{t}_n] \) and if the following hold. 
    \begin{enumerate}
            \item\label{def:virtual-contiguous:not-reinitialized} \( \Strat<h>{}{}{} \) is not reinitialized   during \( (\hat{t}_0, \hat{s}_n) \) 
            \item \label{def:virtual-contiguous:initial} At stage \( \hat{t}_0 \)  \( \Strat<h>{}{}{} \) is in it's initialized state, i.e., hasn't been executed since it was last reinitialized.
            \item \label{def:virtual-contiguous:intervals} \( \hat{s}_k \) is the least active stage greater than \( \hat{t}_k \) 
            \item \label{def:virtual-contiguous:successor} For all \( k < n \), either \( \hat{t}_{k+1} \) is an \( h \)-successor of \( \hat{s}_k \)  or \( \hat{t}_{k+1} = \hat{s}_{k+1} \) and no elements are enumerated during the interval ending at \( \hat{s}_{k+1} \) and started at the prior active stage.
            \item \label{def:virtual-contiguous:contiguous} Every interval \( \hat{t}_k, \hat{s}_k \) satisfying the above between  \( \hat{t}_0  \) and \( \hat{s}_n \) appears in the above list.

        \end{enumerate}    
\end{definition}

Note that if the sequence is genuinely consecutive, i.e.,  \( \hat{t}_{k+1} = \hat{s}_k + 1 \) for all \( k \), then it is trivially \( h \)-virtually consecutive.  Now let's define what it means for a strategy to be successful.

\begin{definition}\label{def:winning}
    Say that the strategy \( \Strat<h>{}{}{} \) is a \textbf{winning strategy} if there is a number \( N_h \) (the winning number for \( h \)) such that if \( [\hat{t}_0, \hat{s}_0], \ldots, [\hat{t}_{N_h - 1}, \hat{s}_{N_h - 1}]  \) is an \( h \) virtually consecutive sequence of intervals of length \( N_h \)  then, for \textbf{any} acceptable enumeration of elements into the sets \( \REset(A){i}, i \in \supp h \)  satisfying \( x < h(i) \) there are distinct  \( k, k' < N_h \) if the recommendations of \( \Strat<h>{}{}{} \) are followed we have \( \hat{s}_k -1 \Yagree \hat{s}_{k'} -1 \).   
\end{definition}

We leave the task of formalizing the notion of relative to any acceptable enumeration to the pedantic reader but take the intended content to be clear.  In particular, a winning strategy is guaranteed to produce agreement in a bounded number of intervals even if we permute the indices to \re sets.  With this notion in place we can finally demonstrate that our strategies are winning strategies.

\begin{lemma}\label{lem:win-induct}
    Suppose that for all \( h_r, r < \card{\supp h} \) the strategy \( \Strat<h_r>{}{}{} \) is a winning strategy with winning number \( N_{h_r} \) then  \( \Strat<h>{}{}{} \) is a winning strategy.
\end{lemma}
% Remember that the \( h_r \) are the \( \card{\supp h} \) functions which characterize potential substrategies called by  \( \Strat<h>{}{}{} \)
\begin{proof}
Suppose that \( [\hat{t}_0, \hat{s}_0], [\hat{t}_1, \hat{s}_1] \ldots, [\hat{t}_{N-1}, \hat{s}_{N-1}] \) is a sequence of \( h \) virtually consecutive intervals with \(\displaystyle  N = \prod_{r < \card{\supp h}} (N_{h_r} +2)   \) with \( z_k \entersat{\hat{t}_k} \REset(A){i_k} \) We argue there are \( k, k' < N \) with  \( \hat{t}_k \Yagree \hat{t}_{k'} \).   

% Let \( \mathfrak{r}_k \) be the value that \( \mathfrak{r} \) takes at stage \( \hat{t}_k \) inside \( \Strat<h>{}{}{} \).   

We now note that there must be a subsequence of length at least \( N_{h_r} + 2 \) on which we always execute the same substrategy but never reinitialize that strategy.  Specifically, we observe that there is an \( \hat{i} \in \supp h \) and a subsequence \( \hat{t}_{n_k}, k < N_{h_r} + 2 \) such that \( i_{n_k} = \hat{i} \)  and for all \( n'  \) with \( n_k < n' < n_{k+1} \)  we have \( i_{n'} > \hat{i}  \).  This follows since if this isn't witnessed by \( i_0 \) then there must be a subsequence of length at least \(\displaystyle  N = \prod_{1 < r < \card{\supp h} } (N_{h_r} +2)   \).  

As the support of \( h \) is finite we must eventually find such a subsequence.   Call this a good \( \hat{i} \) subsequence and let \( \hat{r} = \card{\set{y \in \supp h \land y < \hat{i}}} \), i.e., \( \Strat<{h_\hat{r}}>{}{}{} \) is the strategy called (if necessary) by \( \Strat<h>{}{}{} \) in an interval in which the initial enumeration was into \( \REset(A){\hat{i}} \).   Thinking of this in terms of a finite injury construction this subsequence represents the stages at which the \( \hat{i} \) strategy gets attention without intervening injury.

% We now use this to show there are \( k, k' \) with \( \hat{s}_{n_k} -1 \Yagree \hat{s}_{n_k'} -1 \).  

Suppose that for some \( k < N_{h_{r'}} + 1 \) we fail to  enumerate any \( x < h(i) - \delta_{\hat{i}}(i) \) into \( \REset(A){i} \)   during  \( (\hat{t}_{n_k}, \hat{s}_{n_k}) \).    We claim that  \( \hat{s}_{n_k} - 1 \Yagree \hat{s}_{n_k+1}  - 1 \). By \( h \)-virtual connectivity when \( \Strat<h>{}{}{} \) begins execution at stage \( \hat{t}_{n_{k+1}}  \), \( \setcol{Y_i}{\leq 2} \) is compatible with it's state at stage \( \hat{s}_{n_k} - 1 \).  Note that, by \cref{def:winning}, we may assume that only \( x < h(i) \) are enumerated into any \( \REset(A){i} \).  Thus, in any interval the first element is enumerated at stage \( \hat{t}_k \).   Thus, when \( \Strat<h>{}{}{} \) executes at stage \( \hat{t}_{n_{k+1}}  \) the action of  \cref{proc:odd:jump-back} cancels any elements enumerated both at stage \( \hat{s}_{n_k} \) and by any ancestor strategies executing \cref{proc:odd:restore-state-for-sub} (which must occur earlier in that stage).

As no \( x < h(i) - \delta_{\hat{i}}(i)  \) was enumerated during the interval \( (\hat{t}_{n_k}, \hat{s}_{n_k}) \) \cref{proc:odd:restore-state-for-sub} never enumerates any elements and as no parent strategy enumerates any elements into the first two columns of \( Y_i \) after executing it's child that agreement persists until the end of stage, and, indeed, the end of stage \( \hat{s}_{n_k+1}  - 1 \).

Thus, we can assume that we have a subsequence of length  \([\hat{t}_{n_k} , \hat{s}_{n_k}], k < N_{h_{r'}} + 1  \) with \(  \mathfrak{r}_{n_k} = r' \) of length \( N_{h_{r'}} + 1 \) during which we never  reinitialize \( \mathfrak{S}_{r'} \) and during each interval \( [\hat{t}_{n_k} , \hat{s}_{n_k}) \) there is a \( \hat{t}'_{n_k} \) at which we execute  \( \Strat<h_{r'}>{}{}{} \).  To verify the claim we first show that the sequence \([\hat{t}'_{n_k} , \hat{s}_{n_k}], k < N_{h_{r'}}  \)  is \( h_{r'} \) virtually consecutive.  Note that \( i_{r'} = \hat{i} \) here and that we may presume that this is the first such subsequence, i.e., \( \hat{t}_{n_0} \) is the first execution of \( \Strat<h_{r'}>{}{}{} \) since last reinitialization.

This leaves only part \ref{def:virtual-contiguous:successor} of \cref{def:virtual-contiguous} to be verified but this follows by \cref{lem:bY-h-cancels}.   Now suppose that for some \( k \) the only element enumerated during the interval  \( [\hat{t}_{n_k} , \hat{s}_{n_k}) \) with \( k < N_{h_{r'}}  \)  into \( \REset(A){\hat{i}} \).  By the same reasoning above we'll see agreement at \( \hat{s}_{n_{k+1}} \)  since this means that for \( \hat{i} \) we behave as above and the inductive hypothesis handles enumeration into \( \REset(A){i},i \neq \hat{i} \).  Thus, we can assume that in each interval \( [\hat{t}_{n_k} , \hat{s}_{n_k}) \)   we see some \( x < h(\hat{i}) \) enumerated into \( \REset(A){\hat{i}} \).  In this case, when that \( x \)  is enumerated the operation of \cref{proc:odd:jump-back} by \( \Strat<h_{r'}>{}{}{} \) cancels any element enumerated at stage \( \hat{t}_{n_k} \) during \cref{proc:odd:restore-state-for-sub}.  Thus, the situation with respect to the first two columns of the sets \( Y_i \)  is no different than if enumeration was restricted only to those \( x \) entering \( \REset(A){i} \) where \( x < h_{r'}(i) \).

% s    
\end{proof}

We can now prove our desired result.

\begin{proposition}\label{prop:R-finite}
    \module{R}{j,e} acts only finitely many times.
\end{proposition}

\begin{proof}
Suppose not.  The only possibility is that some module \module{R}{j,e} is initialized at some stage \( s_{-1} \) and never subsequently reinitialized by the action of any higher priority module.  But this can only happen if at all active stages \( s_k, k > 0 \) we have \( \Rstate{s_k}{j,e} = 1 \). 

But, by \cref{lem:win-induct} we have that \( \Strat<h^{j,e}>{}{}{} \) is a winning strategy.   Consider the sequence \( [\hat{t}_k, s_k] \) where \( \hat{t}_k \) is the first stage subsequent to \( s_k \) at which \( \module{R}{j,e} \) claims an enumeration.  As enumerations that go unclaimed or are claimed by modules with lower priority don't enumerate elements that are small relative to the last stage at which \module{R}{j,e} acts into the first two columns of any \( Y_i \) it is trivial that this sequence is \( h^{j,e} \) virtually contiguous.  It thus follows that there are \( k, k' \) with \( s_k - 1 \Yagree s_{k'} - 1 \) and thus at stage \( s_{k'} \) we have \( \Rstate{s_{k'}}{j,e} = 2 \) contrary to our assumption.  Hence every module acts only finitely often.   
\end{proof}

\subsection{Putting It Together}\label{sec:verification:final}

\begin{lemma}\label{lem:cond-Rje-satisfaction}
    Every requirement of the form \req{R}{j,e} is satisfied.
\end{lemma}
\begin{proof}
Assume, for a contradiction, \req{R}{j,e} isn't satisfied (i.e., \( \recfnl{j}{A}{} = X_e \land \recfnl{j}{X_e}{} = A \)) and that \( t_0 \) is the last stage at which the module \( \module{R}{j,e} \) acts.  If \( s > t_0 \)  then \( \Rstate{t_0}{j,e} = \Rstate{s}{j,e}  \) and \( A_{t_0} \subfun A_s \)  since such a change would require \( \module{R}{j,e} \) act at that stage.  Hence, \( A_{t_0} \subfun A \). 

If the variable \( c \) is defined for the \( \module{R}{j,e} \)  at stage \( t_0 \) (i.e., if \( \Rstate{t_0}{j,e}  \geq 2  \)) we leave that value in place but if not set \( c = c_m \) where \( c_m \) is the least \( m \) such that \( c_m \) hasn't been enumerated into \( \setcol{A}{3} \).  Note that \( c_0 \) must be defined at \( t_0 \) since if \( \Rstate{t_0}{j,e} = \diverge \)  then  \( \module{R}{j,e} \) acts at stage \( t_0 + 1 \) 

By \cref{lem:inf-many-good-comps} let \( s > t_0 \) be a stage at which \( \SelfComp{c}{s} \) holds.  But if \( \Rstate{s}{j,e} < 3  \) then we have \( \module{R}{j,e} \) acts at stage \( s + 1 \).  On the other hand if \( \Rstate{s}{j,e} = 3  \) then we have \( \SelfComp<3>{c}{\hat{s}_0, \hat{s}_1, \hat{s}_2, s} \).  However, this is exactly what \cref{lem:max-flip-flops} denies.  Contradiction!    
\end{proof}

Before we can complete the proof of our main theorem we must fulfill the commitment we made in \cref{cond:building-A:infinite-correctness} and verify that \( \Gfunc{i}{A \Tplus \REset(A){i}} \) is total.

\begin{lemma}\label{lem:infinite-correctness}
    There are infinitely many stages \( s \) such that \( A_s \subfun A_t \) for all \( t > s \) and thus \( A_s \subfun A \).  Thus, \cref{cond:building-A:infinite-correctness} hold and \( \Gfunc{i}{A \Tplus \REset(A){i}} \) is total. 

\end{lemma}

Note that, this result entails that \( A \) is \( \Delta^0_2 \).  

\begin{proof}
To verify the primary claim note that if \module{R}{j,e} acts at stage \( s \) it injures lower priority modules and when reinitialized those modules never enumerate an element into any column of \( A \) below \( l_{s+1} \).  Thus, if \( s_{k} \) is last stage that any module with priority \( \leq k \) acts we have \( A_t \supfun A_{s_k} \) for all \( t > s_k \) and \( A_{s_k} \subfun A \).    

\Cref{cond:building-A:infinite-correctness} follows immediately and the totality of \( \Gfunc{i}{A \Tplus \REset(A){i}} \)  was observed to follow from this fact in the discussion immediately following \cref{def:funcs}.

\end{proof}

We can now easily complete the proof of our main theorem.

\mainthm*

\begin{proof}
By \cref{lem:cond-Rje-satisfaction} to show that \( A \) is properly \REA[3] but can't be expanded to a \REA[4] set it is enough to verify that the requirements of the form \req{P}{i} is satisfied.  Clearly \( \Tfunc{Y_i} \) is total for all \( i \) and \( \Gfunc{i}{A \Tplus \REset(A){i}}  \) is total by \cref{lem:infinite-correctness}.  Hence, it's enough to show that these functionals are correct at every stage.   But by \cref{prop:maintain-computations} this holds as long as no module for \req{P}{i} is injured infinitely many times.  However, as only modules of the form \module{R}{j,e} are responsible for injuries this is immediate from \cref{prop:R-finite}.  

The fact that \( A \) can be taken to be \( \Delta^0_2 \) follows from \cref{lem:infinite-correctness}.
\end{proof}

\bibliographystyle{plain}

\end{document}

% \bibliography{Remote}